\documentclass[a4paper,reqno,11pt]{amsart}

\usepackage[utf8]{inputenc,xcolor}  
\usepackage{marginnote}
\usepackage{geometry}
\usepackage{float} 
\usepackage[in]{fullpage}

\newcommand{\ds}{\displaystyle}

\newcommand{\tensor}{\otimes}
\newcommand{\leftsub}[2]{{\vphantom{#2}}_{#1}{#2}} 
\newcommand{\xycirc}[2]{\leftsub{#1}{\circ}_{#2}}

\newcommand{\op}{\mathcal}

\newcommand{\cdc}{,\dots,}

\newcommand{\M}{\mathbb{M}}

\newcommand{\FT}{\mathsf{ft}}
\newcommand{\FTGK}{\mathsf{FT}}

\newcommand{\uex}{\text{\mbox{!`}}}

\usepackage{amsmath}%
\usepackage{amsthm}
\usepackage{amsfonts}%
\usepackage{amssymb}%
\usepackage{graphicx}
\usepackage{xy,amsthm,enumerate,xypic,array}  

\input{xy}
\xyoption{all}

\numberwithin{equation}{section}

\newtheorem{theorem}{Theorem}[section]
\theoremstyle{plain}

\newtheorem{assumption}[theorem]{Assumption}
\newtheorem*{nonumbertheoremA}{Theorem A}
\newtheorem*{nonumbertheoremB}{Theorem B}
\newtheorem*{nonumbertheoremC}{Theorem C}

\newtheorem{corollary}[theorem]{Corollary}
\newtheorem{lemma}[theorem]{Lemma}
\newtheorem{proposition}[theorem]{Proposition}

\theoremstyle{definition}
\newtheorem{definition}[theorem]{Definition}
\newtheorem{example}[theorem]{Example}

\newtheorem{remark}[theorem]{Remark}

\allowdisplaybreaks[2]

\addtocounter{MaxMatrixCols}{2}

\begin{document}

\title{Massey products for graph homology} 
\author{Benjamin C. Ward}
\email{benward@bgsu.edu}


\begin{abstract} This paper shows that the operad encoding modular operads is Koszul.  Using this result we construct higher composition operations on (hairy) graph homology which characterize its rational homotopy type.  \end{abstract}
\maketitle

\section{Introduction.}

Graph complexes are combinatorial objects which can be used to compute invariants of topological spaces. They were introduced by Kontsevich \cite{KontSymp},\cite{Kont} building on earlier combinatorial models for moduli spaces \cite{Penner} and incorporating influence from Feynman diagrams in quantum field theory.  Depending on the particular combinatorics of the graphs involved, graph complexes may be used to calculate cohomology of moduli spaces of Riemann surfaces \cite{KontSymp},\cite{GeK2}, moduli spaces of tropical curves \cite{CGP}, or embedding spaces of manifolds \cite{ALT},\cite{AT}.  Further variants of graph complexes may be used to study the Grothendieck-Teichmuller Lie algebra \cite{WTw} or automorphisms of free groups \cite{KontSymp},\cite{CVog}.

To encode the construction of graph complexes, Getzler and Kapranov introduced in \cite{GeK2} the notions of modular operads and the Feynman transform, which we denote by $\FTGK$.  With this notion, graph homology with labels in a cyclic operad $\op{O}$ may be defined as:
\begin{equation*}
\op{G_O}:= H_\ast(\FTGK(\iota_!\op{O}))
\end{equation*}
where $\iota_!\op{O}$ is the extension of $\op{O}$ to higher genus by $0$.  The functor $\FTGK$ is an involution up to homotopy and preserves quasi-isomorphisms.  This means that {\it if} $\FTGK(\iota_!\op{O})$ were equivalent to its own homology we could conclude $\FTGK(\op{G_O})\sim \iota_!\op{O}$.  This would be a powerful computational tool because it would imply, among other things, that the graph homology was generated in genus $0$.  Alas this formality property is rarely present, and the act of taking homology loses information about the homotopy type of the Feynman transform.

In this paper we give a way to systematically account for this loss of information by constructing an analog of Massey products for modular operads.  Taking a suitable generalization of the Feynman transform which incorporates these higher operations, which we denote $\FT$, we do indeed find 
\begin{equation*}
\FT(\op{G_O})\sim \iota_!\op{O}.
\end{equation*}

In constructing these higher operations, our model to follow comes from rational homotopy theory. 
The cochain complex of a simply connected topological space need not be equivalent to its cohomology in the category of commutative algebras, but there are higher cohomology operations which fit together to form a homotopy commutative algebra such that the cochains and cohomology are equivalent in this larger category (see \cite{Kad}, after \cite{Sullivan}).
This fundamental example motivated the successful homotopy transfer theory for algebras over operads \cite{LV,Berglund} in which the property of Koszulity plays a fundamental role in furnishing small cofibrant resolutions encoding homotopy invariant structure. In this paper we endeavor to resolve the composition in modular operads, and to this end our first needed result is:

\begin{nonumbertheoremA}
	The operad encoding modular operads is Koszul.
\end{nonumbertheoremA}

The notion of Koszulity used in Theorem A requires a modest generalization of the state-of-the-art.  In order to encode modular operads as algebras over a quadratic operad, we use a notion of colored operads where the colors form not just a set but a groupoid.  Generalizing Koszul duality and homotopy transfer theory to this context is a necessary but straight-forward exercise and is carried out in Section $\ref{GCOSec}$.  Beyond just modular operads, this section illustrates how the theory of Koszul duality for groupoid colored operads is an effective tool to establish homotopy transfer theory for generalizations of operads admitting bar-cobar duality, see Remark $\ref{cubicalrmk}$ and Corollary $\ref{cycopscor}$.

With this generalization of Koszul duality in hand, we first define a quadratic groupoid colored operad $\M$ whose algebras are modular operads, we then identify its quadratic dual $\M^!$ to be a suspension of the operad encoding  $\mathfrak{K}$-twisted modular operads, denoted $\M_{\mathfrak{K}}$, and we prove Theorem A by showing that the natural map 
\begin{equation*}
\Omega(\M^\ast)\stackrel{\sim}\longrightarrow \M_\mathfrak{K}
\end{equation*}
is a levelwise quasi-isomorphism (here $\Omega$ denotes the cobar construction).  Our proof of this fact is a novel and direct analysis which emphasizes the combinatorics of the fiber.  We define a weak modular operad to be an algebra over $\Omega(\M_\mathfrak{K}^\ast)$.


Returning to the example of graph homology, we apply our results in the following way.  First we use homotopy transfer theory to give $\op{G_O}$ the structure of a weak modular operad such that $\op{G_O}\sim \FTGK(\iota_!(\op{O}))$ (where the right hand side has its (strong) modular operad structure).  This endows graph homology with operations of the form:
\begin{equation*}
 \text{mp}_\gamma\colon \left(\ds\bigotimes_{i=1}^r \op{G_O}(v_i)\right)
\longrightarrow \op{G_O}(v_0)
\end{equation*}
for each (modular) graph $\gamma$ with vertices of type (color) $(v_1\cdc v_r;v_0)$.

We call these transferred operations ``Massey products''.  This terminology is derived from an analogy between weak modular operads and $A_\infty$-algebras in which the transferred operations, due to Kadeishvili \cite{Kad1}, generalize classical Massey products in algebraic topology, see \cite[Section 9.4.5]{LV} and the references there-in.  The use of operadic versions of Massey products to establish non-formality results finds a forerunner in \cite{Liv15}.  In our case, the fact that graph homology is defined via an extension by zero allows us to prove the following structural result:

\begin{nonumbertheoremB} Let $\op{O}$ be a Koszul cyclic operad with Koszul dual $\op{O}^!$.  Every graph homology class in $\op{G_O}$ is either a generator of $\op{O}^!$ or is in the image of some Massey product.
\end{nonumbertheoremB}

To prove this result, as well as to the organize this family of new operations, we generalize the Feynman transform to this setting.  More precisely we define the weak Feynman transform as pairs of functors
\begin{center}
	$	\FT\colon \left\{ \text{(weak) modular operads} \right\} \leftrightarrows \left\{ \text{(weak) }\mathfrak{K}\text{-modular operads} \right\}\colon \FT$
\end{center}
such that $\FT^2\sim id$ and that $\FT$ preserves $\infty$-quasi-isomorphisms.  In particular, since $\FT(\op{G_O})\sim \iota_!\op{O}$, every graph homology class in genus $g\geq 1$ which is not a boundary is also not a cycle.

 Since graph homology is not just a weak modular operad, but a (strong) modular operad as well, the differential in its weak Feynman transform is a sum of the classical Feynman transform differential with terms corresponding to Massey products of $2$ edges or more.  Filtering $\FT(\op{G_O})$ by internal degree has the effect of isolating the (classical) Feynman transform differential, thus:

\begin{nonumbertheoremC}
	There is a spectral sequence whose first page is the homology of $\FTGK(\op{G_O})$ and which converges to $\iota_!\op{O}$.  The higher differentials correspond to sums of linear duals of Massey products.
\end{nonumbertheoremC}

This is the analog of the Milnor-Moore spectral sequence from rational homotopy theory \cite{RHT}. The computational facility of this and a related spectral sequence considered in subsection $\ref{sssec}$ is that it allows us to, roughly speaking, pair classes in $H_\ast(\FTGK(\iota_!\op{O}))$ with classes in $H_\ast(\FTGK(\iota_\ast(\op{O}^!))$ for $\op{O}$ a Koszul cyclic operad, and $\iota_\ast$ denoting the modular envelope.  

For example, when $\op{O}$ is the Lie operad we may compare Lie graph homology, denoted $H_\ast(\Gamma_{g,n})$ after \cite{CHKV}, with a graph complex computing the top weight homology of the moduli space of punctured Riemann surfaces, denoted $H_\ast(\Delta_{g,n})$ after \cite{CGP}, via:

\begin{corollary}\label{sscor}  Fix $(g,n)$, a pair of natural numbers with $g\geq 1$ and $2g+n\geq 3$.
	\begin{enumerate}
		\item There is an upper half plane spectral sequence whose 0-page is 
	$\FTGK(H_\ast(\Gamma_{\bullet,\bullet}))(g,n)$, whose bottom row computes $H_\ast(\Delta_{g,n})$ and which converges to $0\cong \iota_!(Lie)(g,n)$.
		\item  There is an upper half plane spectral sequence whose 0-page coincides with $\FTGK(H_\ast(\Delta_{\bullet,\bullet}))(g,n)$ as graded vector spaces, whose bottom row computes $H_\ast(\Gamma_{g,n})$ and which converges to $k\cong \iota_\ast(Com)(g,n)$.
	\end{enumerate}	
\end{corollary}

In short, these spectral sequences tell us that every commutative graph homology class is witnessed by a graph labeled by Lie graph homology classes, and vice versa.
As an application we compute:

\begin{corollary}\label{g3cor} $H_{d}(\FTGK(\iota_\ast Com))(3,0) = \begin{cases}
	\mathbb{Q} & \text{ if } d=-6 \\
	0 & \text{ else }
	\end{cases}$
\end{corollary}

This calculation agrees with the computer calculations done in \cite[Appendix A]{CGP}.
The novel feature of this calculation is that it requires no analysis of any differentials; it relies only on the representation theory of Lie graph homology and the elementary properties of the spectral sequences constructed above.  With significant additional work it may be possible to generalize this computation; we refer to Section $\ref{futuredirections}$ for this among a list of additional applications and questions stemming from the results of this paper.

\tableofcontents

\subsection*{Acknowledgement}  I would like to thank Dan Petersen for sharing with me his observation that non-formality of the modular operad $H_\ast(\Gamma_{g,n})$ may be seen as a consequence of the non-surjectivity of the assembly map.  I would also like to thank Ralph Kaufmann for the insight that using groupoid colors/Feynman categories allows one to encode modular operads via a homogeneous-quadratic presentation.  This paper has also benefited from helpful conversations with Alexander Berglund, Martin Markl and Bruno Vallette.  Finally I would like to thank several anonymous referees whose detailed comments have significantly improved this paper.

\section{Koszul duality for groupoid colored operads.}\label{GCOSec}
Our conceptual starting point is the work of Van der Laan \cite{Vdl} who resolves the colored operad encoding non-symmetric operads using Koszul duality.  This colored operad is homogeneous quadratic and its colors form a set.  Encoding modular operads as homogeneous quadratic, it is necessary to allow operads whose colors form a groupoid.  

Groupoid colored operads were introduced by Petersen \cite{Pet} and simultaneously developed from the perspective of symmetric monoidal categories in \cite{KW}.  It has been shown \cite{BKW} that groupoid colored operads may also be viewed as set colored operads in which the automorphisms in the groupoid are viewed as unary operations in the operad.  However, viewing our operads of interest as groupoid colored will be desirable for two reasons.  First the passage from groupoid colored operads to set colored operads does not preserve quadraticity.  By viewing our operads as groupoid colored they may be presented as (homogeneous) quadratic, allowing us to apply a direct generalization of the classical theory of Koszul duality.  Second, in characteristic zero, any chain complex with a finite group action may be viewed as an {\it equivariant} deformation retract of its homology (Lemma $\ref{eqlem}$).  This means that homotopy transfer theory for algebras over colored operads can be implemented without resolving automorphisms (resolving automorphisms is of interest in characteristic $p$, see \cite{DV}). This makes groupoid colored operads a natural starting point for studying homotopy transfer theory in characteristic zero.

Thus, the results of this section are a straight-forward but essential generalization of the state-of-the-art.  After introducing the notion of groupoid colored operads, we develop Koszul duality (Subsection $\ref{KDsec}$), bar and cobar constructions (Subsection $\ref{abcsec}$) and homotopy transfer theory (Subsection $\ref{httsec}$).  We assume familiarity with Koszul duality for set-colored operads and emphasize the novel features and arguments of the groupoid colored case.  In particular, we closely follow the book of Loday-Vallette \cite{LV} which gives the theory in the uncolored case.  The debt this section owes to that work should not be understated.

\subsection{Graphs, trees and colored trees.}\label{graphs}

\begin{definition}\label{graphdef} An {\bf abstract graph} is a 4-tuple $\Gamma=(V,F,a,i)$ consisting of a non-empty finite set $V$, a finite set $F$, a function $a\colon F\to V$ and an involution $i\colon F\to F$.   A {\bf subgraph} is a pair of subsets of $V$ and $F$ which are closed under $a$ and $i$. 
\end{definition}

This definition will be accompanied by the following terminology:

\begin{itemize}
\item The elements of $V$ are called the {\bf vertices} of $\Gamma$.  
\item The elements of $F$ are called the {\bf flags} of $\Gamma$.  They may also be referred to as {\bf half-edges}.  
\item The orbits of $i$ of order 2 are called the {\bf edges} of $\Gamma$. 
\item  The fixed points of $i$ are called the {\bf legs} of $\Gamma$.

\item If $a(f)=w$ we say both that $f$ is {\bf adjacent} to $w$ and that $w$ is {\bf adjacent} to $f$.  The set $a^{-1}(w)$ of flags adjacent to $w$ will also be denoted fl($w$).

\item The integer $|a^{-1}(w)|$ is called the {\bf valence} of $w$.  It is denoted $||w||$.
\item A graph is called {\bf reduced} if $||w||\geq 2 \ \forall \ w\in V$.
\end{itemize}

It is a standard construction to associate a 1-dimensional CW complex to an abstract graph.  In the presence of legs, this construction adjoins an extra $0$-cell to the end of each leg.
 We say a graph is {\bf connected} (resp.\ {\bf simply connected}) if this CW complex is connected (resp.\ simply connected).

\begin{definition}  A simply connected graph will be called a \textbf{tree}. 
\end{definition}

This definition will be accompanied by the following terminology:

\begin{itemize}
	\item A {\bf corolla} is a tree with no edges.
	\item    A {\bf rooted tree} is a tree along with a distinguished leg, called the {\bf root}. 
\item The non-distinguished legs of a rooted tree are called {\bf leaves}.  

\item Given a pointed set $X$ with basepoint $\ast$, an { \bf $X$-labeled rooted tree} is a rooted tree $T$ along with a bijection between $X$ and the set of legs of $T$ which takes $\ast$ to the root. 

\item An $n${\bf-tree} is a $\{0,1\cdc n\}$-labeled rooted tree, where $0$ is taken to be the base point of the set $\{0,1\cdc n\}$.

\end{itemize}
In particular an $n$-tree has $n$ leaves labeled by $\{1\cdc n\}$, while we regard the base point $0$ as labeling the root.  Every rooted tree $T$ may naturally be considered as a $\text{Leg}(T)$-tree.

In the CW complex corresponding to a rooted tree there is a unique shortest path from a vertex $w$ to the root.  This path will intersect a unique flag adjacent to $w$, and we call this distinguished adjacent flag the {\bf output} of $w$.  If $w$ is a vertex in a rooted tree, we consider the set fl$(w)$ to be a pointed set with the output flag as the basepoint.  Flags which are not outputs are called {\bf inputs}.  Every edge in a rooted tree consists of one input flag and one output flag.  The number of inputs adjacent to a vertex $w$ in a rooted tree will be called the {\bf arity} of the vertex and will be denoted $|w|$.

In this paper we will consider colored trees.  We define a {\bf coloring} of a graph $(V,F,a,i)$ by a set $C$ to be a function $\text{clr}\colon F\to C$ such that $\text{clr}=\text{clr}\circ i$.  Informally, a coloring of a graph by $C$ assigns an element of $C$ to each edge and leg of the graph.  We are now prepared to define the class of trees we will consider in this paper.

\begin{definition}\label{vcoloredtree}  Let $\mathbb{V}$ be a small category and $X$ be a set.  A $\mathbb{V}$-{\bf colored $X$-tree} is a reduced, $X$-labeled rooted tree along with a coloring by the set $ob(\mathbb{V})$.   A $\mathbb{V}$-{\bf colored tree} is a reduced $n$-tree along with a coloring by the set $ob(\mathbb{V})$. 
\end{definition}

We will more often have use for the latter definition, in which the leaves of the tree $T$ are labeled by $\{1\cdc n\}$.  By abuse of notation we will often refer to a $\mathbb{V}$-colored tree by simply $T$, keeping the leaf and edge labels implicit.  While the morphisms in $\mathbb{V}$ do not play a role in the definition of a $\mathbb{V}$-colored tree, they will play a vital role in the definition of the monad of $\mathbb{V}$-colored trees below.

The {\bf type} of a $\mathbb{V}$-colored $X$-tree is the function $\nu \colon X \to ob(\mathbb{V})$ which sends an $x\in X$ to the object which colors the flag labeled by $x$.  In the special case of a $\mathbb{V}$-colored tree, this function is equivalent to the ordered list of objects $(\text{clr}(l_1)\cdc \text{clr}(l_n) ; \text{clr}(r))$ where $l_j\in F$ denotes the leaf labeled by $j$ and $r\in F$ denotes the root of $T$.  In this case we also use the notation $\vec{v}=(v_1\cdc v_n;v_0)$ to denote type.  See Figure $\ref{basictree}$ for an example. Similarly we define the {\bf type of a vertex} $w$ in a $\mathbb{V}$-colored tree to be the function fl$(w)\to ob(\mathbb{V})$ sending a flag to the object of $\mathbb{V}$ which colors it.

	\begin{figure}
	\includegraphics[scale=1.0]{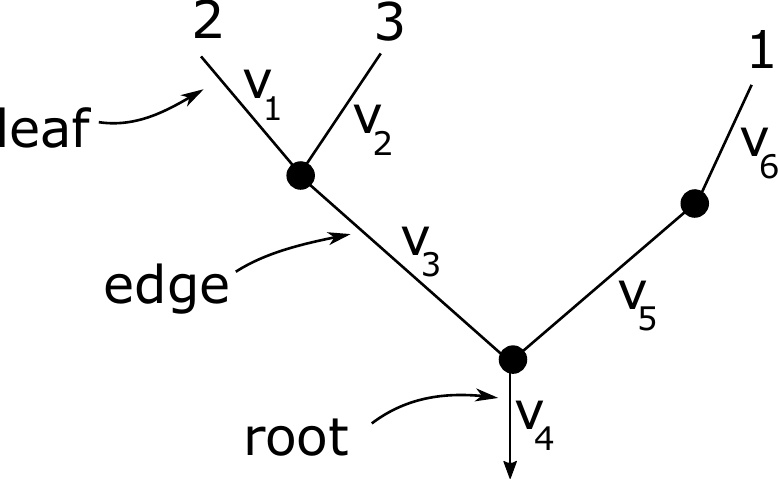} 
	\caption{ A $\mathbb{V}$-colored tree (pictured) is a rooted tree with leaves labeled by $\{1\cdc n \}$ and with edges and leaves labeled by $v_i\in ob(\mathbb{V})$.   This $\mathbb{V}$-colored tree has three vertices, two edges, three leaves, eight flags and is of type $(v_6,v_1,v_2;v_4)$.		} \label{basictree}
\end{figure}

\subsubsection{Substitution of colored trees.} \label{treesub}

We define a substitution of $\mathbb{V}$-colored trees starting from the following data.
Let $T$ be a $\mathbb{V}$-colored $X$-tree of type $\lambda$, let $w$ be a vertex in $T$ of type $\nu$, and let $U$ be a $\mathbb{V}$-colored fl$(w)$-tree of type $\nu$.

 We then form a new $\mathbb{V}$-colored $X$-tree of type $\lambda$, denoted $T\circ_w U$ as follows.  The set of vertices is $V(T\circ_w U)=(V(T)\setminus w)\sqcup V(U)$.  The set of flags is
$F(T\circ_w U)= (F(T)\setminus \text{fl}(w))\sqcup F(U)$. 
The adjacency map $a\colon F\to V$ is the union $a_T$ and $a_{U}$.
Finally the involution is the same except on the former legs of $U$.  To define the involution on these legs, we let $\psi\colon \text{fl}(w)\to \text{leg}(U)$ denote the bijection given by the fl$(w)$ labeling of $U$ and define $i(\ell):= i_T(\psi^{-1}(\ell))$.  The resulting abstract graph $T\circ_w U$ is a rooted tree, its leaves inherit a labeling by the set $X$, and the inherited $\mathbb{V}$-coloring still respects the involution.  Thus $T\circ_w U$ is a $\mathbb{V}$-colored $X$-tree of type $\lambda$.

Let us record several facts about substitution for future use.  The proofs of these lemmas are all immediate.

\begin{lemma}\label{sameedges}  Inclusion of flags induces a bijection $ Edges(T)\sqcup Edges(U) \cong Edges(T\circ_w U)$.
\end{lemma}

\begin{lemma}\label{commutativesub} If $S\circ_y(T\circ_w U)$ is defined then $(S\circ_y T)\circ_w U$ is also defined and $S\circ_y(T\circ_w U)=(S\circ_y T)\circ_w U$.
\end{lemma}

\begin{lemma}\label{totalsub}  Let $T$ be a $\mathbb{V}$-colored $X$-tree of type $\lambda$ with vertices $w_1\cdc w_r$.  Let $w_i$ be of type $\nu_i$, and fix a $\mathbb{V}$-colored fl$(w_i)$-tree of type $\nu_i$, called $U_i$ for each $i$.  Then the iterated substitution
	\begin{equation*}
		 (...((T\circ_{w_1}U_{1})\circ_{w_2} U_{2})... )\circ_{w_r} U_r
	\end{equation*}
is well defined and independent of the order of the vertices.
\end{lemma}

We call this a {\bf total substitution} and denote it by $T\circ \{U_i\}$

\begin{lemma}\label{samevertices} In a total substitution, inclusion of vertices induces a bijective correspondence $V(T\circ \{U_i\})\cong \sqcup_i V(U_i)$.
\end{lemma}

\subsubsection{Empty $\mathbb{V}$-colored trees}\label{emptytree}

We will occasionally refer to the objects of $\mathbb{V}$ as {\bf empty $\mathbb{V}$-colored trees}.  This terminology is meant to conflate an object of $\mathbb{V}$ with a directed line segment labeled by it, which could be thought of as a ``$\mathbb{V}$-colored tree with no vertices'', although this is technically not a tree by our definition.  To reinforce this graphical intuition we will typically denote the empty colored tree corresponding to an object $v$ by the symbol $|_v$.  We consider $|_v$ to be of type $(v;v)$.

Let $T$ be a $\mathbb{V}$-colored $X$-labeled tree with a vertex $w$ of valence two of type $(v;v)$. 
We denote the {\bf substitution of the empty colored tree} $v$ at $w$ by $T\circ_w (|_v)$.  
Informally we picture pasting the empty $\mathbb{V}$-colored tree over the vertex, thus removing it. 
Formally,  $T\circ_w (|_v)$ is defined as follows.  If $w$ is the lone vertex of $T$ we define $T\circ_w (|_v) := |_v$.  Else we define $T\circ_w (|_v)$ to be the $\mathbb{V}$-colored $X$-labeled tree with vertices $V(T)\setminus w$, flags $F(T) \setminus fl(w)$, adjacencies by restriction, and the unique involution such that both $i$ agrees with $i_T$ except on the set $F(T\circ_w(|_v)) \cap i_T(\text{fl}(w))$ and that the resulting graph is connected.

\subsubsection{Isomorphisms of graphs}\label{isomorphismsofgraphs}

\begin{definition} An isomorphism $\phi\colon (V_1,F_1,a_1,i_1) \to (V_2,F_2,a_2,i_2)$ between abstract graphs is a pair of bijections $\phi_V\colon V_1\stackrel{\cong}\to V_2$ and  $\phi_F\colon F_1\stackrel{\cong}\to F_2$ such that $\phi_F i_1 = i_2 \phi_F $ and $\phi_V a_1 = a_2 \phi_F$. 
\end{definition}

Clearly an isomorphism between two graphs specifies isomorphisms between their sets of vertices, flags, edges and legs.  Hence we may define:

\begin{definition} An isomorphism of $\mathbb{V}$-colored $X$-labeled trees is an isomorphism of graphs which respects the coloring and labeling.
\end{definition}


In the remainder of this paper we are primarily interested in such labeled trees and graphs up to color and label preserving isomorphisms.

\subsection{Monad of $\mathbb{V}$-colored trees.}
Fix a co-complete and closed symmetric monoidal category $(\op{C},\tensor, 1_\op{C})$. In practice $\op{C}$ will be either sets or chain complexes.  Also fix a small category and groupoid $\mathbb{V}$.  The purpose of this subsection will be to introduce a certain monad of $\mathbb{V}$-colored trees.  We first introduce the category on which this monad will act.  Throughout we denote the symmetric groups by $S_n$.

\subsubsection{$\mathbb{V}$-colored sequences.}
For a natural number $r\geq 1$ we let $\mathbb{V}$-corollas$_r$ be the following category.  The objects are the objects of $\mathbb{V}^{\times r}\times \mathbb{V}^{op}$.  We call an object of this category a {\bf $\mathbb{V}$-color scheme} of length $r$.  We will occasionally write $\mathbb{V}$-color schemes using vector notation: $\vec{v}= (v_1\cdc v_r; v_0)$, and write $|\vec{v}|=r$.  The morphisms of $\mathbb{V}$-corollas$_r$ having source $(v_1\cdc v_r; v_0)$ coincide with the set 
\begin{equation*}
S_r\times Aut(v_1)\times...\times Aut(v_r)\times Aut(v_0)^{op}.
\end{equation*}

  The target of such a morphism $(\sigma, -)$ is $(v_{\sigma^{-1}(1)}\cdc v_{\sigma^{-1}(r)}; v_0)$.  Composition of morphisms is given by 
\begin{equation}\label{composition}
(\tau,\times_{i\geq 1}\phi_i, \phi_0^{op})\circ (\sigma, \times_{i\geq 1}\gamma_i, \gamma_0^{op}) = (\tau\sigma, \times_{i\geq 1} \phi_{\sigma(i)}\gamma_i, \gamma_0^{op}\phi_0^{op})
\end{equation}
Define $\mathbb{V}$-corollas to be the disjoint union of the categories $\mathbb{V}$-corollas$_r$ over $r\geq 1$. A {\bf  $\mathbb{V}$-colored sequence} is defined to be a functor from $\mathbb{V}$-corollas to $\op{C}$.  Unpacking this definition, a $\mathbb{V}$-colored sequence specifies a collection of objects (in $\op{C}$) $A(v_1\cdc v_r;v_0)$ for $v_i\in ob(\mathbb{V})$ with an action of the group $\times_i Aut(v_i)$ on the left and $Aut(v_0)$ on the right, along with compatible isomorphisms corresponding to permuting the non-zero indices.  

For $\vec{v}= (v_1\cdc v_r; v_0)$ we use the notation:
\begin{equation}\label{autvecv}
Aut(\vec{v}):= Aut(v_1)\times...\times Aut(v_r)\times Aut(v_0)^{op}
\end{equation}

\subsubsection{The object $A(T)$.}\label{ATSec}
In this construction we make use of unordered monoidal products, see \cite{MSS} Definition II.1.58.

Fix vertex $w$ in a $\mathbb{V}$-colored tree of arity $|w|=r$ and fix a $\mathbb{V}$-colored sequence $A$. If we choose an auxiliary ordering of the input flags of $w$, then we may say $w$ is of type $(v_1\cdc v_r;v_0)$.  We then define
\begin{equation}\label{A(w)}
A(w)= \left[\ds\coprod_{\sigma\in S_r}  A(v_{\sigma(1)}\cdc v_{\sigma(r)}; v_0)   \right]_{S_r}
\end{equation}
 This object is independent of the order of the $v_i$ that we chose and it inherits a natural action of the group $(\times_i Aut(v_i))\times Aut(v_0)^{op}$, where here $\times_i$ denotes the unordered product.

Now let $T$ be a $\mathbb{V}$-colored tree with vertices Vert$(T)$ and form the unordered tensor product $\bigotimes_{w \in \text{Vert}(T)} A(w)$.  Given an edge $E$ of $T$ with input flag $x$, output flag $y$, and color $v_E:=\text{clr}(x)=\text{clr}(y)$, the group $Aut(v_E)$ acts on $\bigotimes_{w \in \text{Vert}(T)} A(w)$ by acting simultaneously on the $a(x)$ factor on the left and the $a(y)$ factor on the right.  Formula $\ref{composition}$ ensures this action commutes across different edges, and hence we have an action of the group $\times_{E\in\text{Edges}(T)}Aut(v_E)$ on $\bigotimes_{w \in \text{Vert}(T)} A(w)$.  Let us write $[\bigotimes_{w \in \text{Vert}(T)}A(w)]_{\text{Edges}(T)}$ for the coinvariants of this action and define
\begin{equation}\label{AT}
A(T) := \left[ \ds\bigotimes_{w\in \text{Vert}(T)} A(w)\right]_{\text{Edges}(T)}.
\end{equation}

This definition may be understood informally by regarding $A(T)$ as the space of $\mathbb{V}$-colored trees $T$ whose vertices are labeled by $A$, modulo the relation that an automorphism in $Aut(v)$ can be transported across an edge labeled by $v$.

\subsubsection{The monad $\mathbb{T_V}$.}
We consider $\mathbb{V}$-colored sequences as a category whose morphisms are given by natural transformations.  Let us define a functor
\begin{equation*}
\mathbb{T_V}\colon \mathbb{V}\text{-colored sequences}\to \mathbb{V}\text{-colored sequences}
\end{equation*}
on objects via
\begin{equation}\label{free1}
\mathbb{T_V}(A)(\vec{v}) := \ds\coprod_{\substack{\text{isomorphism classes of} \\ \mathbb{V} \text{-colored trees $T$ } \\ \text{of type }  \vec{v}}} A(T) 
\end{equation}

Here we have abused notation by conflating a tree $T$ with its isomorphism class, but since $A(T)$  is independent of the choice of representative of the isomorphism class of $T$, no confusion should result.  For brevity we no longer write ``isomorphism classes of'' below, but we continue to consider trees only up to label and color preserving isomorphisms.

Observe that $\mathbb{T_V}(A)$ is naturally a $\mathbb{V}$-colored sequence since, if $T$ is a $\mathbb{V}$-colored tree of type $\vec{v}$, then for any leg labeled by some $v_i$ we may define an action of $Aut(v_i)$ on $A(T)$ by acting on the factor corresponding to the unique vertex adjacent to said leg.

Since a morphism of $\mathbb{V}$-colored sequences is a natural transformation, it is given by a collection of equivariant maps $A(\vec{v})\to B(\vec{v})$ and so induces a map $A(T)\to B(T)$.  Taking the coproduct of such maps specifies the image of $\mathbb{T_V}$ on morphisms.

\begin{theorem} \label{monadthm} Substitution of trees endows the functor $\mathbb{T_V}$ with the structure of a monad.
\end{theorem}

\begin{proof}  The proof closely follows the uncolored case, see e.g.\ \cite[p.\ 89]{MSS} or \cite[p.\ 160]{LV}.
		
	Let us first observe that if $A$ is a $\mathbb{V}$-colored sequence and $w$ is a vertex of a $\mathbb{V}$-colored tree of type $\eta\colon \text{fl}(w)\to ob(\mathbb{V})$, then combining formulas $\ref{A(w)}$ and $\ref{free1}$ yields:
	
\begin{equation}
(\mathbb{T_V}(A))(w) \cong \ds\coprod_{\substack{\mathbb{V} \text{-colored fl$(w)$-trees $U$ } \\ \text{of type }  \eta}} A(U) 
\end{equation}
	
	Using this we compute:
	
	\begin{eqnarray*}
	\mathbb{T_V} \circ \mathbb{T_V}(A) (\vec{v})
	& = & \ds\coprod_{\substack{\mathbb{V} \text{-colored trees $T$ } \\ \text{of type }  \vec{v}}} \left[\ds\bigotimes_{w \in \text{Vert}(T)} \mathbb{T_V}(A)(w)\right]_{\text{Edges}(T)} \\
		& \cong & \ds\coprod_{\substack{\mathbb{V} \text{-colored trees $T$ } \\ \text{of type }  \vec{v}}} \left[\ds\bigotimes_{w \in \text{Vert}(T)} 
		\left[  \ds\coprod_{\substack{\mathbb{V} \text{-colored fl$(w)$-trees $U_w$ } \\ \text{of the type of  }   w}} A(U_w)  \right]
		\right]_{\text{Edges}(T)} \\
			& = & \ds\coprod_{\substack{\mathbb{V} \text{-colored trees $T$ } \\ \text{of type }  \vec{v}}} \left[\ds\bigotimes_{w \in \text{Vert}(T)} 
	\left[  \ds\coprod_{\substack{\mathbb{V} \text{-colored fl$(w)$-trees } \\ \text{ $U_w$ of the type of  }   w}}
	\left[ \ds\bigotimes_{x \in \text{Vert}(U_w)} A(x) \right]_{\text{Edges}(U_w)}
	 \right]
	\right]_{\text{Edges}(T)} \\
		& \cong & \ds\coprod_{\substack{\mathbb{V} \text{-colored trees $T$ } \\ \text{of type }  \vec{v} \text{ along with} \\ \text{ a fl(w) tree } U_w \\  \forall \ w \in \text{Vert}(T)}} \left[\ds\bigotimes_{ x \in \coprod_w\text{Vert}(U_w)} A(x)\right]_{\text{Edges}(T \circ \{U_w\})} \\
	\end{eqnarray*}

The validity of the last step rests on our ability to iteratively commute the monoidal product with colimits in a fixed variable, which is allowed due to our assumption that $\op{C}$ is closed monoidal.  We then employ the bijective correspondence of the edges of the total substitution $T \circ \{U_w\}$ and the union of the edges of its pieces (Lemmas $\ref{sameedges}$ and $\ref{totalsub}$).  

Using the bijective correspondence between the vertices of $T\circ\{U_w\}$  and the disjoint union of the vertices of the $U_w$ (Lemma $\ref{samevertices}$), each term in this coproduct specifies a summand of $\mathbb{T_V}(A)(\vec{v})$ and taking the coproduct of these inclusions gives us the morphism $\mathbb{T_V} \circ \mathbb{T_V}(A) (\vec{v}) \to  \mathbb{T_V}(A) (\vec{v})$.  The associativity of this composition follows from the associativity of tree substitution (Lemma $\ref{commutativesub}$).  The monadic unit $id \Rightarrow \mathbb{T_V}$ is given by corolla inclusions.
\end{proof}

\subsubsection{The monad $\overline{\mathbb{T}}_\mathbb{V}$.}
There is a variant of the above construction which includes the empty $\mathbb{V}$-colored trees (see Subsection $\ref{emptytree}$).  Given a $\mathbb{V}$-colored sequence $A$ we define $A(|_v):=  \coprod_{Aut(v)} 1_\op{C}$  with its induced left and right $Aut(v)$-action.  We then define a functor
\begin{equation*}
\overline{\mathbb{T}}_\mathbb{V}\colon \mathbb{V}\text{-colored sequences}\to \mathbb{V}\text{-colored sequences}
\end{equation*}
via $\overline{\mathbb{T}}_\mathbb{V}(A)(\vec{v}) := \coprod A(T)$, where now the coproduct is taken over isomorphism classes of {\it possibly empty} $\mathbb{V}$-colored trees of type $\vec{v}$.

\begin{theorem}  Substitution of possibly empty trees endows the functor $\overline{\mathbb{T}}_\mathbb{V}$ with the structure of a monad.
\end{theorem}
\begin{proof} 
	This follows similarly to the  proof of Theorem $\ref{monadthm}$.  The most noteworthy difference being that one or more of the $U_w$ could be the empty tree.  In this case we note that if $U_w = |_v$, then there is a surjective map  
	\begin{equation}\label{s}
	 Edges(T)\sqcup Edges(|_v) = Edges(T) \twoheadrightarrow Edges(T\circ_w |_v)
	 \end{equation}
	  which identifies two edges if they were formerly adjacent to the vertex $w$ (the generalization of Lemma $\ref{sameedges}$ to the case of possibly empty trees).  Using the fact that the $Aut(v)$ action on each $A(|_v)$ is free, we may cancel a tensor factor of $A(|_v)$ for each identification made in the surjection $\ref{s}$.	Thus the isomorphism \begin{equation*}
\left[\ds\bigotimes_{w \in \text{Vert}(T)}  A(U_w) 
		\right]_{\text{Edges}(T)} \cong  A(T \circ \{U_w\})
	\end{equation*}
	is still valid in the case that one or more of the $U_w$ is empty.  With this observation, the proof follows as in the proof of Theorem $\ref{monadthm}$.	
\end{proof}

\subsection{Groupoid colored operads}

We are now prepared to define {\bf groupoid colored operads}.  The generic terminology groupoid colored operad will mean a $\mathbb{V}$-colored operad for some groupoid $\mathbb{V}$, and is  defined as follows:

\begin{definition}  A non-unital $\mathbb{V}$-colored operad is an algebra over the monad $\mathbb{T_V}$.  A unital $\mathbb{V}$-colored operad is an algebra over the monad $\overline{\mathbb{T}}_\mathbb{V}$.  Morphisms of $\mathbb{V}$-colored operads are morphisms of algebras over the relevant monad.
\end{definition}

In this paper $\mathbb{V}$-colored operads will often be referred to simply as operads if $\mathbb{V}$ is given, and we shall work with both the unital and non-unital variants as the situation dictates.

By definition, then, a $\mathbb{V}$-colored operad $\op{P}$ comes with a morphism
\begin{equation}\label{operadstructuremap}
 \eta_{T}\colon\op{P}(T)\to \op{P}(\vec{v})
\end{equation}

 for every $\mathbb{V}$-colored tree of type $\vec{v}$, which we call contracting the tree.  These contractions are equivariant with respect to permuting the leaf labels of $T$ and acting by automorphisms of the objects coloring the leaves.  The property of being an algebra over the monad ensures that contracting a tree can be done one edge at a time, all at once, or via any combination of subtrees without changing the result.  In the unital case we may take $T$ to be the empty $\mathbb{V}$-colored trees, which gives us maps $1_\op{C} \to \op{P}(v,v)$ which are compatible with the contractions.

\begin{example}  Every uncolored operad is a groupoid colored operad for the trivial groupoid.
	Every set colored operad is a groupoid colored operad for the groupoid whose objects are the set of colors and which has only identity morphisms.
\end{example}

\begin{example}{(Free operads).} By general considerations in the theory of Monads, the forgetful functor which takes a $\mathbb{V}$-colored operad to its underlying $\mathbb{V}$-colored sequence has a left adjoint given simply by evaluation.
Here we may consider both the unital and non-unital variants of this free operad construction.  If we denote them by denoted $\overline{F}$ and $F$ respectively,  they are related  by:
\begin{equation}\label{freeunital}
\overline{F}(A)(\vec{v})=\begin{cases}
1_\op{C}\coprod F(A)(\vec{v}) & \text{ if } \vec{v}=(v;v) \text{ for some } v\in ob(\mathbb{V})  \\
F(A)(\vec{v}) & \text{else}
\end{cases}
\end{equation}
\end{example}

\begin{example} {(Endomorphism operads)}.  A functor $X\colon \mathbb{V}\to\op{C}$ will be called a {\bf $\mathbb{V}$-module}. To such an $X$ we associate the endomorphism $\mathbb{V}$-colored sequence:
	\begin{equation*}
	End_X(v_1\cdc v_n; v_0) := Hom_\op{C}(X(v_1)\tensor\dots\tensor X(v_n),X(v_0)),
	\end{equation*}
	with the inherited $Aut(v_i)$ actions and $S_n$ action by permuting the inputs.  This $\mathbb{V}$-colored sequence carries the natural structure of a unital $\mathbb{V}$-colored operad via composition of functions. 	Note that associativity of composition of functions ensures both associativity of the operadic structure maps $\eta_T$ (as usual) but also compatibility with the action of the groups $Edges(T)$.
\end{example}

\begin{example} {(Feynman categories).}  A Feynman category \cite{KW} with vertices $\mathbb{V}$ gives rise to a $\mathbb{V}$-colored operad in sets.  In this way we can use the examples given in \cite{KW} to build groupoid colored operads which encode operads, cyclic operads, modular operads etc.  We do not assume familiarity with {\it loc.\ cit.,} and give a detailed construction of the groupoid colored operad which encodes modular operads below (Section $\ref{secmodops}$).
	\end{example}

\begin{example}\label{I} (Initial operad).  Since $\op{C}$ is closed monoidal, there is a unique non-unital operad formed by taking the initial object in each color scheme. Likewise, there is a unique unital operad structure on the following $\mathbb{V}$-colored sequence:	\begin{equation*}
	\op{I}(\vec{v})=\begin{cases}
	\coprod_{Aut(v)}1_\op{C} & \text{ if } \vec{v}=(v;v) \text{ for some } v\in ob(\mathbb{V})  \\
	\text{initial object in } \op{C} & \text{else}
	\end{cases}
	\end{equation*}
Any unital operad $\op{P}$ admits a unique morphism $\op{I} \to \op{P}$.

\end{example}

\begin{example} (Trivial operad).  There is a non-unital operad formed by taking the terminal object (when it exists) in every color scheme.  Similarly, there is a unique unital operad structure, denoted $\op{T}$, on the $\mathbb{V}$-colored sequence:
	\begin{equation*}
	\op{T}(\vec{v})=\begin{cases}
	1_\op{C} & \text{ if } \vec{v}=(v;v) \text{ for some } v\in ob(\mathbb{V})  \\
	\text{terminal object in } \op{C} & \text{else}
	\end{cases}
	\end{equation*}
\end{example}

\begin{definition}\label{augdef}  An augmentation of a unital groupoid colored operad is a morphism $\op{P}\to\op{T}$.  An augmented operad is an operad along with an augmentation.  If $\op{C}$ is an Abelian category, the augmentation ideal is the kernels of these maps, and is denoted $\overline{\op{P}}$.
\end{definition}

\subsubsection{Monoidal Definition}  The original definition of a groupoid colored operad was given in \cite[Definition 3.8]{Pet} as a monoid for a monoidal product on the category of what we have called $\mathbb{V}$-colored sequences.  To be precise, \cite{Pet} does not make the restriction that the category $\mathbb{V}$ is a groupoid (see however his Remark 3.13), nor does he make the restriction that trees have at least one input as we have done (see Remark $\ref{arityzerormk}$ below).  With these restrictions we prove the equivalence of the monoidal and monadic definitions of groupoid colored operads.

\begin{definition}\label{levels} 
	A $\mathbb{V}$-colored rooted tree is {\bf level} if the unique shortest path from each leaf to the root passes through the same number of vertices.  We say a tree has $n$ levels if it is a level tree for which this number is $n$.  We define the empty $\mathbb{V}$-colored trees to be level with $n=0$.
\end{definition}

In any rooted tree the set of vertices has a partial order given by saying $w_i\leq w_j$ if $w_i$ lies on the unique shortest path between $w_j$ (inclusive) and the root.  We define the height of a vertex $ht(w)$ to be the number of vertices which are less than or equal to it in this partial order.

\begin{definition} \label{monoidaldef}  Let $T$ be a tree with $n$ levels and $A_1\cdc A_n$ be $\mathbb{V}$-colored sequences.  Define:
\begin{equation*}
(A_1\circ ... \circ A_n) (T) := \left[ \ds\bigotimes_{w\in \text{Vert}(T)} A_{ht(w)}(w)\right]_{\text{Edges}(T)}
\end{equation*}	
\end{definition}

The notation is as in Section $\ref{ATSec}$.  We then define

\begin{equation}
(A_1\circ ... \circ A_n)(\vec{v}) := \ds\coprod_{\substack{  \text{$T$ of type }  \vec{v}  \\ \text{with $n$ levels} }} (A_1\circ ... \circ A_n)(T) 
\end{equation}
where the coproduct is taken over isomorphism classes of possibly empty $\mathbb{V}$-colored trees. Taken over all $\vec{v}$, this gives $A_1\circ ... \circ A_n$ the structure of a $\mathbb{V}$-colored sequence. 

\begin{lemma} \cite[Proposition 3.6]{Pet} 
The operation $\circ$ is a monoidal product on the category of $\mathbb{V}$-colored sequences with monoidal unit $\op{I}$ (as in Example $\ref{I}$). 
\end{lemma}

\begin{definition}\cite[Definition 3.8]{Pet}  A $\mathbb{V}$-colored operad is a monoid in the monoidal category of $\mathbb{V}$-colored sequences.
\end{definition}

\begin{theorem}  The monoidal and (unitial) monadic definitions of groupoid colored operads are equivalent.
\end{theorem}

\begin{proof}
This result is the groupoid colored analog of \cite[Theorem 1.68]{MSS}, and we merely sketch the equivalence.  Given a $\overline{\mathbb{T}}_\mathbb{V}$-algebra $A$ we have maps $\eta_{T} \colon A(T)\to A(\vec{v})$ (Formula $\ref{operadstructuremap})$ for each tree, and so in particular for each level tree.  The coproduct of such maps gives us $A^{\circ n}(\vec{v})\to A(\vec{v})$ as desired.  In addition, the operations corresponding to the empty tree allow us to define the unit for the monoid structure $\op{I}\to A$.  

On the other hand given a monoid $A$ for $\circ$ we have an operation $A(T)\to A(\vec{v})$ for every level tree $T$.  Given a tree $T$ which is not level, there is a unique level tree $\hat{T}$ with a specified isomorphism $A(T)\stackrel{\cong}\to A(\hat{T})$, constructed by grafting vertices labeled by the monoidal unit $\op{I}(v,v)$ to those leaves of $T$
whose adjacent vertices are of sub-maximal height
 until the result is level.  This gives us the structure map $\eta_{A,T}\colon A(T)\stackrel{\cong}\to A(\hat{T})\to A(\vec{v})$.  
The verification that associativity of the monoidal product corresponds to the associativity condition for algebras over a monad is straightforward.	
\end{proof}	

\begin{remark}  \label{arityzerormk}
	When considering Definition $\ref{monoidaldef}$, it should be recalled (Definition $\ref{vcoloredtree}$) that we consider only reduced trees, meaning all vertices have arity $\geq 1$.  However, it is possible to drop this assumption while allowing our $\mathbb{V}$-colored sequences to have empty inputs to produce a monoidal product for such non-reduced $\mathbb{V}$-colored sequences (compare \cite[Definition 3.8]{Pet}).  
\end{remark}

\subsubsection{Algebras over operads.}
Let $X$ be a $\mathbb{V}$-module and $\op{P}$ be a unital $\mathbb{V}$-colored operad.  We define a $\op{P}$-{\bf algebra structure on} $X$ to be a $\mathbb{V}$-colored operad map $\op{P}\to End_X$.  Such a $\op{P}$-algebra structure specifies an $S_{|\vec{v}|}$-invariant family of $Aut(v_0)$-equivariant maps
\begin{equation}\label{algmaps}
\op{P}(\vec{v})\tensor_{in(\vec{v})} X(in(\vec{v}))\to X(v_0)
\end{equation}
where $X(in(\vec{v})):=\tensor_{i\geq 1} X(v_i)$ and where $\tensor_{in(\vec{v})}$ means coinvariants with respect to the $\times_{i\geq1}Aut(v_i)$ action.

We may also give a monoidal characterization of algebras over groupoid colored operads.  First, to a $\mathbb{V}$-module $X$ we may associate a non-reduced $\mathbb{V}$-colored sequence by defining $X(\emptyset; v):=X(v)$ and $X(\vec{v})$ to be the initial object otherwise.  If $A$ is any $\mathbb{V}$-colored sequence then the non-reduced composition $A\circ X$ is a non-reduced $\mathbb{V}$-colored sequence which is non-trivial only for color schemes $\vec{v}$ having empty inputs.  We may thus consider $A\circ X$ to be a $\mathbb{V}$-module.  Explicitly:

\begin{equation*}
A\circ X(v_0) = \left[\ds\coprod_{\vec{v}=(-\cdc-;v_0)} A(\vec{v})\tensor_{in(\vec{v})} X(in(\vec{v})) \right]_{S_{|\vec{v}|}}
\end{equation*}

A $\op{P}$-algebra structure on $X$ determines a morphism of $\mathbb{V}$-modules $\op{P}\circ X \to X$ via the maps in Formula $\ref{algmaps}$, and a $\op{P}$-algebra structure on $X$ may be characterized as such a morphism which is associative and unital with respect to the monoid structure on $\op{P}$, compare \cite[p.\ 133]{LV} in the uncolored case.  With this, \cite[Propoition 5.2.1]{LV} carries over verbatim to prove:

\begin{lemma} 
	The forgetful functor from $\op{P}$-algebras to $\mathbb{V}$-modules has a left adjoint $\mathsf{F}_\op{P}$ defined by  $\mathsf{F}_\op{P}(X)=\op{P}\circ X$ with algebra structure maps $\op{P}\circ\mathsf{F}_\op{P}\to \op{P}\circ X $ induced by the monoid structure $\op{P}\circ \op{P} \to \op{P}$.	
\end{lemma}

\begin{remark}\label{non-unital algebras}  Algebras over non-unital operads are defined as follows.  The forgetful functor from unital to non-untial $\mathbb{V}$-colored operads has a left adjoint given by $-\oplus \op{I}$ (as in Example $\ref{I}$).  We define $\underline{\op{P}}:= \op{P}\oplus \op{I}$ for a non-unital operad $\op{P}$.
We may then equivalently define a $\op{P}$-algebra structure on a $\mathbb{V}$-module to be either a unital $\mathbb{V}$-colored operad map $\underline{\op{P}}\to End_X$ or a non-unital $\mathbb{V}$-colored operad map $\op{P}\to End_X$.  In particular, for $\op{P}$ non-unital the free $\op{P}$-algebra functor is given by $\mathsf{F}_\op{P}(X)= \mathsf{F}_{\underline{\op{P}}}(X) = \underline{\op{P}}\circ X\cong (\op{P}\circ X)\coprod X$.
\end{remark}

\subsection{Algebraic $\mathbb{V}$-colored cooperads.}\label{coopsec}

In this paper we will not study the categorical dual of the notion of a groupoid colored operad in full generality.  Rather, after \cite{LV}, we will consider a specialization suited to Koszul duality.  These will be called conilpotent groupoid colored cooperads.   

For this, we now narrow our focus in two ways.  First, we fix the base category $\op{C}$ to be the category of dg vector spaces over a field $k$ of characteristic $0$ with homological grading conventions.  Second we impose the following assumption, which we carry for remainder of this paper:
\begin{assumption}\label{assumption:finiteaut}
	The group Aut$(v)$ is finite for every $v\in ob(\mathbb{V})$.
\end{assumption}

This assumption would eventually be needed below, in that several constructions require summing over automorphism groups, but imposing it at this point allows us the convenience of identifying invariants with coinvariants in the definition of cooperads and their coalgebras.

\subsubsection{The comonad structure on $\mathbb{T_V}$.}
Recall our calculation:

\begin{equation*}
\mathbb{T_V}\circ\mathbb{T_V}(A)(\vec{v}) \cong \ds\bigoplus_{\substack{ T, \{U_w\}  \\ T \text{ of type }  \vec{v} \text{ and}  \\  T\circ\{U_w\} \text{  well defined.}}} \left[\ds\bigotimes_{ x \in \coprod_w\text{Vert}(U_w)} A(x)\right]_{\text{Edges}(T \circ \{U_w\})}
\end{equation*}
For a $\mathbb{V}$-colored tree $R$ of type $\vec{v}$ we define a map
$A(R) \to \mathbb{T_V}\circ\mathbb{T_V}(A)(\vec{v})$ as follows.  Pick a summand, thus fixing $T$ and $\{U_w\}$ for which the total substitution $T\circ\{U_w\}$ is defined.  If $R=T\circ\{U_w\}$, then this summand is canonically isomorphic to $A(R)$ and we include via this isomorphism.  If $R\neq T\circ\{U_w\}$ then the map to this summand is $0$.  Observe that for any tree $R$ there are only finitely many total substitutions, since each total substitution is specified by a fixed quantity of (possibly empty) subsets of the finite set of edges.  Thus this map lands in the direct sum.  It is also clearly $Aut(\vec{v})$-equivariant.  Taking the coproduct of these maps induces a natural transformation $\mathbb{T_V}\Rightarrow \mathbb{T_V}\circ\mathbb{T_V}$, and associativity of the total substitution (Lemma $\ref{totalsub}$) ensures this natural transformation is comonadic, with counit given by projection to corollas.

\begin{definition} \label{cooperads} A conilpotent $\mathbb{V}$-colored cooperad is a coalgebra over the comonad $\mathbb{T_V}$.
\end{definition}
All cooperads we consider in this paper are conilpotent unless explicitly stated.  The reader expecting conilpotent cooperads to come equipped with a counit is referred to Remark $\ref{counitrmk}$.  Note the functor $\overline{\mathbb{T}}_\mathbb{V}$ does not have a similar comonadic structure, since in the presence of the empty $\mathbb{V}$-colored trees there are infinitely many total substitutions which can result in a given tree.

By definition, then, a conilpotent $\mathbb{V}$-colored cooperad $\op{Q}$ comes with a morphism $\op{Q}(\vec{v}
)\to \op{Q}(T)$ for every $\mathbb{V}$-colored tree of type $\vec{v}$, which we call blowing up or expanding the tree.  Since $\mathbb{T_V}$ is defined using a direct sum, for each vector in $\op{Q}(\vec{v})$, there are only finitely many such expansions which are non-zero.  Notice that in the $\mathbb{V}$-colored case, this is a substantially more restrictive condition than in the uncolored case.  In particular there may be infinitely many $\vec{v}$-trees with one edge, but only finitely many such expansions can be non-zero. The sum of such expansions specifies a coassociative map of $\mathbb{V}$-colored sequences $\Delta_\op{Q}\colon \op{Q}\to\op{Q}\circ\op{Q}$.

By general considerations in the theory of comonads, the forgetful functor which takes a $\mathbb{V}$-colored cooperad to its underlying $\mathbb{V}$-colored sequence has a right adjoint, which we denote by $F_c$,
Explicitly, $F_c(A)=\mathbb{T_V}(A)$, with cooperad structure given by the comonadic structure map $\mathbb{T_V}(A)\to \mathbb{T_V}(A)\circ\mathbb{T_V}(A)$.

\begin{remark}\label{counitrmk}  In the uncolored case, our definition of a conilpotent cooperad is equivalent but not equal to the definition given in \cite{LV}, which is given as a condition on coaugmented and counital cooperads.  Concretely, if $\op{Q}$ is a conilpotent cooperad in the sense of Definition $\ref{cooperads}$ then the $\mathbb{V}$-colored sequence $\underline{\op{Q}} : = \op{I}\oplus \op{Q}$ is a conilpotent cooperad in the sense of \cite{LV} (see {\it loc.cit.} Proposition 5.8.5).  With this equivalence in mind we may refer to a $\mathbb{V}$-colored sequence as a {\bf unital conilpotent cooperad} if it is isomorphic to a $\mathbb{V}$-colored sequence of the form $\op{I}\oplus \op{Q}$, for a (non-unital) conilpotent cooperad $\op{Q}$.
\end{remark}

\subsubsection{Coalgebras over cooperads.}  

\begin{definition}\label{conilcoalg}  Let $\op{Q}$ be a conilpotent cooperad.    A conilpotent $\op{Q}$-coalgebra is a map of $\mathbb{V}$-colored sequences $\eta_X\colon X\to \op{Q}\circ X$ such that $(\Delta_\op{Q} \circ id_X) \circ \eta_X = (id_\op{Q} \circ \eta_X) \circ \eta_X$.  For brevity, we refer to conilpotent $\op{Q}$-coalgebas as just $\op{Q}$-coalgebras from now on.
\end{definition}

In particular, a $\op{Q}$-coalgebra structure on $X$ is a collection of morphisms
	\begin{equation*}
	X(v_0)\to \left[ \ds\bigoplus_{\vec{v}=(-...-; v_0)}\op{Q}(\vec{v})\tensor_{in(\vec{v})} X(in(\vec{v})) \right]_{S_{|\vec{v}|}}
	\end{equation*}
	
which are compatible with cooperadic tree expansions in $\op{Q}$.  We remark that a $\op{Q}$-coalgebra structure always extends to a map $X \to \underline{\op{Q}}\circ X$ via the identity.  

\begin{lemma} \label{cofree coalgebra} The forgetful functor from $\op{Q}$-coalgebras to $\mathbb{V}$-modules has a right adjoint, which we denote $F^c_\op{Q}$.  It is given by the formula $F^c_\op{Q}(X) = \underline{\op{Q}}\circ X = (\op{Q}\oplus \op{I})\circ X \cong (\op{Q}\circ X)\oplus X$.
\end{lemma}

\begin{proof}  
	This is the coalgebraic analog of the formula in Remark $\ref{non-unital algebras}$ above.  The $\mathbb{V}$-module $(\op{Q}\oplus \op{I})\circ X$ has the structure of a $\op{Q}$-coalgebra via the sum of the maps $\op{Q}\circ X \stackrel{\Delta_\op{Q}}\longrightarrow \op{Q}\circ \op{Q}\circ X $ and $\op{Q}\circ X\cong \op{Q}\circ \op{I}\circ X$ with $\op{I}\circ X$ mapping to zero.

To show that $F^c_\op{Q}$ is right adjoint to the forgetful functor we first observe that a $\op{Q}$-coalgebra map $A\to \op{Q}\circ X \oplus X$ specifies a $\mathbb{V}$-module map $A\to X$ by projection to $X$.   Conversely, if $A$ is a $\op{Q}$-coalgebra and $X$ is a $\mathbb{V}$-module, then a $\mathbb{V}$-module map $\phi \colon A\to X$ specifies the $\op{Q}$-coalgebra map $\op{Q}(\phi) \eta_A \oplus \phi$. These correspondences are natural and inverse to each other, hence the claim.
\end{proof}

\subsubsection{Effect of linear dualization}

We denote the linear dual of a vector space $V$ by $V^\ast$.  If $A$ is a $\mathbb{V}$-colored sequence then $A^\ast$ has the natural structure of a $\mathbb{V}$-colored sequence.

\begin{lemma}
	The linear dual of a conilpotent cooperad is a non-unital operad.
\end{lemma}
\begin{proof}  Given a conilpotent cooperad structure $\op{Q} \to  \mathbb{T_V}(\op{Q})$ we have structure map, we take its linear dual and precomposing with the canonical injections:
\begin{equation*}
\mathbb{T_V}(\op{Q}^\ast)(\vec{v})=\ds\bigoplus \op{Q}^\ast(T) \hookrightarrow \ds\bigoplus (\op{Q}(T))^\ast \hookrightarrow \ds (\bigoplus \op{Q}(T))^\ast \to \op{Q}^\ast(\vec{v}).
\end{equation*}

The coassociativity of the original coalgebra structure maps will give associativity of this composite, and hence an operad structure on $\op{Q}^\ast$.
\end{proof}
For the converse of this result, we first define:
\begin{definition}\label{fddef}
	A $\mathbb{V}$-colored sequence $A$ is finite dimensional if each $A(\vec{v})$ is finite dimensional.  An operad or cooperad is finite dimensional if its underlying $\mathbb{V}$-colored sequence is.  A $\mathbb{V}$-module is finite dimensional if each $X(v)$ is finite dimensional.
\end{definition}
Unlike in the uncolored case, an operad being finite dimensional is not a sufficient condition for the linear dual to be a cooperad if there are infinitely many colors.  In practice however, our examples of interest satisfy the following additional condition which alleviates this issue:

\begin{definition}\label{reduced} 
	A $\mathbb{V}$-colored sequence $A$ is called reduced if for each $v_0\in \text{ob}(\mathbb{V})$ there are only finitely many $\mathbb{V}$-colored trees $T$ with root colored by $v_0$ such that $A(T)$ is non-zero.  A non-unital $\mathbb{V}$-colored operad $\op{P}$ is called reduced if its underlying $\mathbb{V}$-colored sequence is reduced.  A unital, augmented $\mathbb{V}$-colored operad is called reduced if its augmentation ideal is reduced.
\end{definition}

The reduced hypothesis will ensure that direct products which arise by dualizing are also direct sums.  The following results are immediate.

\begin{lemma}\label{lindualop}  The linear dual of a finite dimensional and reduced non-unital operad is a conilpotent cooperad.
\end{lemma}

\begin{lemma}\label{fdfreeop} Let $A$ be a finite dimensional, reduced $\mathbb{V}$-colored sequence.  Then $F(A)$ is a finite dimensional, reduced $\mathbb{V}$-colored operad and $F(A)^\ast\cong F_c(A^\ast)$.
\end{lemma}

 \begin{lemma}\label{fdfree} Let $\op{P}$ be a finite dimensional, reduced $\mathbb{V}$-colored operad and let $X$ be a finite dimensional $\mathbb{V}$-module.  
 	Then $F_\op{P}(X)^\ast$ is also a finite dimensional and $F_\op{P}(X)^\ast \cong  F^c_{\op{P}^\ast}(X^\ast)$.
 \end{lemma}

\subsection{Quadratic duality for groupoid-colored operads.}\label{KDsec}
We now begin the development of Koszul duality for groupoid colored operads.  To fix a context where this is possible we make;

\begin{assumption}  From now on, all unital groupoid colored operads are augmented unless stated otherwise.
\end{assumption}
Formally adjoining a unit to a non-unital operad gives a unital operad with a canonical augmentation, and most of our examples of interest arise in this manner.   Below we may consider either non-unital or unital and augmented operads and use the variant of the free operad construction $F$ which is appropriate to the context, keeping in mind Formula $\ref{freeunital}$. 

 \subsubsection{Suspension and determinants}  We write $\Sigma$ and $\Sigma^{-1}$ for degree shift operators in dg vector spaces.  If $A$ is a $\mathbb{V}$-module, we define $\Sigma^\pm A$ as the object-wise application of $\Sigma^{\pm}$.
By abuse of notation we let $\Sigma^\pm k$ be the constant $\mathbb{V}$-module with target $\Sigma^\pm k$  -- this means the field $k$ in degree $\pm 1$ as a graded vector space and every automorphism acts by the identity. 
Then we define $End_{\Sigma k }=: \Lambda^{-1}$ and $End_{{\Sigma^{-1}} k}=: \Lambda$.  We define $s^\pm := \Lambda^{\pm}\tensor - $, denoting the color-wise tensor product.  In particular $s$ raises degrees and $s^{-1}$ lowers degrees.  

We call the functor $s$ (resp.\ $s^{-1}$) {\bf operadic suspension} (resp. desuspension), but we may view it as an endofunctor on the category of operads, the category of $\mathbb{V}$-colored sequences and the category of conilpotent cooperads.    For cooperads we observe that the standard basis of each 1-dimensional vector space $\Lambda^\pm(\vec{v})$ gives rise to a canonical isomorphism of $\mathbb{V}$-colored sequences $(\Lambda^{\pm})^\ast\cong \Lambda^\mp$.  This isomorphism endows $\Lambda^{\pm}$ with decomposititon maps, and although $\Lambda^{\pm}$ do not form cooperads by our definition (there are potentially infinitely many such expansions with a fixed source), the result when tensoring $\Lambda^{\pm}$  with a conilpotent cooperad will be a conilpotent cooperad.

\begin{definition}\label{det}
	
	Let $V$ be a finite dimensional vector space of dimension $n$.  We define Det$(V)$ to be the top exterior power of $V$.  In particular Det$(V)$ is a 1-dimensional graded vector space concentrated in degree $n$.  We also define det$(V)$ to be $\Sigma^{-n}Det(V)$.  In particular det$(V)$ is concentrated in degree $0$.

	For a finite set $X$ we define Det$(X)$ to be Det(span$\{X\}$). In particular Det$(X)$ is a one-dimensional vector space concentrated in degree $|X|$ with an alternating action of $S_{|X|}$.  We also write det$(X)$ for $\Sigma^{-|X|}\text{Det}(X)$, concentrated in degree $0$.  We refer to a unit vector in det$(X)$ as a {\bf mod 2 order of the set} $X$.  Finally we define $\text{Det}^{-1}(X)$ to be $\Sigma^{-2|X|}\text{Det}(X)$ and observe that $\text{Det}(X)\tensor \text{Det}^{-1}(X)\cong k$ as $S_{|X|}$-modules.\end{definition}

\begin{lemma}\label{detfacts} \cite[Lemma 4.7]{GeK2} Given finite dimensional vector spaces $V$ and $W$, 
$\text{Det}(V\oplus W)\cong \text{Det}(V)\tensor \text{Det}(W)$.	
\end{lemma}

\begin{lemma} \label{suspfreelemma}
	Suspension and desuspension commute with the free operad and cofree cooperad constructions.
\end{lemma}

\begin{proof}  Note the isomorphism $\Lambda(\vec{v})\cong \text{Det}(\{v_1\cdc v_n\})\tensor \text{Det}^{-1}(\{v_0\})$ for each color scheme $\vec{v}=(v_1\cdc v_n; v_0)$.
	Then the only difference with the uncolored case \cite[Section 2.10]{GeK1} is that we mod out by the action of the automorphisms at each edge of a tree.  But since $\Lambda(\vec{v})$ has trvial $Aut(\vec{v})$-action we may still, as in {\it loc.cit.,} apply Lemma $\ref{detfacts}$ to show that $(sA)(T)\cong A(T)\tensor \Lambda(\vec{v})$ for a $\vec{v}$-tree $T$, from which the claim follows.
\end{proof}

\subsubsection{Operadic ideals} See {\it e.g.\ }\cite[Section 2.1.3]{GK} in the uncolored case. Given a $\mathbb{V}$-colored sequence, a $\mathbb{V}$-colored subsequence $I\subset A$ is a collection of subspaces $I(\vec{v})\subset A(\vec{v})$ closed under the $S_{|\vec{v}|}$ and $Aut(\vec{v})$ actions.  If $\op{P}$ is an operad, an ideal $I\subset \op{P}$ is a $\mathbb{V}$-colored subsequence satisfying the property that the image of any structure map $\eta_{T}$ (see Formula $\ref{operadstructuremap}$) on a tensor with a factor in $I$ lands in $I$.
Note this property is unambiguous with respect to the edge action, since each $I(\vec{v})$ is $Aut(\vec{v})$ closed.

Given an ideal $I\subset \op{P}$, we may form an operad $\op{P}/I$ with underlying $\mathbb{V}$-colored sequence $\op{P}(\vec{v})/I(\vec{v})$ and the inherited structure maps.  In particular the collection of quotient maps $\op{P}\to \op{P}/I$ assemble to a morphism of operads. Given a family of subspaces (or subsets) $R(\vec{v})\subset \op{P}(\vec{v})$, we let $\langle R \rangle$ denote the intersection of all ideals of $\op{P}$ which contain $R$.  This is easily seen to be an ideal itself.

\subsubsection{Quadratic data.}
Let $E$ be a $\mathbb{V}$-colored sequence. 
A homogeneous element in the direct sum $F(E)(\vec{v})$ (Equation $\ref{free1}$) 
specifies a tree, and hence a number of edges.  We call this integer the {\bf weight} of the homogeneous element.  It specifies an additional grading on $F(E)(\vec{v})$, which we denote $F^r(E)(\vec{v})$.

A quadratic presentation of an operad $\op{P}$ is an isomorphism $\op{P}\cong F(E)/\langle R \rangle$, where $R$ is a $\mathbb{V}$-colored subsequence concentrated in weight $1$, i.e.\ $R(\vec{v})\subset F^1(E)(\vec{v})$.  Conversely, to any such $E$ and $R$ we may present a quadratic operad $\op{P}(E,R):= F(E)/\langle R \rangle$, and so we call such a pair $(E,R)$ {\bf quadratic data}.  Here there is a notion of both unital and non-unital quadratic operads depending on the which variant of the free operad construction is used.
 
To quadratic data $(E,R)$ we may also associate a conilpotent cooperad, denoted $\op{Q}(E,R)$.  It is defined to be the union of all sub-cooperads of $\op{Q}\subset F_c(E)$ for which the following composite of $\mathbb{V}$-colored sequences vanishes:
$\op{Q}\hookrightarrow F_c(E)=F(E) \twoheadrightarrow  F^1(E)/R.$

\subsubsection{Quadratic duality.} 

If $\op{P}(E,R)$ is a non-unital quadratic operad, the quadratic dual conilpotent cooperad is denoted $\op{P}^\uex$ and is defined to be  $\op{Q}(\Sigma E,\Sigma^2 R)$.  If $\op{P}(E,R)$ is a unital quadratic operad, the quadratic dual unital conilpotent cooperad (see Remark $\ref{counitrmk}$) is also denoted $\op{P}^\uex$ and is defined to be  $\op{Q}(\Sigma E,\Sigma^2 R)\oplus \op{I}$.  In both cases, the quadratic dual operad of the quadratic operad $\op{P}(E,R)$ is defined to be the operad $s(\op{P}^\uex)^\ast$ and is denoted $\op{P}^!$.

\subsection{Bar-cobar duality for groupoid colored operads}
In this subsection we define a bar-cobar adjunction between non-unital operads and conilpotent cooperads.   
For (co)augmented, (co)unital operads and cooperads, the bar-cobar constructions are defined by applying these constructions to the (co)augmentation ideal and then adjoining the (co)unit.

\subsubsection{The bar construction of an operad.}

Let $\op{P}$ be a non-unital operad and consider the sequence of $\mathbb{V}$-colored sequences
\begin{equation*}
F_c(\Sigma\op{P}) \to F^1_c(\Sigma \op{P}) = F^1(\Sigma \op{P})\cong \Sigma^2 F^1( \op{P}) \to\Sigma^2\op{P}\to \Sigma\op{P}
\end{equation*}
The first map is projection, the second is the operadic structure map and the last is a degree shift.  This composite has a unique extension to a cooperad map by the universal property of cofreeness.  Call this cooperad map $\partial$.  

We then argue $\partial^2=0$ as usual.  To see this use Lemma $\ref{detfacts}$ to identify
\begin{equation}
(\Sigma\op{P})(T)\cong  \text{Det}(\text{Vert}(T))\tensor \op{P}(T) \cong \Sigma \text{Det}(\text{Edges}(T))\tensor \op{P}(T)
\end{equation}

This last isomorphism is given by identifying every edge with the vertex above it.  The terms in $\partial^2$ will contract two edges in a fixed order, and the presence of the Det(Edges$(T)$) factor means that switching this order will produce a negative sign.  The sum of all such contractions will thus cancel in pairs.  We define $B(\op{P})=(F_c(\Sigma\op{P}), \partial+d_\op{P})$, where $d_\op{P}$ is the differential induced by that of $\op{P}$. 

\subsubsection{The cobar construction of a cooperad.}

Let $\op{Q}$ be a conilpotent cooperad and consider the sequence of $\mathbb{V}$-colored sequences
\begin{equation*}
\Sigma^{-1}\op{Q} \to
\Sigma^{-2}\op{Q} \to
\Sigma^{-2} F^1_c( \op{Q}) \cong
F^1(\Sigma^{-1} \op{Q}) \hookrightarrow
F(\Sigma^{-1}\op{Q}) 
\end{equation*}
The first map is a shift, the second is the cooperadic structure map and the last is inclusion.  This composite has a unique extension to an operad map $\partial$ by the universal property of freeness, and $\partial^2=0$ as above.  We define $\Omega(\op{Q})=(F(\Sigma^{-1}\op{Q}), \partial+d_\op{Q})$.

\begin{definition}
	A map of $\mathbb{V}$-colored sequences $\phi\colon A\to A^\prime$ is called a quasi-isomorphism if each $\phi_{\vec{v}}\colon A(\vec{v})\to A^\prime(\vec{v})$ is a quasi-isomorphism.  A map of operads or cooperads is a quasi-isomorphism if the underlying map of $\mathbb{V}$-colored sequences is.
\end{definition}

\begin{lemma}\label{barcobaradjunction} The pair $(\Omega,B)$ form an adjoint pair whose unit and counit are quasi-isomorphisms.
\end{lemma}
\begin{proof}  This is a straight-forward generalization of \cite[Theorem 3.2.16]{GK} as well as \cite[Theorem 6.6.3]{LV}.  It may also be viewed as a specific case of \cite[Theorem 7.4.3]{KW} or as a special case of Lemma $\ref{barcobaralg}$ below.
\end{proof}

\subsubsection{Koszulity}

Let $\op{P}\cong F(E)/\langle R \rangle$ be a quadratic operad.  From $E\hookrightarrow \op{P}$, form the cooperad map $F_c(\Sigma E)\to F_c(\Sigma\op{P})$.  Precomposing with the inclusion $\op{P}^\uex \to F_c(\Sigma E)$ induces a map of cooperads $\op{P}^\uex \to F_c(\Sigma\op{P})$.  This latter cooperad map sends generators $\Sigma E$ to weight zero in the target, so contracting an edge is not possible (non-trivially) on the image of this map.  Thus this is a dg map with respect to the differential in the bar construction.  We call this map of cooperads $\zeta_{\op{P}^\uex}\colon \op{P}^\uex \to B(\op{P})$.

\begin{definition} \label{Koszuldef} A quadratic $\mathbb{V}$-colored operad $\op{P}$ is Koszul if the map $\zeta_{\op{P}^\uex}\colon \op{P}^\uex \to B(\op{P})$ is a quasi-isomorphism.  
\end{definition}

Combining this definition with Lemma $\ref{barcobaradjunction}$ we see that such a $\op{P}$ is Koszul if and only if the composition $\Omega(\op{P}^\uex)\to \Omega(B(\op{P}))\to\op{P}$ is a quasi-isomorphism.

\begin{definition}  If $\op{P}$ is Koszul we define $\op{P}_\infty:=\Omega(\op{P}^\uex)$.
\end{definition}

\subsection{Bar-cobar duality for algebras over groupoid colored operads.}\label{abcsec}
We define a {\bf twisting morphism} from a conilpotent cooperad $\op{Q}$ to a non-unital operad $\op{P}$ to be a map of $\mathbb{V}$-colored sequences $\alpha \colon\op{Q}\to\op{P}$ such that the induced map $\Omega(\op{Q})\to\op{P}$ is a morphism of operads, or equivalently such that the coinduced map $\op{Q}\to B(\op{P})$ is a map of cooperads.

\begin{example}\label{twistingmorphisms}  The counit of the adjunction $\Omega(B(\op{P}))\to\op{P}$ gives rise to a twisting morphism $B(\op{P})\to\op{P}$.  The unit of the adjunction $\op{Q}\to B(\Omega(\op{Q}))$ gives rise to a twisting morphism $\op{Q}\to \Omega(\op{Q})$.  If $\op{P}$ is quadratic, the cooperad map $\op{P}^\uex\to B(\op{P})$ gives rise to a twisting morphism $\op{P}^\uex\to\op{P}$.
\end{example}

A twisting morphism $\alpha\colon\op{Q}\to\op{P}$ may be used to define a differential on the monoidal product $\op{P}\circ\op{Q}$ of Definition $\ref{monoidaldef}$. This differential is the unique derivation which extends $\op{Q}\stackrel{\Delta}\to\op{Q}\circ\op{Q}\stackrel{\alpha\circ id}\longrightarrow\op{P}\circ\op{Q}$ with $\Delta$ as in Subsection $\ref{coopsec}$.  The resulting dg $\mathbb{V}$-colored sequence is called the (left) Koszul complex of $\alpha$ and is denoted $\op{P}\circ_\alpha\op{Q}:=(\op{P}\circ\op{Q}, d_\alpha)$.  One may also form the (right) Koszul complex of $\alpha$, $(\op{Q}\circ\op{P}, d_\alpha)$.  A twisting morphism is called Koszul if these complexes are acyclic.  In the case where $\op{P}$ is quadratic, comparing $\op{P}\circ\op{P}^\uex$ with $\op{P}\circ B(\op{P})$ with differentials coming from the twisting morphisms of Example $\ref{twistingmorphisms}$ gives the following criterion for Koszulity.  See \cite[Theorem 6.6.1]{LV} in the uncolored case.

\begin{lemma}[Koszul Criterion] \label{KC}
	Let $\op{P}$ be a quadratic operad with twisting morphism $\alpha\colon\op{P}^\uex\to\op{P}$.  Then $\op{P}$ is Koszul if and only if $\alpha$ is Koszul.
\end{lemma}

We now give the bar-cobar construction associated to a Koszul twisting morphism.  For the remainder of this section we assume that $\op{P}$ is Koszul and consider $\alpha \colon \op{P}^\uex\to \op{P}_\infty$.  We also recall (Definition $\ref{conilcoalg}$) our standing convention that coalgebras are conilpotent.

\subsubsection{Bar construction of a $\op{P}_\infty$-algebra}
In this section we define a functor $
\mathsf{B}\colon \{\op{P}_\infty\text{-algebras} \}\to \{\op{P}^\uex\text{-coalgebras} \}$
in two steps.  We first define the underlying graded $\op{P}^\uex$ coalgebra $\mathsf{B}(A):=\mathsf{F}^c_{\op{P}^\uex}(A) \cong (\op{P}^\uex \circ A)\oplus A$ after Lemma $\ref{cofree coalgebra}$.  We then endow $\mathsf{B}(A)$ with a differential as follows.
First, given a color scheme $\vec{v}$, consider the sequence
\begin{equation*}
\op{P}^\uex(\vec{v})\tensor_{in(\vec{v})} A(in(\vec{v}))
\to \op{P}_\infty(\vec{v})\tensor_{in(\vec{v})} A(in(\vec{v})) \to
A(v_0)
\end{equation*}
The first map in this sequence is induced by the twisting morphism $\op{P}^\uex\to \Omega\op{P}^\uex$ (Example $\ref{twistingmorphisms}$).   The second is the $\op{P}_\infty$ structure map (Formula $\ref{algmaps}$).  In particular this composite is degree $-1$ and $Aut(v_0)$-equivariant.  These maps assemble to a map of $\mathbb{V}$-modules $\op{P}^\uex \circ A \to A$, and by taking the direct sum with $id_A$, we have a map of $\mathbb{V}$-modules $\mathsf{F}^c_{\op{P}^\uex}(A)\to A$.  By the cofreeness of $F^c_{\op{P}^\uex}$, this map of $\mathbb{V}$-modules extends uniquely to a map of $\op{P}^\uex$-coalgebras.  Call this map $\partial \colon \mathsf{B}(A)\to\mathsf{B}(A)$; it has degree $-1$.

\begin{lemma} $\partial \colon \mathsf{B}(A)\to\mathsf{B}(A)$ is square zero.
\end{lemma}
\begin{proof}  A straight forward generalization of \cite[Lemma 11.2.1]{LV}.
	Colimits in the category of $\mathbb{V}$-colored sequences are created by the functor which forgets the $Aut(\vec{v})$ action at each color scheme.  One may identify $\mathsf{B}(A)$ as a coequalizer in the category of $\mathbb{V}$-colored sequences for the maps $(\op{P}^\uex\circ_\alpha \op{P}_\infty) \circ \op{P}_\infty \circ A \rightrightarrows (\op{P}^\uex\circ_\alpha \op{P}_\infty) \circ A$ induced by the operad structure on $\op{P}_\infty$ and the $\op{P}_\infty$-algebra structure on $A$.  
\end{proof}
This defines $\mathsf{B}$ on objects, namely: $\mathsf{B}(A):= (\mathsf{F}^c_{\op{P}^\uex}(A),\partial)$.  On morphisms we observe that for each color scheme $\vec{v}$, a $\op{P}_\infty$-algebra map $A\to A^\prime$ determines a sequence
\begin{equation}\label{barmorph}
\op{P}^\uex(\vec{v})\tensor_{in(\vec{v})} A(in(\vec{v}))
\to
\op{P}_\infty(\vec{v})\tensor_{in(\vec{v})} A(in(\vec{v})) 
\to
A^\prime(v_0) 
\end{equation}
which in turn assemble to a morphism of $\op{P}^\uex$-coalgebras $\mathsf{B}(A)\to \mathsf{B}(A^\prime)$.  

Since the coinvariants in line $\ref{barmorph}$ are taken with respect to a finite group, if the $\op{P}_\infty$-algebra map $A\to A^\prime$ is a quasi isomorphism, then, by the Kunneth theorem, the induced map 
\begin{equation}\label{eq}
\op{P}^\uex(\vec{v})\tensor_{in(\vec{v})} A(in(\vec{v}))
\to
\op{P}^\uex(\vec{v})\tensor_{in(\vec{v})} A^\prime(in(\vec{v}))
\end{equation}
will be as well, if we restrict to the differential induced by $d_A$, respectively $d_A^\prime$.  Taking the direct sum over all such $\vec{v}$, we find the induced map between the first pages of the spectral sequences associated to $\mathsf{B}(A)$ and $\mathsf{B}(A^\prime)$ via the degree filtration of $\op{P}^\uex$.  Since this map converges to an isomorphism, the original map was a quasi-isomorphism, thus:

\begin{lemma}\label{barqi}
The bar construction $\mathsf{B}$ preserves quasi-isomorphisms.
\end{lemma}


\subsubsection{Cobar construction of a $\op{P}^\uex$-coalgebra}
The cobar construction associated to $\alpha$ is a functor
\begin{equation*}
\Omega\colon \{\op{P}^\uex\text{-coalgebras} \}\to \{\op{P}_\infty\text{-algebras} \},
\end{equation*}
defined as follows.  First define $\Omega(C):= \mathsf{F}_{\op{P}_{\infty}}(C)$ as underlying graded $\op{P}_\infty$-algebra.  We then endow this space with a differential.  Given a color scheme $\vec{v}$, consider the sequence
\begin{equation*}
C(v_0)
\to
\op{P}^{\uex}(\vec{v})\tensor_{in(\vec{v})} C(in(\vec{v})) 
\to
\op{P}_\infty(\vec{v})\tensor_{in(\vec{v})} C(in(\vec{v}))
\end{equation*}
The first map in this sequence is the $\op{P}^\uex$-coalgebra structure map.  The second is the twisting morphism $\op{P}^\uex\to \op{P}_\infty$ (Example $\ref{twistingmorphisms}$).  
These composites assemble to $Aut(v_0)$-equivariant maps $C(v_0)\to \mathsf{F}_{\op{P}_\infty}(C)(v_0)$, and hence to a map of $\mathbb{V}$-modules $C\to \mathsf{F}_{\op{P}_\infty}(C)$.  By the universal property of the free algebra, this map of $\mathbb{V}$-modules extends uniquely to a map of $\op{P}_\infty$-algebras.  Call this map $\partial \colon \Omega(C)\to\Omega(C)$; notice it has degree $-1$.  As above, the Koszul criterion implies:
\begin{lemma} $\partial \colon \Omega(C)\to\Omega(C)$ is square zero.
\end{lemma}


\subsubsection{Bar-cobar resolution}

The following holds for any Koszul twisting morphism, but we will be interested in the $\alpha\colon \op{P}^\uex\to\op{P}_\infty$.


\begin{theorem}\label{barcobaralg}
	$(\Omega, \mathsf{B})$ is an adjoint pair for which the counit $\Omega \mathsf{B}(A)\to A$ is a quasi-isomorphism.
\end{theorem}
\begin{proof} 	
	This follows as in \cite{LV} Theorem 11.3.3; we sketch the ingredients	in the case of the twisting morphism $\alpha\colon \op{P}^\uex\to\op{P}_\infty$.
	Viewed as a $\mathbb{V}$-module, $\Omega\mathsf{B}(A)$ may be identified with $\op{P}_\infty\circ\op{P}^\uex \circ A$, after Remark $\ref{arityzerormk}$.  This complex may be filtered by the total weight of $\op{P}_\infty$ and $\op{P}^\uex$.  Considering the spectral sequence associated to this filtration at the $E^0$ page only sees the differential coming from the  twisting morphism $\op{P}^{\uex}\to\op{P}_\infty$.
	
	 Since $\alpha$ is Koszul, this differential restricted to $\op{P}_\infty\circ\op{P}^\uex$ is acyclic.	
	The result then follows from the convergence of this spectral sequence along with the fact that $H_\ast(\op{P}_\infty\circ\op{P}^\uex \circ A)\cong H_\ast(\op{P}_\infty\circ\op{P}^\uex) \circ H_\ast(A)$.  This last fact follows by employing Maschke's theorem to conclude that each $k[Aut(\vec{v})]$ is semi-simple thus $k[Aut(\vec{v})]$-modules are projective, and so here we have used Assumption $\ref{assumption:finiteaut}$ implying each $Aut(\vec{v})$ is a finite group.
\end{proof}

\begin{remark}\label{cubicalrmk}  One could also prove a converse of this result which says that if the counit of the bar-cobar adjunction is always a quasi-isomorphism, then the operad $\op{P}$ is Koszul.  See \cite[Theorem 11.3.3]{LV} in the uncolored case.  This result would establish Koszulity of the groupoid colored operads encoding many different generalizations of operads (not just modular operads).  One precise statement would be that the linearization of any cubical Feynman category (cf \cite{KW}) is Koszul as a groupoid colored operad.  We save an investigation along these lines for future study.
\end{remark}

\subsubsection{Infinity morphisms}

\begin{definition}  Let $A$ and $A^\prime$ be $\op{P}_\infty$-algebras.  An $\infty$-morphism from $A$ to $A^\prime$, denoted $A \rightsquigarrow A^\prime$, is a $\op{P}^\uex$-coalgebra map $\mathsf{B}(A)\to\mathsf{B}(A^\prime)$.  
\end{definition}
From Lemma $\ref{cofree coalgebra}$, we see that an $\infty$-morphism $f\colon A\rightsquigarrow A^\prime$ has an underlying map $A\to A^\prime$.  We say that an $\infty$-morphism is an $\infty$-{\bf quasi-isomorphism} if this underlying map is a quasi-isomorphism.  In this way, the category of $\op{P}_\infty$-algebras with their strict morphism may be viewed as a non-full subcategory of the category of $\op{P}_\infty$-algebras with their $\infty$-morphisms.  The functor $\mathsf{B}$ extends tautologically to this larger category and combining Lemma $\ref{barqi}$ with Theorem $\ref{barcobaralg}$, shows that 
 $\mathsf{B}$ takes $\infty$-quasi isomorphisms to quasi-isomorphisms.

\subsubsection{Finite dimensional and reduced case.}\label{fdsec}
Consider the special case that $\op{P}$ is finite dimensional and reduced.  Let $A$ be a finite dimensional $\op{P}$-algebra (after Definition $\ref{fddef}$).

\begin{lemma}\label{fdbar} 
	Associated to the twisting morphism $\op{P}^\uex \to \op{P}$ is
	 is an isomorphism of $s^{-1}\op{P}^!$-algebras
	$\mathsf{B}(A)^\ast \cong \Omega(A^\ast)$
\end{lemma}

In this case we define $\mathsf{D}(A):= \mathsf{B}(A)^\ast \cong \Omega(A^\ast)$.  It is a contravariant functor $
\mathsf{D}\colon \{\op{P}\text{-algebras} \}\to \{s^{-1}\op{P}^!\text{-algebras} \}
$.
It follows from Theorem $\ref{barcobaralg}$ that there is a natural quasi-isomorphism $\mathsf{D}^2(A)\to A$.

\subsection{Homotopy transfer theorem} \label{httsec}
The purpose of this section is to show that any $\op{P}$-algebra is $\infty$-quasi-isomorphic to its homology.  
\begin{definition} A  deformation retract of a $\op{P}$-algebra $A$ onto a $\op{P}$-algebra $B$ is a family of deformation retracts:
	\begin{equation}
	\xymatrix{  \save !R(-.7) \ar@(ul,dl)_{h_v} \restore  A(v)  \ar@/_1pc/@{->>}[rr]_{\pi_v} & & \ar@/_1pc/@{->}[ll]_{\iota_v} B(v) 
	} \end{equation}
	indexed by $ob(\mathbb{V})$, such that $h_v,\iota_v,\pi_v$ are $Aut(v)$-equivariant.
\end{definition}

\begin{lemma}\label{eqlem}  Any $\mathbb{V}$-module admits its homology as a deformation retract.  
\end{lemma}

\begin{proof}
	From \cite[Proposition 1.5]{FH} we recall that given a finite group $G$, a representation $Y$ and a $G$-invariant subspace $X\subset Y$, there is a $G$-invariant subspace $Z\subset Y$ with $Y = X\oplus Z$.  We call this subspace a $G$-complement of $X$ in $Y$.  Notice this result requires our characteristic $0$ assumption.
	
Let $A$ be a $\mathbb{V}$-module and fix an object $v\in\mathbb{V}$.  Then $(A(v),d)$ is a chain complex with an action of a group $G:=Aut(v)$ which is finite by Assumption $\ref{assumption:finiteaut}$.  The cycles in each degree $Z_n\subset A_n(v)$ are a $G$-invariant subspace, so we may find a $G$-complement $Z_n\oplus C_n = A_n$.  Next the boundaries $B_n$ form a $G$-invariant subspace of the cycles $B_n\subset Z_n$.  Thus we may choose a $G$-complement $H_n\oplus B_n = Z_n$.  All together this gives us a $G$-invariant decomposition
	\begin{equation*}
	A_n(v)\cong H_n\oplus B_n\oplus C_n
	\end{equation*}
	in each degree.
	
	Observe that the map sending $a\in H_n$ to $[a]\in H_n(A(v))$ specifies an isomorphism 
	$H_n\cong H_n(A(v))$.  Indeed any cycle $z\in Z_n$ may be written $z=(a,b,0)$ which is homologous to $(a,0,0)$, whereas if $(a,b,0)\sim (a^\prime,b^\prime, 0)$ then  $a=a^\prime$. Therefore we may define $\iota_v$ and $\pi_v$ by inclusion and projection respectively.  That is, $\pi_v(a,b,c)=[(a,0,0)]$ and $\iota_v([a,b,0]) = (a,0,0)$.
	Since the decomposition is $G$-invariant, the maps $\pi_v$ and $\iota_v$ are $G$-equivariant.
	
	Relative to this decomposition, the differential $d_n\colon A_n(v)\to A_{n-1}(v)$ takes the form $d(a,b,c)= (0,dc,0)$.  Restricting $d_n$ to $C_n$ specifies a $G$-equivariant isomorphism $C_n\stackrel{\cong}\to B_{n-1}$, and we define $d^{-1}_n\colon B_{n-1}\to C_n$ to be its inverse.  With this we define the homotopy $h_v\colon A_n(v)\to A_{n+1}(v)$ by $h_v(a,b,c)=(0,0,d^{-1}(b))$.  Since the decomposition is $G$-invariant, $h_v$ is $G$-equivariant.  
	
	We then calculate $dh_v(a,b,c)=(0,b,0)$ and $h_vd(a,b,c)=(0,0,c)$, whereas
	$\iota_v\circ \pi_v(a,b,c)=(a,0,0)$. Thus $dh_v+h_vd=id-\iota_v\circ \pi_v$, and so taken over all $v$, the maps $h,\pi,\iota$ form a deformation retract as desired.\end{proof}

Using Lemma $\ref{eqlem}$ we now record the homotopy transfer theorem in this context.  

\begin{theorem} \label{htt} Let $\op{P}$ be a Koszul operad and let $A$ be a $\op{P}_\infty$-algebra.  There is a $\op{P}_\infty$ structure on $H_\ast(A)$ extending the induced $\op{P}$-algebra structure, for which  $H_\ast(A)$ and $A$ are $\infty$-quasi-isomorphic.
\end{theorem}

\begin{proof}  
	
		This result is proven in the uncolored case via explicit combinatorial formulas involving trees labeled by deformation retract data.  These formulas are equally valid for $\mathbb{V}$-colored trees.  The one difference worth emphasizing is that in the groupoid colored context, elements in the free operad are not represented uniquely by colored trees, since we may move automorphisms along internal edges (Formula $\ref{AT}$).  Therefore, for such formulas to be well defined it is essential that the maps making up the deformation retract are equivariant.  Lemma $\ref{eqlem}$ ensures this is the case.  With this in mind we simply sketch the proof, following \cite{LV} in the uncolored case.

	First use Lemma $\ref{eqlem}$ to fix a deformation retract $(\iota,\pi,h)$ between $A$ and $H_\ast(A)$.
	We then use this deformation retract to construct a map of dg cooperads $BEnd_{A}\to BEnd_{H(A)}$.  It is induced by a map of $\mathbb{V}$-colored sequences
$	F_c(\Sigma End_{A})\to \Sigma End_{H(A)}.$
	The map may be described diagrammatically following \cite{LV} p.378.  
	Starting with a $\mathbb{V}$-colored tree whose vertices of type $\vec{v}$ are labeled by functions $A(in(\vec{v})) \to A(out(\vec{v}))$, we label the (internal) edges of the tree by $h$, the leaves by $\iota$ and the root by $\pi$.  This labeling may then be read as a flow chart to construct a linear map  $H_\ast(A)(\text{leaves}(T)) \to H_\ast(A)(\text{root}(T))$.  Degrees are preserved since the internal edges contribute degree $+1$, while the vertices contribute degree $-1$ resulting in the degree $-1$ shift in the target.  The fact that $\pi$ and $\iota$ are equivariant ensures that this is a map of $\mathbb{V}$-colored sequences.  It remains to verify that the induced map respects the differential in the bar construction.  This follows, as in Proposition 10.3.2 of \cite{LV}, from the deformation retract equation.

	Now from Proposition $\ref{barcobaradjunction}$ we know that a $\op{P}_\infty$ algebra structure may be recast as a cooperad map $\op{P}^\uex\to BEnd_{A}$.  Therefore if $A$ is a $\op{P}_\infty$ algebra we may compose with $BEnd_{A}\to BEnd_{H(A)}$ to endow $H_\ast(A)$ with the structure of a $\op{P}_\infty$ algebra.  It remains to observe that $\iota$ may be extended to an $\infty$-morphism compatible with this transferred structure.  Such an $\infty$-morphism is given by maps
	\begin{equation*}
	\op{P}^\uex(\vec{v})\to  Hom(\tensor_{i=1}^n H_\ast(A(v_i)), A(v_0)).
	\end{equation*}
	Following \cite{LV} Theorem 10.3.6, these are defined by decomposing $\op{P}^\uex$ and using the same flow chart formula as above, except we label the root by $h$ in place of $\pi$.  We remark it is also possible to extend $\pi$ to an $\infty$-quasi-isomorphism, see \cite{LV} Proposition 10.3.9. 
\end{proof}

\section{Weak modular operads.} \label{secmodops}
In this section we will
\begin{itemize}
	\item  Define the quadratic operad encoding modular operads, call it $\M$.
	\item  Interpret $\M$ and its quadratic presentation in terms of (nested) graphs. 
	\item  Calculate the quadratic dual of $\M$, specifically $\M^!=s\mathbb{M}_\mathfrak{K}$.
	\item  Prove $\Omega(\mathbb{M}^\ast)\stackrel{\sim}\to\mathbb{M}_\mathfrak{K}$ is  a quasi-isomorphism and conclude $\M$ is Koszul.
	\item Discuss the implicit role of graph associahedra in our proof.
	\item  Introduce the category of weak modular operads as the category of $\Omega(\mathbb{M}_\mathfrak{K}^\ast)$-algebras.%
\end{itemize}
\subsection{Modular Operads}  Modular operads were introduced in \cite{GeK2}.  They are generalizations of operads allowing composition across all connected graphs -- not just rooted trees.  In this subsection we will present the groupoid colored operad encoding modular operads in sets (denoted $M$) and the groupoid colored operad encoding modular operads in dg vector spaces, (denoted $\mathbb{M}$).   The latter will be presented quadratically as $\mathbb{M}=F(\hat{E})/\langle S \rangle$. 

To this end we now specialize to a particular groupoid $\mathbb{V}$ as follows.  Its objects are pairs of non-negative integers $(g,n)$ such that $n+2g-3\geq 0$.  Its only morphisms are automorphisms and $Aut((g,n))=S_n$.  We fix this $\mathbb{V}$ for the remainder of the paper.

\subsubsection{Modular generators}\label{modulargenerators}

Define the $\mathbb{V}$-colored sequence $E$ in sets as follows.  
First, we let
\begin{equation*}
E((g,n);(g+1,n-2)):= \{ \xi_{i,j} :\{i\neq j\}\subset \{1\cdc n\} \}  \times S_{n-2}.
\end{equation*}
The right $S_{n-2}$-action is given by multiplication on the right hand factor and the left $S_n$-action is given 
by $\sigma  \xi_{i,j}=\xi_{\sigma^{-1}(i),\sigma^{-1}(j)} \sigma^\prime$,
where $\sigma\in S_n$ and $\sigma^\prime$ is the composite
\begin{equation}\label{selfgluings}
\{1\cdc n-2 \} \to \{1\cdc n\} \setminus \{i,j\} \stackrel{\sigma}\to  \{1,\cdc n\} \setminus \{\sigma(i),\sigma(j)\} \to \{1\cdc n-2 \}
\end{equation}
where the first and last arrows are the unique order preserving bijections.

Second we let
\begin{equation}\label{nonselfgluings}
E((g_1,n),(g_2,m);(g_1+g_2,n+m-2)):=  \{ \xycirc{i}{j} :1\leq i \leq n, 1 \leq j \leq m \} \times S_{n+m-2}
\end{equation}
Given $i,j$ we define a total order on the disjoint union  $(\{1\cdc n\} \setminus \{i\}) \sqcup (\{1\cdc m\} \setminus \{j\}) $ by
\begin{equation}\label{relabelingorder}
\{1 \cdc i-1\}^l < \{j+1\cdc m\}^r < \{1\cdc j-1\}^r < \{i+1\cdc n\}^l
\end{equation}

Here $l$ and $r$ indicate that the subset is identified with either the left hand or right hand term in the disjoint union  $(\{1\cdc n\} \setminus \{i\}) \sqcup (\{1\cdc m\} \setminus \{j\}) $.  The four subsets in Formula $\ref{relabelingorder}$ are each taken to be totally ordered as usual.

Note this total order is not symmetric, in the sense that switching the order of $n$ and $m$ in its construction gives a different total order to this set.  However these two orderings have the same underlying cyclic order and are related by $t^{m-j+i-1}$, where $t$ is the generator of the cyclic group of order $n+m-2$.

The right $S_{n+m-2}$ action is given by multiplication.  The left $S_n\times S_m$ action is defined to be $(\sigma_1,\sigma_2)  \xycirc{i}{j} = \xycirc{\sigma^{-1}_1(i)}{\sigma^{-1}_2(j)} \sigma^\prime $ where $\sigma^\prime \in S_{n+m-2}$ is defined by
\begin{align*}
\{1\cdc n+m-2\} \to (\{1\cdc n\} \setminus \{i\}) \sqcup (\{1\cdc m\} \setminus \{j\})  \stackrel{\sigma_1,\sigma_2}\to 
\\
(\{1,\cdc n\} \setminus \{\sigma_1(i)\})  \sqcup (\{1\cdc m\} \setminus \{\sigma_2(j)\})  \to \{1\cdc n+m-2 \} 
\end{align*}
where the first map is the unique order preserving bijection with respect to the order $i,j$ and the last map is the unique order preserving bijection with respect to the order $\sigma_1(i),\sigma_2(j)$.

The $S_2$ action 
\begin{equation*}
E((g_1,n_1),(g_2,n_2);(g_1+g_2,n_1+n_2-2))\to E((g_2,n_2),(g_1,n_1);(g_1+g_2,n_1+n_2-2))
\end{equation*}
 is then defined by $\xycirc{i}{j} \times id\mapsto  \xycirc{j}{i} \times t^{m-j+i-1}$.  %
Finally we define $E(else)=\emptyset$.  

\subsubsection{Modular relations}\label{sec:relations}
Define $\hat{E}$ to be the linear $\mathbb{V}$-colored sequence given by the color-wise span of $E$.  We now define the elementary relations, denoted $S\subset F(\hat{E})$.  These relations come in four families and are depicted in Figures $\ref{fig:rel1}$ and $\ref{fig:rel2}$.
Here the notation $i|j$ labeling a vertex corresponds to  $\xycirc{i}{j}$, and the notation and $i,j$ corresponds to $\xi_{i,j}$ above.  By applying permutations to these relations (and moving automorphisms along edges) we may use these relations to produce relations of the form
  \begin{equation}\label{rel1}
   \xi_{l,k} \ \xycirc{i}{j} = \xi_{i^\prime, j^\prime} \ \xycirc{l^\prime}{k^\prime},
  \ \ \ \ \ \ \ \ \ \ 
   \xi_{l,k} \ \xi_{i,j} = \xi_{i^\prime, j^\prime} \ \xi_{l^\prime,k^\prime}
  \end{equation}
  and
   \begin{equation}\label{rel2} \xycirc{i}{j} \ \xi_{l,k}=\xi_{l^\prime, k^\prime} \ \xycirc{i^\prime}{j^\prime},
   \ \ \ \ \ \ \ \ \ \ 
   \xycirc{i}{j} \ \xycirc{l}{k}=\xycirc{l^\prime}{k^\prime}, \ \xycirc{i^\prime}{j^\prime}   
   \end{equation}
 where the $\prime$ notation indicates a relabeling which may be deciphered from the commutation relations of subsection $\ref{modulargenerators}$.  Strictly speaking, the notation in Formulas $\ref{rel1}$ and $\ref{rel2}$ suppresses the color scheme, which the $\prime$-relabeling depends on.  Thus it's best to view these Formulas as four families of relations indexed over all possible color schemes, as the tree pictures indicate.

\begin{figure}
		\includegraphics[scale=.65]{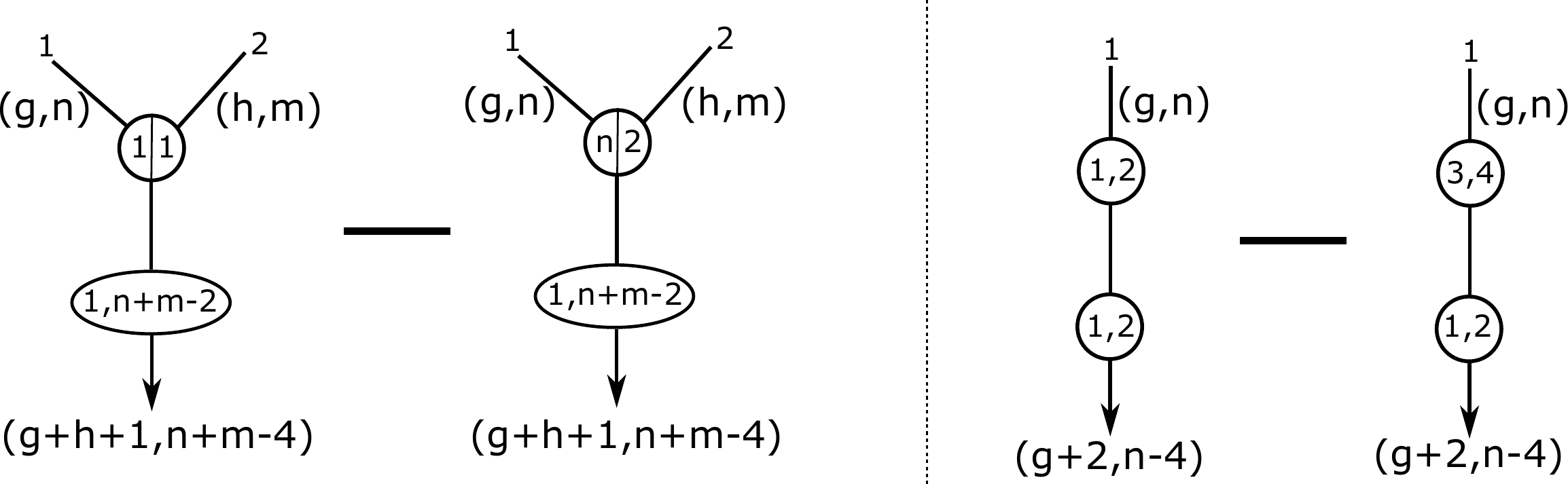} 
		\caption{Elements of $S$.  Left: the relation $(\xi_{1,n+m-2})({\protect\xycirc{1}{1}}) = (\xi_{1, n+m-2})( {\protect\xycirc{n}{2}})$.
		Right: the relation $(\xi_{1,2})(\xi_{1,2}) = (\xi_{1,2})(\xi_{3,4})$.}
				\label{fig:rel1}	
\end{figure}

\begin{figure}
	\includegraphics[scale=.65]{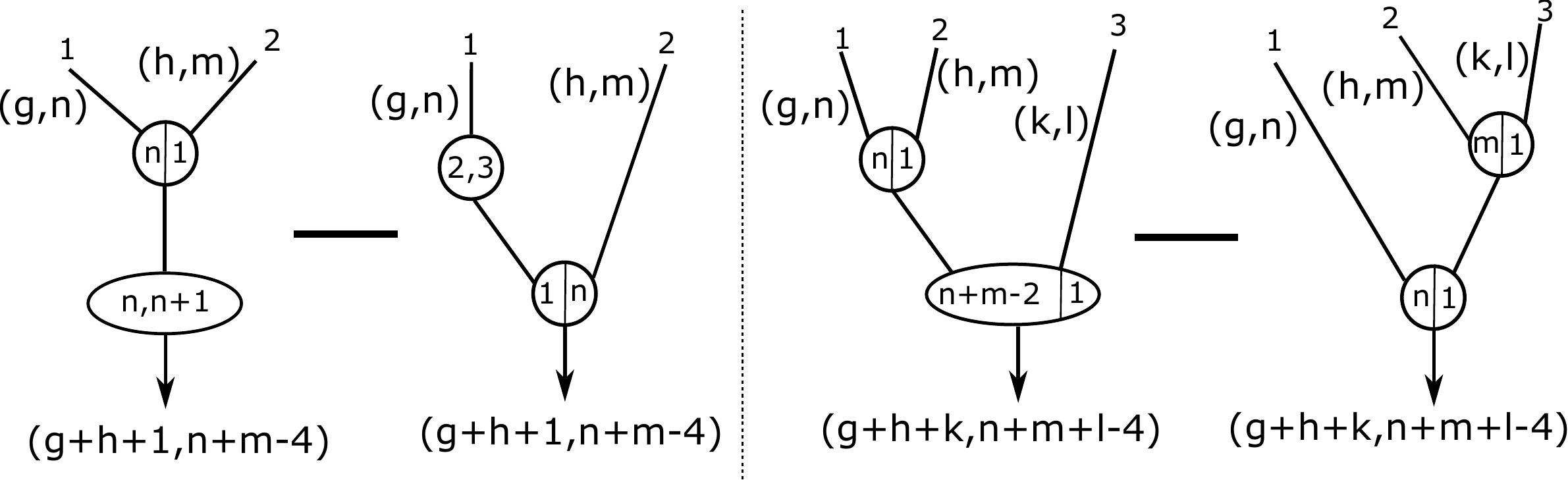}	
		\caption{Elements of $S$.  Left: the relation $(\xi_{n,n+1}) ({\protect\xycirc{n}{1}}) = (\protect\xycirc{1}{n})(\xi_{2,3})$.  Right: the relation $(\protect\xycirc{n+m-2}{1})( \protect\xycirc{n}{1})=(\protect\xycirc{n}{1})(\protect\xycirc{m}{1})$.}
						\label{fig:rel2}
\end{figure}
\subsubsection{Modular Operads}\label{modopssec}

Define $\M:=F(\hat{E})/\langle S \rangle$.  The category of algebras over $\mathbb{M}$ is defined to be the category of {\bf dg modular operads}.  We remark that one may easily verify this agrees with the usual definition by a direct comparison with \cite[Theorem 3.7]{GeK2}.

 Since the span functor is a symmetric monoidal left adjoint, there is a map of sets $F(E)(\vec{v})\to F(\hat{E})(\vec{v})$ for each color scheme $\vec{v}$ which is compatible with the operad structure.  Here we use the notation $F$ for the free operad in both sets and in vector spaces, with the convention that the context is fixed by the input. 

Define an equivalence relation on each set $F(E)(\vec{v})$ by saying two elements are equivalent if they map to the same vector in the composite $F(E)(\vec{v})\to F(\hat{E})(\vec{v})\twoheadrightarrow  \mathbb{M}$.  If we denote the equivalence classes by $M(\vec{v})$, then $M$ inherits the structure of a $\mathbb{V}$-colored operad in sets. We define the category of algebras over $M$ to be the category of {\bf modular operads in sets}. 

There is an inclusion map of sets $M(\vec{v})\hookrightarrow\mathbb{M}(\vec{v})$.
Vectors in the image of this inclusion are called {\bf homogeneous}.   More generally, if $T$ is a $\vec{v}$-tree, vectors in the inclusion $M(T)\hookrightarrow \mathbb{M}(T)$ are called homogeneous.
Homogeneous vectors form bases in their respective vector spaces.  Since each $\mathbb{M}(\vec{v})$ is a finite dimensional vector space with a given basis, it is canonically isomorphic to its linear dual $\mathbb{M}(\vec{v})^\ast$.

\subsubsection{The operad $\mathbb{M}_\mathfrak{K}$.}

\begin{definition}  We define the following quadratic $\mathbb{V}$-colored operad in the category of dg vector spaces: $\mathbb{M}_\mathfrak{K}:= F(\Sigma^{-1} \hat{E})/\langle \Sigma^{-2} S\rangle$.
\end{definition}

\begin{lemma}\label{duallem} There are isomorphisms of operads $\M^!\cong s \mathbb{M}_\mathfrak{K}$ and cooperads $\mathbb{M}^\ast = (\mathbb{M}_\mathfrak{K})^\uex$.
	
\end{lemma}
\begin{proof} 
	We will use the fact that both $\M$ and $\M_\mathfrak{K}$ are finite dimensional and reduced.  In particular $\M^\ast$ is a cooperad in this case (Lemma \ref{lindualop}), and the second statement is a direct consequence of the first, which we now consider.
	
	By definition $\M^!=s(\M^\uex)^\ast $, so it is enough to show that $\M_\mathfrak{K}\cong (\M^\uex)^\ast$.  Now $\M^\uex\subset F_c(\Sigma\hat{E})$ is the maximal cooperad which vanishes on the projection $F_c(\Sigma \hat{E})  \to F^1(\Sigma \hat{E})/\Sigma^2 S$.  Therefore the linear dual $(\M^\uex)^\ast$ is a quotient of $F_c(\Sigma \hat{E})^\ast$ by the ideal generated by $(\Sigma^2 S)^\ast\cong \Sigma^{-2} S^\ast$,  using Lemma $\ref{fdfreeop}$ to identify $F_c(\Sigma \hat{E})^\ast\cong F(\Sigma^{-1}\hat{E}^\ast)$.  We may then use the distinguished basis of each $\hat{E}(\vec{v})$, namely $E(\vec{v})$, to identify $\hat{E}\cong \hat{E}^\ast$, and likewise $S\cong S^\ast$.  Under these identifications we find $(\M^\uex)^\ast$ is isomorphic to a quotient of $F(\Sigma^{-1}\hat{E})$ by the ideal generated by $\Sigma^{-2} S$.
\end{proof}

Getzler and Kapranov defined in \cite{GeK2} the notion of $\mathfrak{K}$-modular operads which were needed to define the algebraic bar construction of a modular operad, which they called the Feynman transform.  The introduction of $\mathbb{M}_\mathfrak{K}$ allows the following definition.

\begin{definition}  Algebras over $\mathbb{M}_\mathfrak{K}$ are called $\mathfrak{K}$-modular operads.
\end{definition} 

That this definition coincides with the original definition of \cite{GeK2} can be seen by a direct comparison with the generators and relations presentation of $\mathfrak{K}$-modular operads in \cite{KWZ} and \cite{Marklodd}.  See also Remark $\ref{graphoperations}$.

We note that an incarnation of Lemma $\ref{duallem}$ has been proven independently by Batanin and Markl (Theorem 12.10 of \cite{BatMar}) as part of a broader study of Koszul duality for the authors' notion of operadic categories.  

\subsection{Graphs and nested graphs.} \label{graphssec2}

We now revisit our definition of an abstract graph (Definition $\ref{graphdef}$) to define the class of graphs which may be used to describe modular operads.  If $X$ is a set, an {\bf $X$-labeled} graph is a graph along with a bijection between $X$ and the set of legs.  A {\bf genus labeled} graph is a graph along with a function from $V$ to the set of non-negative integers.  A genus labeled graph is {\bf stable} if each vertex $v$ has $||v||+2g-3\geq 0$, where $||v||$ is the valence and $g$ is the genus label of $v$.

\begin{definition} \label{modulargraphdef}
	A {\bf modular $X$-graph} is an abstract graph which is
	$X$ labeled, genus labeled, stable and connected and which has at least one edge.  A {\bf modular graph} is a modular $X$-graph for some $X=\{1\cdc n\}$.
\end{definition}

\begin{remark}\label{noedgesremark}  The fact that we defined a modular graph to have at least one edge reflects the fact that $M$ and $\M$ are by definition non-unital operads.  We could consider their unital analogs and allow modular graphs to have no edges, but since we endeavor to take the bar construction of $\M$, it is convenient to have already disregarded these elements and to work in the non-unital framework.  Although it is not technically a modular graph by our definition, the corolla with genus label $g$ and $n$ adjacent flags plays a role below, and we denote it by $\ast_{g,n}$.
\end{remark}

The {\bf internal genus} of a graph is the rank of the first homology group of its associated CW complex.  The {\bf total genus} of a modular graph is the sum of the internal genus and the genus labels of all vertices.  We say a modular graph is of type $(g,n)$ if it has $n$ legs and total genus $g$.  We also say a vertex is of type $(g,n)$ if it has valence $n$ and genus label $g$. 

\subsubsection{Gluing operations for modular graphs.}\label{gluingsec}
Let $\gamma$ be a modular graph with $n$ legs.  For $1\leq i \neq j\leq n$ we define $\xi_{ij}(\gamma)$ to be the modular graph given by gluing together legs $i$ and $j$, thus forming a new edge, and relabeling the legs in the unique way such that the total order of the remaining $n-2$ legs is preserved.  We also define $\sigma \xi_{ij}(\gamma)$ to be the result of permuting the leg labels of $\xi_{ij}(\gamma)$ by $\sigma\in S_{n-2}$.  

Let $\gamma$ and $\gamma^\prime$ be two modular graphs with $n$ and $m$ legs respectively.  For $1\leq i \leq n$ and $1\leq j \leq m$ we define $\gamma \xycirc{i}{j}\gamma^\prime$ to be the modular graph formed by attaching leg $i$ of $\gamma$ to leg $j$ of $\gamma^\prime$, thus forming a new edge, and relabeling the remaining $n+m-2$ legs using the total order given in Formula $\ref{relabelingorder}$. The symmetric group $S_{n+m-2}$ acts on such a graph by permuting the leg labels and we also define $\sigma(\gamma \xycirc{i}{j}\gamma^\prime)$ to be the resulting modular graph.

\subsubsection{Nested graphs}\label{nestsec}

Let $\gamma$ be a graph.  Any subset $A\subset Edges(\gamma)$ specifies a subgraph of $\gamma$ whose edges are $A$ and whose vertices are all those vertices adjacent to an edge in $A$.  We call this graph the {\bf closure of} $A$.

\begin{definition}\label{nestdef}  A {\bf nest} $N$ of a (modular) graph $\gamma$ is a proper, non-empty subset of $Edges(\gamma)$ whose closure is a connected subgraph.
\end{definition} 

	We say two nests are {\bf compatible} if one is contained in the other (ie they are nested) or if their closures are disjoint (so share no edges and share no vertices).  A collection of compatible nests is called a {\bf nesting}.  Every nesting is a poset by containment.  A {\bf nested graph} is a graph along with a choice of nesting, see Figure $\ref{layers}$.  The set of all nestings of a graph is itself a poset by adding or removing nests.  A maximal element in this poset is said to be {\bf fully nested}.  By convention we allow the empty nesting consisting of no nests.  We typically denote nests by $N_i$ and a nesting by $\mathfrak{N}=\{N_1\cdc N_r\}$.  The number of nests in a nesting is denoted $|\mathfrak{N}|$.

If $(\gamma,\mathfrak{N})$ is a nested graph with $N\in \mathfrak{N}$, we define the {\bf $N$-layer} of $\gamma$ to be the graph formed by taking the closure of $N$ along with all adjacent flags and then, for each immediate predecessor of $N$, collapsing its vertices and edges to a single new vertex, keeping track of the genus of the collapse. 
We also refer to the graph outside of all nests, collapsing maximal nests to vertices, as a layer of $(\gamma,\mathfrak{N})$.  In particular the number of layers is always one greater than the number of nests, see Figure $\ref{layers}$.  Note that layers are technically not modular graphs because they do not have a leaf labeling.

\begin{figure}
	\includegraphics[scale=1.1]{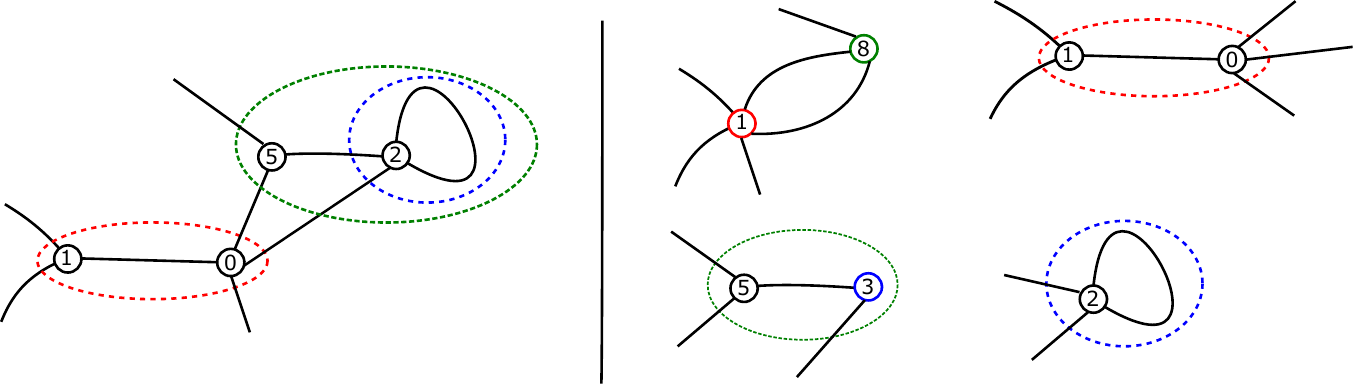} 
	\caption{A nested graph on the left.  This nesting has three nests, each consisting of the set of edges inside a dashed oval.  The four corresponding layer graphs are on the right.  Here the leg labels on the left hand side are suppressed; the genus labels of vertices are shown.}
	\label{layers}
\end{figure}

We now record some basic properties of nestings for future use.
\begin{lemma}\label{eelem}  Let $(\gamma, \mathfrak{N})$ be a nested graph and let $e$ be an edge of $\gamma$.  The subset of the poset $\mathfrak{N}$ consisting of those nests containing $e$ is totally ordered. 
\end{lemma}
\begin{proof}
	Since any two nests in this subset share an edge, they are not disjoint.  Thus any two nests in this subset are comparable, hence the claim.
\end{proof}

Consequently, we may define a function $min_\mathfrak{N}\colon Edges(\gamma)\to \mathfrak{N}\cup \ast$ which sends an edge to $\ast$ if it is contained in no nest, and otherwise sends an edge to the smallest nest it is contained in. 

\begin{lemma}\label{nestssurjective}
	The function $min_\mathfrak{N}$ is surjective. 
\end{lemma}

\begin{proof}
	First, fix $N_j\in \mathfrak{N}$ and let $\{N_i\colon i\in I\} \subset \mathfrak{N}$ be those nests which are contained in but not equal to $N_j$.  There must be an edge which is contained in $N_j$ but not in any of the $N_i$, since otherwise the maximal elements of $\{N_i\colon i\in I\}$ would partition the closure of $N_j$, implying the closure of $N_j$ is disconnected. By definition such an edge is sent to $N_j$ by the function $min_\mathfrak{N}$. 
	
	 Likewise, there must be an edge of $\gamma$ which is contained in no nest, lest the maximal nests in $\mathfrak{N}$ partition $\gamma$, which is not possible since the modular graph $\gamma$ is connected by assumption.  By definition such an edge is sent to $\ast$ by the function $min_\mathfrak{N}$.  Thus $min_\mathfrak{N}$ is surjective.
\end{proof}

\begin{lemma} A nested graph $(\gamma, \mathfrak{N})$ is fully nested if and only if $min_\mathfrak{N}$ is a bijection.
\end{lemma}
\begin{proof} 
	
	First, if this function is not injective we may choose distinct edges $e_1$ and $e_2$ with $min_\mathfrak{N}(e_1)=min_\mathfrak{N}(e_2)$.  Let us first assume that $min_\mathfrak{N}(e_1)=min_\mathfrak{N}(e_2)\neq \ast$, and so $min_\mathfrak{N}(e_1)=min_\mathfrak{N}(e_2)=:N_j\in\mathfrak{N}$. 
The closure of $N_j\setminus e_2$ is a (possibly disconnected) subgraph of the closure of $N_j$.  Let $N_0$ be the set of edges in the connected component of $N_j\setminus e_2$ that contains $e_1$.  We claim $N_0\cup\mathfrak{N}$ is a nesting of $\gamma$.

Indeed, fix $N_i\in\mathfrak{N}$ for which $N_i\cap N_0$ is not empty.  Then $N_j\cap N_i$ is not empty, and so these two nests are comparable (aka nested).  If $N_j\subseteq N_i$ then $N_0\subset N_i$, whence $N_0$ and $N_i$ are compatible.  Otherwise $N_i\subset N_j$ is a proper subset, and since $N_j$ is the smallest nest containing both $e_1$ and $e_2$ we know that neither is in $N_i$.  It follows that $N_i$ is a subset of a connected component of $N_j\setminus e_2$.   Since $N_i\cap N_0$ is not empty it is a subset of the component containing $e_1$, which is $N_0$.  We have thus shown that $N_i$ and $N_0$ are either nested or disjoint.  Thus $N_0\cup\mathfrak{N}$ is a nesting of $\gamma$, and so $\mathfrak{N}$ is not maximal.  

The case where  $min_\mathfrak{N}(e_1)=min_\mathfrak{N}(e_2)= \ast$ follows similarly, replacing the nest $N_j$ in the above argument with the entire graph $\gamma$ (which is technically not a nest by our definition).  We thus conclude that if the function $min_\mathfrak{N}$ is not injective, then $(\gamma, \mathfrak{N})$ is not fully nested nested.  Combining this with Lemma $\ref{nestssurjective}$ proves $(\Rightarrow)$ by contraposition.

Conversely, if $\mathfrak{N}$ is not fully nested then there is a nesting $\mathfrak{N}^\prime$ which contains it.  Applying Lemma $\ref{nestssurjective}$ to $\mathfrak{N}^\prime$ we know that $|Edges(\gamma)|\geq  |\mathfrak{N}^\prime|+1 > |\mathfrak{N}|+1$ and so $min_\mathfrak{N}\colon Edges(\gamma)\to \mathfrak{N}\cup\ast$ can't be injective.  This proves $(\Leftarrow)$ by contraposition.  
\end{proof}

\begin{corollary}\label{edgesvsnests}  Let $\gamma$ be a modular graph.
\begin{enumerate}
	\item Any two full nestings of $\gamma$ have the same number of nests (namely $|Edges(\gamma)-1|$).
	\item A nested graph is fully nested if and only each of its layers has one edge.
	\item  Any nesting of $\gamma$ satisfies $0 \leq  |\mathfrak{N}|\leq Edges(\gamma)-1$.
	\item  In a fully nested graph $(\gamma,\mathfrak{N})$, the function $min_\mathfrak{N}$ induces a canonical isomorphism
	 \begin{equation} \label{ne}
Det(Edges(\gamma))\cong	\Sigma Det(\mathfrak{N}).
	\end{equation}
\end{enumerate}

\end{corollary}

\subsubsection{Correspondence between $F(M)$ and nested graphs.}\label{treegraphcorr}

In this subsection we will give a construction which  associates a nested graph to each element of the free operad $F(M)$.  We emphasize $M$ is a $\mathbb{V}$-colored operad in sets, as is $F(M)$ here.  We first consider the subsets $F(E)\subset F(M)$.


An element in $F(E)$ is represented (non-uniquely) by a tree, all of whose vertices have arity 1 or 2 and whose vertices carry labels of the form $(\xi_{ij},\sigma)$ or $(\xycirc{i}{j},\sigma)$ (Formulas $\ref{selfgluings}$ and $\ref{nonselfgluings}$.).  The leaves of the tree are colored by pairs $(g_i,n_i)$.

To such an element we form a modular graph by reading the tree as a flowchart, starting by labeling the leaves $(g_i,n_i)$ by corollas $\ast_{g_i,n_i}$ and performing the graph operation (see Subsection $\ref{gluingsec}$) corresponding to the label at each vertex (under the convention that the label $(\xi_{ij}, \sigma)$ corresponds to the operation $\sigma^{-1}\xi_{ij}(-)$, and likewise for $\xycirc{i}{j}$.    This construction is clearly independent of the choice of representative of an element in the free operad, since relabelling the legs by $\sigma^{-1}$ followed by $\sigma$ at any stage in the flow chart does nothing.

The resulting modular graph has edges which are in bijective correspondence with the vertices of the tree we started with and it has vertices which are in bijective correspondence with the leaves of the tree we started with.  This graph is also canonically nested, as follows.  To each edge $e$ in the tree $T$ we associate a nest $N_e$ on $\gamma$.  This nest contains those edges of $\gamma$ associated to those vertices of $T$ which lie above $e$.  The collection of such nests forms a nesting on $\gamma$ which is full by Corollary $\ref{edgesvsnests}$, since the number of nests is one less than the number of graph edges.  See Figure $\ref{corr1}$.

The following table summarizes this construction.  We emphasize that the graph $\gamma$ depends on the labels of the tree, even if some of its characteristics depend only on the tree itself.

			\begin{equation}\label{table}
\text{\begin{tabular}{c|c}
 $E(T)\subset F(E)(\vec{v})$ 	& associated nested modular graph $(\gamma,\mathfrak{N})$ of type $\vec{v}$ \\ \hline
  leaves of $T$  &		vertices of $\gamma$   \\ 
vertices of $T$ of arity $1$	& edges of $\gamma$ adjacent to $1$ vertex  	   \\
vertices of $T$ of arity $2$	& edges of $\gamma$ adjacent to $2$ vertices  	   \\
edges of $T$  &	 nests in $\mathfrak{N}$ \\
$E-$labels of $T$  &	layers of $(\gamma, \mathfrak{N})$
	\end{tabular} }
\end{equation}

\begin{figure}

	\includegraphics[scale=1.6]{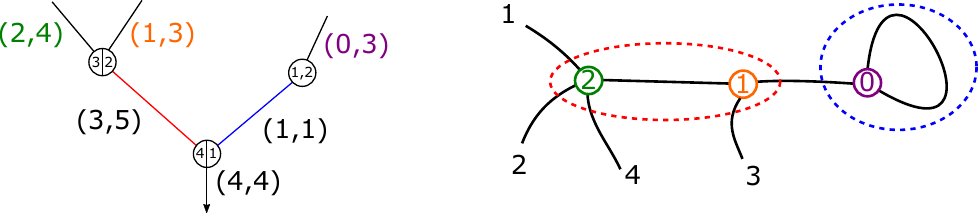} 
	\caption{An element in $E(T)$ on the left corresponds to a fully nested modular graph on the right.  Note here we have suppressed the leaf labeling of the tree on the left hand side.  On the right the dotted lines indicate nests, the numbers inside vertices represent genus labels.
	}
		\label{corr1}
\end{figure}

\begin{remark} \label{bijectiveremark}
 This construction is equivariant with respect to the $in(\vec{v})$ and $S_{|\vec{v}|}$ actions and gives a bijective correspondence between the coinvariants of $F(E)(\vec{v})$ with respect to these actions and nested modular graphs of the type $\vec{v}$ .  
\end{remark}

We may thus informally consider the elements of $F(E)$ as ways to assemble a modular graph.  The following Lemma says that in the quotient $M$ the order in which the edges are glued is forgotten.

\begin{lemma}\label{felem}  Under this construction, two elements in $F(E)$ representing the same element in the quotient $M$ have the same underlying modular graph.  
\end{lemma}
\begin{proof}
 It's enough to observe that the underlying graphs in the relations $S$ are the same.  In particular the relations in $S$ identify the two possible nestings on graphs with two edges. 
\end{proof}

\begin{remark}\label{graphoperations}
One familiar consequence of Lemma $\ref{felem}$ is that modular operads have operations indexed by modular graphs.  
Similarly an algebra over the operad $\mathbb{M}_\mathfrak{K}$ has an operation for each $\Sigma^{-1}E$ labeled tree.  Under the isomorphism 
\begin{equation*}
(\Sigma^{-1} E)(T)   \cong    Det^{-1} (Vert(T)) \tensor  E(T),
\end{equation*}
 and the correspondence of Table $\ref{table}$, we see algebras over $\mathbb{M}_\mathfrak{K}$ have operations indexed by modular graphs having a mod 2 order on their set of edges, as expected.
 
\end{remark}

We may generalize the above construction, replacing $F(E)$ with $F(M)$.  An element in $F(M)$ may be represented by a tree, all of whose vertices are labeled by elements of $M$ of the appropriate color scheme.  These vertex labels in turn each specify (by the above construction) a way to assemble a modular graph, and we may still read such a tree as a flow chart.

The resulting modular graph has vertices which are in bijective correspondence with the leaves of the tree we started with.  Each edge of the graph is associated to a vertex in the tree, but this need no longer be a bijective correspondence.  This modular graph still has a canonical nesting with nests corresponding to the edges of the tree.  In this case a nest associated to a given tree edge $e$ includes those graph edges associated to the tree vertices above $e$.  See Figure $\ref{corr2}$.

In summary, the result of this construction is a nested modular graph with the following correspondences:

			\begin{equation}
\text{\begin{tabular}{c|c}
	$M(T)\subset F(M)(\vec{v})$ 	& associated nested modular graph $(\gamma,\mathfrak{N})$ of type $\vec{v}$ \\ \hline
	leaves of $T$  &		vertices of $\gamma$   \\ 
	edges of $T$  &	 nests in $\mathfrak{N}$ \\
	vertex labels of $T$  &	layers of $(\gamma, \mathfrak{N})$
	\end{tabular} }
\end{equation}

\begin{figure}
	\includegraphics[scale=1.4]{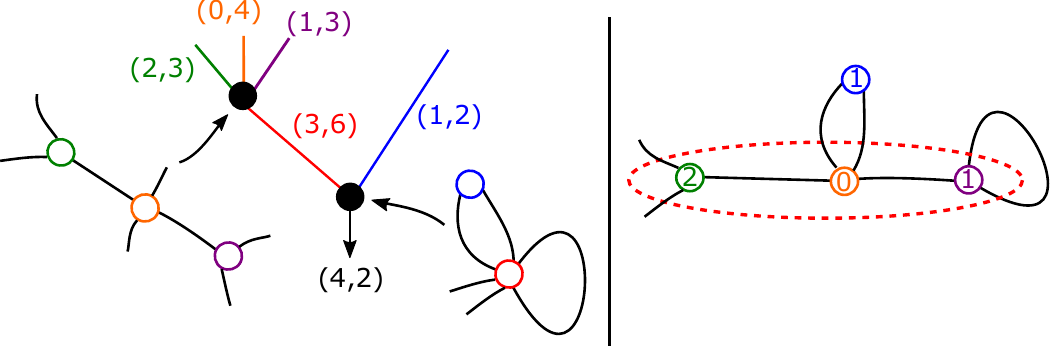} 
	\caption{An element in $M(T)$ on the left specifies a nested modular graph on the right.   As above, the dotted oval indicates a nest, the numbers inside vertices represent genus labels.  We have suppressed the leaf labeling on both sides.
	}
\label{corr2}
\end{figure}

\subsection{$\M$ is Koszul.} We now turn to the main technical result of this paper: that $\M$ is Koszul as a groupoid colored operad.  We give a novel and direct proof of this fact which emphasizes the combinatorics of a certain fiber.  

\begin{theorem}\label{koszulthm}  The $\mathbb{V}$-colored operad $\M$ is Koszul.
	\end{theorem}
\begin{proof}  
From Lemma $\ref{duallem}$ it is enough to show that, for each color scheme $\vec{v}$, the natural map $\Omega(\mathbb{M}^\ast)(\vec{v}) \to  \mathbb{M}_\mathfrak{K}(\vec{v})$ is a quasi-isomorphism.
By definition, as graded vector spaces we have:
\begin{equation*}
\Omega(\mathbb{M}^\ast)(\vec{v}) = \ds\bigoplus_{\vec{v}-\text{trees } T}  (\Sigma^{-1}\mathbb{M}^\ast)(T) \cong \ds\bigoplus_{\vec{v}-\text{trees } T} \Sigma^{-1}(Det^{-1}(Edges(T)))\tensor \mathbb{M}^\ast(T),
\end{equation*}
and the differential in $\Omega(\mathbb{M}^\ast)(\vec{v})$ is given by expanding trees using the dual of the operadic structure maps for $\M$.

Given a $\vec{v}$-tree $T$, the vector space $\mathbb{M}(T)$ has a canonical basis of homogeneous vectors (see Subsection $\ref{modopssec}$) and this basis may be used to identify $\mathbb{M}^\ast(T)\cong \mathbb{M}(T)$.  This produces a splitting of the complex $\Omega(\mathbb{M}^\ast)(\vec{v})$ over the set $M(\vec{v})$.  Indeed, we may write a vector in $\Sigma^{-1}(Det^{-1}(Edges(T)))\tensor \mathbb{M}(T)$ uniquely as a sum of pure tensors which are homogeneous on the right, and then apply the operadic structure map $M(T)\to M(\vec{v})$ to the homogeneous vector in each term.  Those vectors which hit a given $\gamma \in M(\vec{v})$ form a subcomplex, which we denote by $C_\ast(\gamma)$, and these subcomplexes give a direct sum decomposition:
\begin{equation*}
\Omega(\M^\ast)(\vec{v})\cong  \bigoplus_{\gamma\in M(\vec{v})} C_\ast(\gamma).
\end{equation*}

As a graded vector space, $C_\ast(\gamma)$ splits into 1-dimensional subspaces indexed by those elements of $M(T)$ which contract to $\gamma$.  Fixing a mod 2 order (see Definition $\ref{det}$) on the set of edges of $T$ gives a vector in such a subspace.  Under the correspondence given above (Subsection $\ref{treegraphcorr}$), this is the same data as a nesting $\mathfrak{N}$ on the graph $\gamma$ along with a mod 2 order on the set of nests.  The differential is given by ways to add a nest to a given ordered nesting on $\gamma$, with the convention that the new nest goes in the last position.  The degree of such a vector is $-1-|\mathfrak{N}|$.

Applying Corollary $\ref{edgesvsnests}$, we see that $C_\ast(\gamma)$ is concentrated between degrees $-1$ and $-|\gamma|$ (the number of edges of $\gamma$), inclusive. A chain in minimal degree $C_{-|\gamma|}(\gamma)$ is a linear combination of elements, each of which is contained in some $E(T)\subset M(T)$ along with an order on the set of edges of $T$.   Under the identifications
\begin{equation*}
\Sigma^{-1}Det^{-1} (Edges(T)) \tensor E(T) \cong  Det^{-1} (Vert(T)) \tensor  E(T) \cong
(\Sigma^{-1} E)(T),
\end{equation*}
a chain in minimal degree determines an element of $\M_\mathfrak{K}$.  The relations in $\M_\mathfrak{K}$ are precisely the boundaries in $C_{-|\gamma|}(\gamma)$, so to prove the theorem it suffices to prove that $H(C_\ast(\gamma))= \Sigma^{-|\gamma|} k$.

 We will do more, namely we will construct deformation retracts between the chain complexes $C_\ast(\gamma)$ and $\Sigma^{-|\gamma|} k$.  This will be done by induction on the number of edges of $\gamma$.  For the base step of the induction we observe that the complex $C_\ast(\gamma)\cong \Sigma^{-1} k$ if $\gamma$ has one edge. 
So now suppose that $\gamma$ has more than one edge and choose an edge $e$ such that removing $e$ either does not disconnect the graph, or disconnects the graph into two components, one of which is a lone vertex.  (Clearly such an $e$ always exists).  We define $\gamma\setminus e$ to be the graph formed by removing $e$ from $\gamma$ -- in the case that this disconnects the graph, we take it to mean the non-trivial component.  In particular the graph $\gamma\setminus e$ has one less edge then $\gamma$, so for the induction step it suffices to show that the complexes $C_\ast(\gamma)$ and  $\Sigma^{-1} C_\ast(\gamma\setminus e)$ are chain homotopic.

To this end, the remainder of the proof will be dedicated to constructing a deformation retract:
\begin{equation}\label{def}  
\xymatrix{   \Sigma^{-1} C_\ast(\gamma\setminus e) \  \  \ar@/_1pc/@{^{(}->}[rr]_\iota & & \ar@/_1pc/@{->>}[ll]_\pi C_\ast(\gamma) \save !R(.8) \ar@(ur,dr)^H \restore 
} \end{equation}

{\bf Notation:} We use the following notation for dealing with nested graphs.  A nest will be denoted capital $N$.  Let $N_e$ be the nest on $\gamma$ which contains only the fixed edge $e$.  Let $N_{\text{max}}$ be the nest on $\gamma$ which contains all edges except $e$.

A chain, or equivalently a mod 2 ordered nesting, will often be denoted simply by $\mathfrak{N}$, leaving the mod 2 order implicit.  If needed, we use a subscript $\mathfrak{N}_\gamma$ to denote the graph on which the nesting is given.  To a nesting $\mathfrak{N}_{\gamma \setminus e}$ there is a nesting $\mathfrak{N}_\gamma$, given by taking just the same nests.  There is also a nesting $\mathfrak{N}_\gamma\cup N_{max}$ with the induced order (so $N_{max}$ placed in the last position).

 {\bf Intuition:}  The idea behind the construction of the deformation retract in Formula $\ref{def}$ will be to realize $C_\ast(\gamma)$ as a ``cylinder'' whose top is given by those nestings containing $N_e$ and whose bottom is given  by those nestings containing $N_{\text{max}}$.  Given the technical nature of the construction, the reader is encouraged to work through the examples pictured in Figure $\ref{cyclo}$ concurrent to reading the general construction, which we now give.

{\bf Definition of $\iota$}:  To a nesting of the graph $\gamma\setminus e$ we get a nesting of $\gamma$ by keeping all the nests we started with and adding the nest containing all the edges of $\gamma$ except $e$ in the last position.  In the notation above:
\begin{equation*}
\iota(\mathfrak{N}_{\gamma \setminus e}) := \mathfrak{N}_\gamma\cup N_{\text{max}}
\end{equation*}

Observe that because $\iota$ adds one nest to a graph with one additional edge, it is a degree $0$ map. To see that $\iota$ commutes with the differential observe if I first apply $\iota$, then sum over all ways to add a nest, none of these nests can contain the edge $e$, because $e$ is already outside a nest of maximal size.  Each of these nests then could be added to the nesting on $\gamma\setminus e$ before applying $\iota$, which is in turn the definition of $d$ before $\iota$.

{\bf Definition of $\pi$}:  Informal definition:  remove the fixed edge $e$ from all nests.

More precisely: starting from a mod 2 ordered nesting $\mathfrak{N}=\{N_1\cdc N_r\}$ we consider the graphs $N_i-e$.  If $e\notin N_i$ then $N_i-e:=N_i$.  Else, $N_i-e$ is defined to be the graph formed by removing the two flags of $e$ from the subgraph $N_i$.  We then use the list of graphs $\{N_i-e\}$ to form a nesting by the following procedure.  If $N_\text{max}-e=N_\text{max}$ or $N_e-e=\emptyset$ appears on the list discard it.  If $N_i-e=N_j-e$ (which happens if $N_i\cup e= N_j$ or vice versa), identify these (identical) sets in the list.  Finally it may be the case that some $N_i-e$ is disconnected in to two components, each with a non-zero number of edges.  In this case, we split this entry into its two connected components which are added to the list.  If one or both of these components already appears in the list it is discarded/identified.  Modifying the list $\{N_i-e\}$ according to these rules yields a nesting on $\mathfrak{N}$.  If this nesting has $r-1$ nests we define it to be $\pi(\mathfrak{N})$ (up to sign/orderings which are fixed below).  Otherwise we define $\pi(\mathfrak{N})=0$.  Observe that since $\pi$ always has one fewer nest in the target than in the source, it has degree $0$. 

To fix signs/orders we first characterize the nestings $\mathfrak{N}$ for which $\pi(\mathfrak{N})\neq 0$.  For a fixed $\mathfrak{N}$ the following are mutually exclusive:
\begin{enumerate}
	\item $N_\text{max}\in \mathfrak{N}$,
	\item $N_e\in \mathfrak{N}$,
	\item $N_i= N_j\cup e$ for some $i,j$,
	\item $N_i = N_j\cup e \cup N_l$ for some $i,j,l$ with $N_j\cap N_j=\emptyset$.
\end{enumerate}  
Moreover if (3) or (4) happens it does so for unique indicies.
Then $\pi(\mathfrak{N})\neq 0$ if and only if both one of (1-4) happens and each (other) time $N-e$ is disconnected then exactly one of the two components already appears as a nest (``other'' means excluding (4) above, which is the case where both components are nests).%

We now assume $\pi(\mathfrak{N})\neq 0$ and fix the mod 2 ordering on $\pi(\mathfrak{N})$ as follows.  We first permute the elements of $\mathfrak{N}$ so that in case (1) $N_{max}$ is in the last position, in case (2) $N_{e}$ is in the penultimate position, in case (3) $N_i$ is in the penultimate position, in case (4) $N_i$ is in the penultimate position. Then $\pi(\mathfrak{N})$ carries the order induced by removing the last entry in case (1) and the penultimate entry in cases (2,3,4).  Note that if there exists some other nest $N\in\mathfrak{N}$ with $N-e$ disconnected, then (as above) exactly one of these components already appears on the list. The induced order on the terms in $\pi(\mathfrak{N})$ is given by replacing $N$ with the new component, while keeping the redundant component in its original position. 
Finally, we note that in the special case that $\mathfrak{N}=\{ N_e \} $, there is no penultimate position, but we simply declare $\pi(\mathfrak{N})$ to be the negative of the empty nesting.
Observe that with these sign conventions, $\pi\circ \iota$ is the identity.

{\bf Proof that $\pi$ is a chain map}: Consider the cases $1, 2a, 2b, 3a, 3b, 4a, 4b$ which correspond to (1-4) above, but where ``$a$'' means $\pi(\mathfrak{N})\neq 0$ (so every time $N-e$ is disconnected (if any), at least one of the components was already a nest) and case ``$b$'' means ``not $a$'', and so in cases $2b,3b,4b$, $\pi(\mathfrak{N})=0$.  Observe that in case (1), $e$ is contained in no nest so there is only one case here.  Since the terms in $d(\mathfrak{N})$ add a nest, if (1-4) applies to $\mathfrak{N}$ then it also applies to every term in $d(\mathfrak{N})$.  We use this fact to show that $\pi d=d\pi$ by looking at the possible cases.

{\bf Case 1:} If $\mathfrak{N}$ contains $N_{\text{max}}$, then applying $d$ before $\pi$ we add nests, but they can't contain $e$ without crossing $N_{\text{max}}$.  These are the same nests we find if we apply $\pi$ before $d$.

{\bf Case 2a:}   The terms in $d(\mathfrak{N})$ land in case $(2a)$ or case $(2b)$.  Since $\pi$ of the terms in $(2b)$ is zero, it suffices to show that $\pi$ of the terms of type $(2a)$ is $d\pi(\mathfrak{N})$. The terms in $d(\mathfrak{N})$ of type $(2a)$ are formed by adding a nest which either contains $e$ or doesn't contain $e$.  Those terms which don't contain $e$ correspond to adding nests in $\gamma\setminus e$, and so correspond to terms in $d\pi(\mathfrak{N})$.  On the other hand, consider those terms in $d(\mathfrak{N})$ of type $(2a)$ formed by adding a nest $N$ which does contain $e$.  If $N-e$ is connected, it specifies a nest on $\gamma\setminus e$ which is not an element of $\pi(\mathfrak{N})$.  If $N-e$ is disconnected, then the type (2a) assumption implies exactly one of the components is not an element of $\pi(\mathfrak{N})$.  In either case adding this new nest gives the corresponding term in $d\pi(\mathfrak{N})$.

{\bf Case 2b:}  The terms in $d(\mathfrak{N})$ land in case $(2b)$ unless the added nest is the union of $e$ with one of the components of a nest in $\mathfrak{N}$ which was disconnected by removing $e$, (then we land in case (2a)).  These terms come in pairs, one for each such component, which cancel after applying $\pi$ (see Figure $\ref{pi1}$.)  To see that they carry opposite sign note that the differential places the nest corresponding to each component in the last position, while $\pi$ of each of these replaces the disconnected nest with the other component.  To relate these two terms requires transposing the last term with the position of the disconnected term.  On the other hand $d\pi(\mathfrak{N})=0$, whence this case.
	
	\begin{figure}
		\includegraphics[scale=1.0]{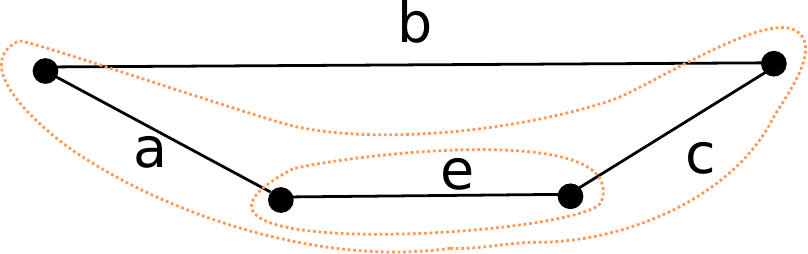} 
		\caption{A minimal example of case $2b$.  The graph carries nesting $\mathfrak{N}= \{ ace,e \}$.  Here $\pi(\mathfrak{N})=0$ and the terms of $ d(\mathfrak{N})$ which don't vanish via $\pi$ are $\{ ace,e,ae \}$ and $\{ ace,e,ce \}$.  Applying $\pi$ to the former gives $\{c,a\}$ and $\pi$ of the latter is $\{a,c\}$, hence $\pi d(\mathfrak{N})=0$, and in particular $\pi d(\mathfrak{N})=d\pi (\mathfrak{N})$ in this case.}
		\label{pi1}
	\end{figure}
	
{\bf Cases 3,4}: This works the same as case $(2)$, except the role of $N_e$ is played by the smallest nest containing $e$ (denoted $N_i$ above).  In particular, $\mathfrak{N}$ being of case 3 or 4 is a property closed under $d$, and added nests in $d(\mathfrak{N})$ can't contain $e$ without also containing all of $N_i$.  

{\bf Case 5}:   Suppose none of the above cases occur. Then $d\pi(\mathfrak{N})=0$ and we will show $\pi d(\mathfrak{N})=0$ as well.  The only possible terms in $d(\mathfrak{N})$ which are not annihilated by $\pi$ are those which are of type (1-4) above.  Case (1) is possible only if no nest of $\mathfrak{N}$ contains $e$.  In this case the terms in $d(\mathfrak{N})$ which don't vanish under $\pi$ correspond to adding $N_{max}$ and adding the smallest valid nest containing $e$. (The fact that there is a smallest valid nest containing $e$ is ensured since $\mathfrak{N}$ is not of type $(1)$.)  The term which adds and removes $N_{max}$ will occur with a $+$ sign, while the term which adds the smallest valid nest containing $e$ is of type $(2)$,$(3)$ or $(4)$, and so must be put in the penultimate position to apply $\pi$, thus occurring with opposite sign, or occurs with opposite sign by convention if $\mathfrak{N}$ was the empty nesting.

We now suppose there is a nest in $\mathfrak{N}$ which contains $e$, and we let $N_i$ be the smallest such nest.  The terms in $d(\mathfrak{N})$ which are not immediately seen to vanish under $\pi$ are those which add the smallest valid nest containing $e$ and which adds a connected component of the layer graph of $N_i-e$.  Recall here that the layer graph inside $N_i$ is the graph formed by collapsing its subnests to vertices.  The fact that $N_i$ is the smallest nest containing $e$ means that $e$ is an edge in this layer graph.  If the layer graph of $N_i$ is disconnected in to two nontrivial components when removing $e$, then adding the smallest valid nest containing $e$ lands in case (2b),(3b) or (4b) since $N_i-e$ is disconnected with neither component appearing in $d(\mathfrak{N})$.  On the other hand a term in $d(\mathfrak{N})$ which adds a component of the layer graph will be of type (5), since there are two nontrivial connected components.  So in this case $\pi$ of each term in the differential vanishes, and in particular $\pi d(\mathfrak{N})=0$.

Let us now consider the case that the layer graph of $N_i$ is not disconnected by removing $e$. Then there are two terms in $d(\mathfrak{N})$ which are sent to something non-zero by $\pi$, they corresponds to adding the smallest sub-nest in $N_i$ containing $e$ and adding the nest corresponding to the layer graph of $N_i-e$.  The fact that $\mathfrak{N}$ is not of type $(2-4)$ ensures these are valid new nests.  Taking $\pi$ of these two differential terms removes the just added nests. These two nestings have the same list of nests and it remains to verify that they have opposite sign.  Assume without loss of generality that $N_i$ was in the final position of $\mathfrak{N}$.  In the former term the sign is determined by transposing the new nest into the penultimate position before applying $\pi$.  In the latter term the sign is determined by applying $\pi$ without any permutation, hence these terms have opposite signs, and so $\pi d(\mathfrak{N})=0$ in this case as well.  See Figure $\ref{pi2}$ for an example.

	\begin{figure}
		\includegraphics[scale=1.0]{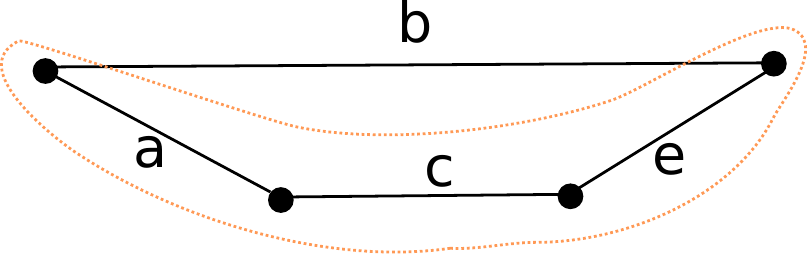} 
		\caption{An example of case (5).  The graph carries nesting $\mathfrak{N}= \{ ace \}$.  In this case $\pi(\mathfrak{N})=0$ and the terms of $ d(\mathfrak{N})$ which don't vanish via $\pi$ are $\{ ace,e \}$ and $\{ ace,ac \}$.  Applying $\pi$ to the former term gives $-\{ac\}$ and $\pi$ of the latter is $\{ac\}$, hence $\pi d(\mathfrak{N})=0$, and in particular $\pi d(\mathfrak{N})=d\pi (\mathfrak{N})$ in this case.
		} \label{pi2}
	\end{figure}

This completes our verification that $d\pi=\pi d$.

{\bf Definition of H:}  Informal description:  remove $e$ one nest at a time.  

Precise definition: we define the chain homotopy $H\colon C_\ast(\gamma)\to C_{\ast-1}(\gamma)$ as follows. Fix a nesting of $\gamma$, denoted $\mathfrak{N}=\{N_1\cdc N_r\}$.   If $e$ is not contained in any nest of $\mathfrak{N}$ then $H(\mathfrak{N})=0$.  Else we assume that $N_r<N_{r-1}<... < N_{r-q}$ are the nests containing $e$ (after Lemma $\ref{eelem}$), and in particular $N_r$ is the smallest nest containing $e$.  In this case we define $H$ in two steps.  The first step is to remove $e$ from $N_r$.  We may view this is a nest (or a pair of nests in the disconnected case) on $\gamma\setminus e$ (as above) and hence a nesting on $\gamma$.  We say $H(\mathfrak{N})=0$ unless this reduces the number of nests by 1, in which case we call this new nesting the leading term of $H(\mathfrak{N})$.  It has nests $N_{r-1}<...< N_{r-q}$ which contain $e$.  We then proceed to remove $e$ from each of these nests one at a time in the nested order to form a new nesting at each stage.  We continue until we reach the end, or until we get a nesting of the wrong degree (which can only happen due to disconnectivity since $N_i\setminus e$ and $N_{i-1}$ are separated by at least two edges).  $H(\mathfrak{N})$ is defined to be the sum of these nestings.  In particular $H(\mathfrak{N})$ is a sum of at most $q+1$ nestings.

Regarding signs:  $H(\mathfrak{N})$ is nonzero only if removing $e$ from the smallest nest containing it reduces the number of nests by $1$.  This happens if and only if one of the mutually exclusive cases (2),(3),(4) above occur.  In these cases we follow the opposite reordering conventions as we did above when we defined $\pi$ -- namely we assume the smallest nest containing $e$ is in the final position before removing/identifying it.

{\bf Proof that $H$ is a chain homotopy}.
It remains to verify that $dH+Hd = id-\iota\pi$.  We will do this by analyzing the possible cases for $\mathfrak{N}$ as considered above.

{\bf Case 1:} If $N_{\text{max}}\in \mathfrak{N}$ then $(id-\iota\pi)(\mathfrak{N})=0$, while $e$ is contained in no nest of $\mathfrak{N}$ and no nest of $d(\mathfrak{N})$.  Hence $dH+Hd(\mathfrak{N})=0$ as well.

{\bf Case 2a:}	This means $N_e=N_r$, and each time $N_i-e$ is disconnected, exactly one if its components appears in $\mathfrak{N}$.  Terms in $Hd$ are given by first adding a nest, call it $N_d$, and then applying $H$ which removes $e$ one nest at a time.  Under case (2a), these terms appear in $dH$ by adding either $N_d$ or (a connected component of) $N_d-e$ to each nesting appearing in $H(\mathfrak{N})$.  Here there is the possibility that $N_d-e$ is disconnected in to two components, neither of which belong to the given nesting, but $H$ applied to such a term will vanish.
Thus, the terms of $Hd$ $\subset$ the terms in $dH$ (up to sign) in this case.  
	
The signs in the above correspondence are opposite because $dH$ removes the smallest nest containing $e$ in the last position, while $Hd$ does the same thing from the penultimate position.  In particular $dH+Hd$ is the sum of terms in $dH$ which do not appear in $Hd$.  Let us describe these.	There are (at most) two terms for each such $N_i$ containing $e$.  The $i^{th}$ term (reindexing to count right to left) in $H(\mathfrak{N})$ has nests $N_i-e$ and $N_{i+1}$.  Applying $d$, the two distinguished terms are adding $N_i$ or adding $N_{i+1}- e$.   These terms cancel in pairs, except for the first and last ones which are $N_\text{max}$ and $N_e$ respectively (interpreting the last stage as adding $N_\text{max}$).  Adding back $N_e$ gives $id$ and adding in $N_\text{max}$ after having removed all the $e$ terms gives $-\iota\pi$.  The minus sign occurs because $\pi$ must transpose $N_e$ into the penultimate position.  See Figure $\ref{pi5}$ for an example of this.

\begin{figure}
	\includegraphics[scale=1.0]{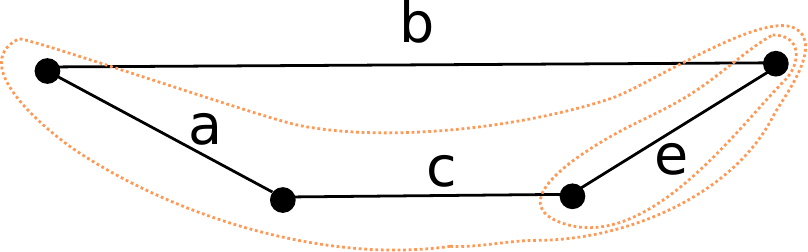} 
	\caption{ The graph carries nesting $\mathfrak{N}= \{ ace,e \}$.  We compute 
	$dH(\mathfrak{N}) = d(\{ ace \}+\{ac\}) = \{ ace,a \} + \{ ace,c \} + \{ ace,e \}+ \{ ace,ac \}+ \{ ace,ce \} + \{ ac,a \} + \{ ac,c \} + \{ ac,abc \}+ \{ ac,ace \}$, 
	and we may cancel $\{ ace,ac \}$ with $\{ ac,ace \}$.  Then we compute
	$Hd(\mathfrak{N}) = H(\{ ace,e,a \}+\{ ace,e,ce \}) = -\{ace,a\}-\{ac,a\} - \{ace,ce\}-\{ace,c\}-\{ac,c\}
	$
	and 
	$\iota\pi(\mathfrak{N}) =\iota(-\{ac\})=-\{ac,abc\} $, from which we find $dH+Hd=id-\iota\pi$.}
	\label{pi5}
\end{figure}

{\bf Case 2b:}
Now suppose that $\mathfrak{N}$ is of type $(2b)$, and so $\pi(\mathfrak{N})=0$.  This happens due to disconnectivity of some $N_{r-j}-e$ such that neither component appears for some $j\geq 1$.  In this case, the terms of $Hd$ which do not appear in $dH$ correspond to adding one or the other component unioned with $e$ via $d$.  Applying $H$ identifies these terms with opposite sign, corresponding to permuting the components.  Thus $dH+Hd$ may again be described as the terms in $dH$ which do not appear in $Hd$.  $H$ consists of $j$ nestings and there are $2j-1$ terms appearing in $dH$ but not in $Hd$.  They correspond to adding nests $N_{r-i+1}$ and $N_{r-i}-e$ to the $i^{th}$ term of $H(\mathfrak{N})$ for $1\leq i \leq j-1$, and $N_{r-j+1}$ to the $j^{th}$ term.  These terms cancel in pairs, except for the first term which was the identity.

{\bf Cases 3,4:}	These cases again follow similarly to case 2, except the role of $N_e$ is played by the smallest nest containing $e$.  See Figure $\ref{pi4}$ for an example of this.

\begin{figure}
	\includegraphics[scale=1.0]{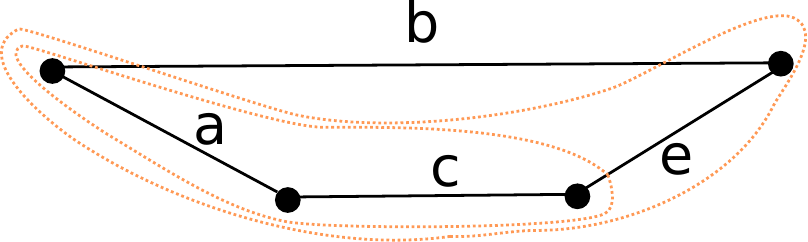} 
	\caption{ The graph carries nesting $\mathfrak{N}= \{ac, ace \}$.  We compute $dH(\mathfrak{N}) = d(\{ ac \}) = \{ ac,a \} + \{ ac,c \} + \{ ac,abc \}+ \{ ac,ace \}$ and $Hd(\mathfrak{N}) = H(\{ ac,ace,a \}+\{ac,ace,c\}) = -\{ ac,a \}-\{ ac,c \}$ and $\iota\pi(\mathfrak{N}) =\iota(-\{ac\})=-\{ac,abc\} $, from which we find $dH+Hd=id-\iota\pi$.}
	\label{pi4}
\end{figure}
	
{\bf Case 5}:  We now suppose cases (1-4) do not occur.
If $e$ is not contained in any nest, then $\pi(\mathfrak{N})=0$ and $H(\mathfrak{N})=0$, so it remains to analyze $Hd(\mathfrak{N})$.  The only non zero terms in $Hd(\mathfrak{N})$ are given by adding $e$ and then taking away $e$.  That is, there is a unique smallest nest which adds $e$ appearing in the differential, and $H$ of this term gives the identity back (with correct sign by convention). The other terms in the differential are annihilated by $H$, and so $ Hd(\mathfrak{N})= \mathfrak{N}$, which in this case implies $Hd+dH=id-\iota\pi$.	
	
So now suppose $e$ is contained in some $N_r<...<N_{r-q}$.  Having excluded cases (2-4), we must have one of the following mutually exclusive situations:
\begin{enumerate}
	\item[i.] $N_r-e$ is a valid new nest.  In otherwords replacing $N_r$ with $N_r-e$ in $\mathfrak{N}$ yields a valid nesting with the same number of nests.
	\item[ii.] $N_r-e$ is disconnected and one component is a nest in $\mathfrak{N}$.
	\item[iii.] $N_r-e$ is disconnected and neither component is a nest in $\mathfrak{N}$.
\end{enumerate}

Observe that the case that $N_r-e$ is disconnected and both nests appear is excluded by our assumption that $H(\mathfrak{N})=0$.  Notice that in each of these cases $\pi(\mathfrak{N})=0$ (we have already excluded 1,2 and the above clearly imply not 3,4).  So it remains to show that $Hd(\mathfrak{N})=\mathfrak{N}$.

The terms in $d(\mathfrak{N})$ may be subdivided into those which are annihilated by $H$ and those which are not.  Differential terms which add a nest outside of $N_r$ are annihilated, as are those which result in $e$ appearing in a layer of size 2 or greater.  The terms which are not annihilated by $H$ are, in the three cases i) add $N_e$ and the complement of $e$ in $N_r$, ii) add the union of $e$ with the appearing component and the complement of the appearing component in $N_r$, iii) add $N_e$.

Let us first analyze case i. The leading term of $H(\mathfrak{N}\cup N_e)$ is $\mathfrak{N}$.  The terms $H(\mathfrak{N}\cup N_r-e)$ are precisely the non-leading terms of $H(\mathfrak{N}\cup N_e)$.  In particular these replace (resp.\ identify) $N_r$ with $N_r-e$ (and subsequently for all nests containing $e$).  In the former $N_r$ appeared in the last position and in the latter it appeared in the penultimate position, and thus occur with opposite sign.  So in this case $Hd(\mathfrak{N})=\mathfrak{N}$.  See Figure $\ref{pi3}$ for an example.

	\begin{figure}
		\includegraphics[scale=1.0]{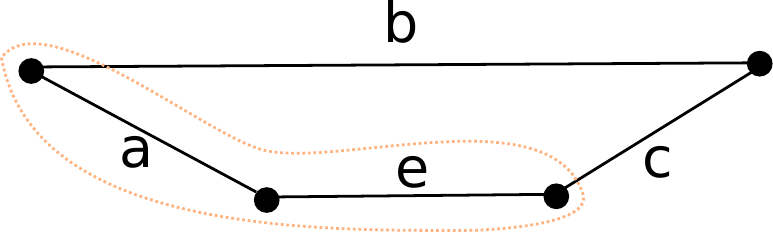} 
		\caption{ The graph carries nesting $\mathfrak{N}= \{ ae \}$.  Then $H(\mathfrak{N})=0$ for degree reasons according to subcase i, since $N_r-e = a$ is a valid nest.  In this case $\pi(\mathfrak{N})=0$ as well and the terms of $d(\mathfrak{N})$ which don't vanish under $H$ are $\{ea,a\}$ and $\{ea,e\}$.  Applying $H$ to the former yields $-\{a\}$.  Applying $H$ to the latter yields $\{ea\}+\{a\}$.  Thus $Hd(\mathfrak{N})=\mathfrak{N}$ in this case.
		}\label{pi3}
	\end{figure}

Case ii follows similarly, replacing $N_e$ with the smallest nest containing it and replacing $N_r-e$ with the complement of the appearing component.

In case iii, the only term which survives is $H(\mathfrak{N}\cup N_e)$.  This has leading term $\mathfrak{N}$.  Putative subsequent terms would remove $e$ from $N_r$ and hence would be disconnected with neither component appearing and so be of wrong degree.  Hence the leading term is the only term in this case and $Hd(\mathfrak{N})=\mathfrak{N}$.

Thus in all cases $dH+Hd=id-\iota\pi$, completing the proof of the theorem. \end{proof}

\begin{corollary}
	The natural map $\Omega((\mathbb{M}_\mathfrak{K})^\ast)\to \mathbb{M}$ is a quasi isomorphism.
\end{corollary}

The following immediate corollary of Theorem $\ref{koszulthm}$ comes by restricting attention to only those graphs of genus $0$:

\begin{corollary}\label{cycopscor}  The groupoid colored operad encoding (cyclic) operads is Koszul.
\end{corollary}

This corollary gives us a definition of weak operads and cyclic operads and allows us to implement homotopy transfer theory for them.  This result should be of independent interest.

\subsection{Relationship to Graph Associahedra}

Our proof of Theorem $\ref{koszulthm}$ was guided by the fact that the poset of nestings of a graph is the face poset of a convex polytope.  In this section we provide a proof of this fact.  The results of this subsection have been informed by many conversations with E. J{\"o}nson and S. Karlsson  during their participation in the 2018 RAYS Summer Research Symposium.  We refer to \cite{EJ} for a preliminary write up of some related results.

In this section we let $\gamma$ denote a graph with no legs and genus labels.  Equivalently we could consider $\gamma$ to be a modular graph whose legs and genus labels are disregarded.  A standard construction  (see eg \cite[Chapter 8]{Graphtheory}) associates a graph $L(\gamma)$ (the line graph of $\gamma$) to $\gamma$ as follows.  The vertices of $L(\gamma)$ are by definition the edges of $\gamma$.  Two such vertices in $L(\gamma)$ are joined by an edge if their corresponding edges in $\gamma$ shared a vertex.  We remark that while $\gamma$ may have loops and parallel edges, the line graph $L(\gamma)$ will not.

A graph associahedron \cite{CDev} is a convex polytope associated to a graph via a poset structure on certain subsets of the power set of the vertices of the graph.  We let $K_\gamma$ denote the graph associahedron associated to $\gamma$.
	
\begin{lemma}\label{polyhedra}  The poset of nestings on a graph $\gamma$ is isomorphic to the face poset of the graph associahedron $K_{L(\gamma)}$.
\end{lemma}
\begin{proof}
	
	A nest $N$ on $\gamma$ (Definition $\ref{nestdef}$) is a proper subset of the edges of $\gamma$ and hence a proper subset of the vertices of $L(\gamma)$.  The fact that the closure of $N$ is connected, implies that the induced graph (the vertices $N$ and all edges between them) is a connected subgraph of $L(\gamma)$.  Thus a nest $N$ on $\gamma$ specifies a ``tube'' on $L(\gamma)$ in the parlance of \cite{CDev}; by definition a proper subset of the vertices such that the induced graph is connected.  Given such a nest $N$ on $\gamma$ we denote its associated tube on $L(\gamma)$ by $\hat{N}$.
	
	Recall (Subsection $\ref{nestsec}$) that two nests on $\gamma$ are compatible if one is a subset of the other or if their closures are disjoint.  We consider the implication on the associated tubes.  Let $N_1$ and $N_2$ be two nests on $\gamma$.  Clearly $N_1\subset N_2$ if and only if $\hat{N_1}\subset \hat{N_2}$.  Likewise $N_1$ and $N_2$ are disjoint if and only if $\hat{N}_1$ and $\hat{N}_2$ are disjoint.  If $N_1$ and $N_2$ are disjoint but their closures are not, then their closures share a vertex.  Thus the union of the tubes $\hat{N_1}$ and $\hat{N_2}$ has an induced graph which is connected.  Thus $\hat{N_1}\cup \hat{N_2}$  is itself a tube or the entire vertex set.  Conversely if $\hat{N_1}\cup \hat{N_2}$  is itself a tube or the entire vertex set, then the closures of $N_1$ and $N_2$ share a vertex.
	
	By direct comparison we conclude that two nests on $\gamma$ are compatible if and only if their associated tubes are compatible on $L(\gamma)$ in the sense of \cite[Definition 2.2]{CDev}.  It follows that a nesting on $\gamma$ is equivalent to a tubing on $L(\gamma)$.  Since the poset structures of both nestings and tubings are by inclusion, the claim follows.  (When comparing the definitions of nestings and tubings, one should note that \cite[Definition 2.2 (3)]{CDev} implicitly allows the union to be the entire vertex set.  This point is clarified in \cite{DF} (see the Remark after Defintion 1)).
	\end{proof}

\begin{figure}
	\includegraphics[scale=.6]{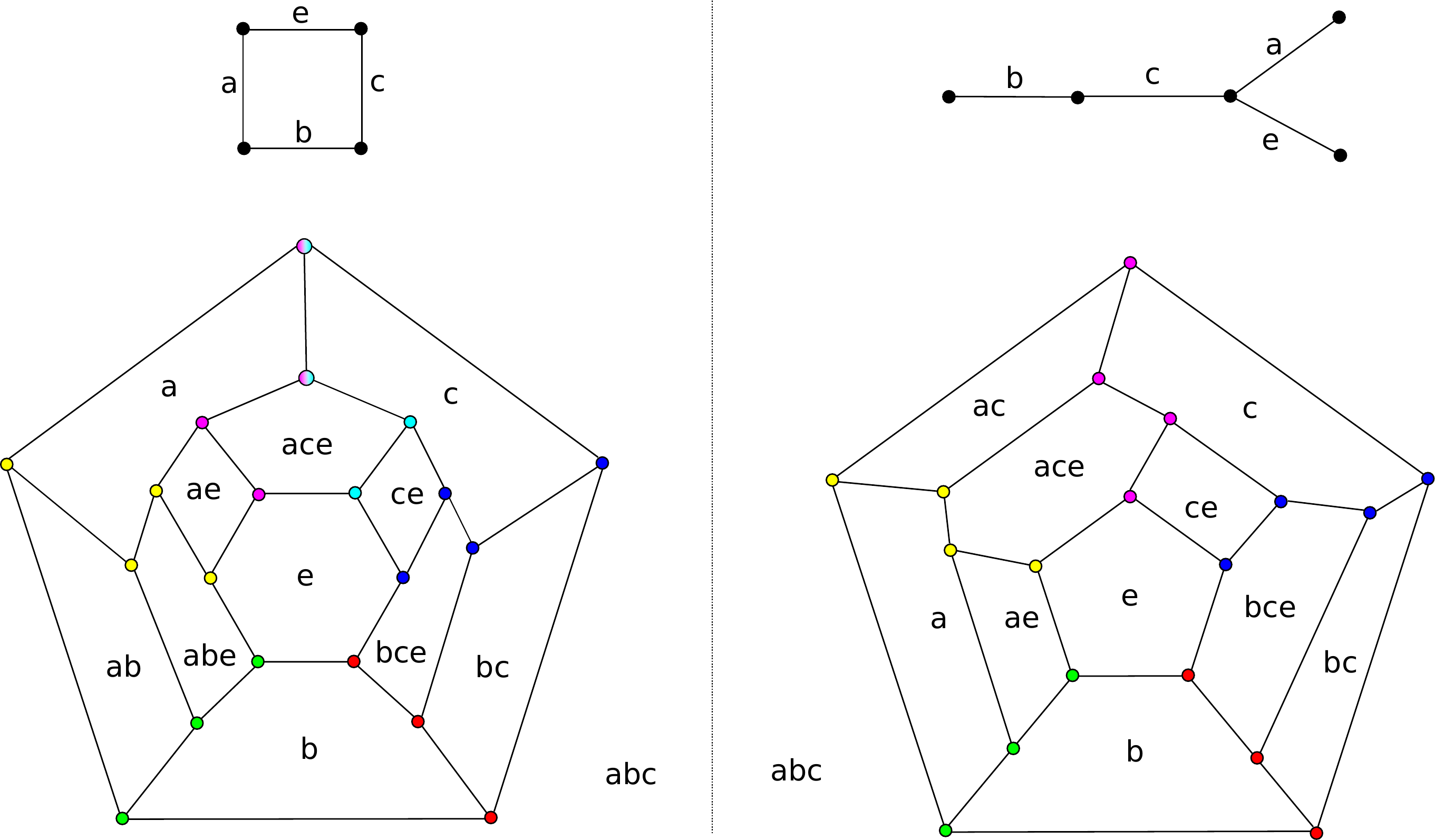} 
	\caption{For $\gamma$, either of the $4$ edged graphs above,
		the poset of nestings is the face poset of a polyhedron.
		Here we depict the polyhedra in top down/annulus view -- the three cells are not depicted.  Only the codimension 1 cells are labeled.  The colors indicate the flow of a contraction in which the inside/top cells (labeled by $e$) are contracted to the outside/bottom solid pentagons labeled by $abc$. This contraction encodes the maps $H,\iota, \pi$ constructed in the proof of Theorem $\ref{koszulthm}$, up to signs and degrees.} 
	\label{cyclo}
\end{figure}

Graph associahedra have long appeared in the study of algebras over operads both classically  \cite{Stasheff} and in the Renaissance era \cite{MarklSimp}. The appearance of graph associahedra in the context of homotopy operads dates at least to \cite{Brinkmeier}.  More recent related results may also be found in \cite{MarklPerm}, \cite{Ob}.  As a consequence of Lemma $\ref{polyhedra}$ we see that the familiar families of associahedra, cyclohedra and permutohedra arise as the fiber for graphs which are respectively rows of vertices, polygons, and bouquets of loops meeting at a common vertex.  See Figure $\ref{cyclo}$. This connection to the well-studied line graph construction raises myriad interesting questions.  For example, the question of when two graphs have the same fiber may be partially addressed via \cite[Theorem 1]{Whitney}.

\begin{remark}  Adding directed structures to the graph would not change the poset of nestings.  Thus, the above arguments could be adapted to give a proof of Koszulity for generalizations of operads based on directed graphs.  Results along these lines have recently been obtained in the preprint \cite{MB2}. 
\end{remark}

\begin{remark}  It would be possible to turn Lemma $\ref{polyhedra}$ in to a proof of Theorem $\ref{koszulthm}$ by identifying the fibers of $\Omega((\mathbb{M}_\mathfrak{K})^\ast)\to \mathbb{M}$ with the cellular chain complexes on the given polytopes.  However, this would require significant book keeping of sign/orientation data at each step, so we chose to instead prove $\Omega(\mathbb{M}^\ast)\to \mathbb{M}_\mathfrak{K}$ is a quasi-isomorphism, which has the advantage that signs may be encoded combinatorially via nest orders.
\end{remark}

\subsection{Weak Modular Operads} \label{wmosec}

\begin{definition}\label{wmodef}  A {\bf weak modular operad} is an algebra over $\M_\infty:=\Omega((\M_\mathfrak{K})^\ast)$.	
\end{definition}


The remainder of this section will be devoted to carefully unwrapping this definition. 

\subsubsection{Generating operations}

Let $A$ be a $\mathbb{V}$-module.  We may extend $A(g,-)$ to be valued in any finite set via left Kan extension (see \cite[Section 1.1]{GK}).  Explicitly for a set $X$ with $n$ elements we define
\begin{equation}\label{LKan}
A(g,X):= \left[ \bigoplus_{ \text{Bij}(S_n,X)} A(g,n)\right]_{S_n}  
\end{equation}
where Bij means bijections.  

If $\gamma$ is a modular graph with a vertex $v$ of genus label $g$, we define $A(v):= A(g, \text{fl}(v))$.  We then define $A(\gamma):=\tensor_{v\in \text{Vert}(\gamma) } A(v)$.  In addition, for a fixed vertex $v_0\in\gamma$, we also define $A(\gamma\setminus v_0) :=\tensor_{v\in \text{Vert}(\gamma)\setminus \{v_0\} } A(v)$.

Now suppose $A$ is a weak modular operad.  Then we have a morphism of operads
$$\Omega((\mathbb{M}_\mathfrak{K})^\ast) \to End_A$$
which is induced by a morphism of $\mathbb{V}$-colored sequences $\Sigma^{-1}(\mathbb{M}_\mathfrak{K})^\ast \to End_A$.  By definition, such a morphism is a linear map 
\begin{equation}\label{mk}
\Sigma^{-1}(\mathbb{M}_\mathfrak{K}(\vec{v}))^\ast\tensor A(in(\vec{v})) \to A(g,n)
\end{equation}
for each color scheme $\vec{v}=((g_1,n_1)\cdc (g_m,n_m);(g,n))$  which is $S_{m}$ and $\times_{i\geq 1} Aut(v_i)$ invariant.

A homogeneous element in the source of Formula $\ref{mk}$ specifies a modular graph $\gamma$ along with a mod 2 order on its set of edges, whose vertices are ordered and labeled by elements of $A$.  Taking invariants of the $S_{m}$ action has the effect of simply discarding the vertex order of the graph.  Taking invariants with respect to the $\times_{i\geq 1} Aut(v_i)$ action ensures that permuting labels of flags which are then glued together to form edges has no effect.  Define $Aut(\gamma)$ to be the group of flag permutations which preserve leaf labels and vertex adjacencies of $\gamma$.  The linear map in Formula $\ref{mk}$ lifts to a map from
\begin{equation}
[\Sigma^{-1}(\mathbb{M}_\mathfrak{K}(\vec{v}))^\ast\tensor_{in(\vec{v})} A(in(\vec{v}))]_{S_{m}}\cong  \bigoplus_{\substack{\text{modular graphs }\gamma \\ \text{w/ vert.\ of type }in(\vec{v})}}(\Sigma^{-1}\mathfrak{K}(\gamma)^\ast \tensor A(\gamma))_{Aut(\gamma)}
\end{equation}
where we define
\begin{equation}\label{k}
\mathfrak{K}(\gamma):= Det^{-1}(Edges(\gamma)).
\end{equation}

Thus, a weak modular operad specifies an $Aut(\gamma)$-invariant, degree $0$ map 
\begin{equation*} \Sigma^{-1}\mathfrak{K}(\gamma)^\ast \tensor A(\gamma)\stackrel{\mu_\gamma}\to A(g,n)
\end{equation*}
for each modular graph $\gamma$. We call these {\bf generating operations} for the weak modular operad structure on $A$.  We remark that for any finite set $X$, such generating operations may be extended ala Formula $\ref{LKan}$ to maps $\Sigma^{-1}\mathfrak{K}(\gamma)^\ast \tensor A(\gamma)\stackrel{\mu_\gamma}\to A(g,X)$
for each modular $X$-graph $\gamma$ (Definition \ref{modulargraphdef}).

\subsubsection{Composition of operations}  Let $v$ be a vertex of genus $g$ in a modular graph $\gamma$ and let $\gamma^\prime$ be a modular fl$(v)$-graph of total genus $g$.  We define the {\bf vertex substitution} $\gamma\circ_v\gamma^\prime$ to be the modular graph formed by substituting $\gamma^\prime$ in to vertex $v$ as in Subsection $\ref{treesub}$.  In particular the flags at $v$ are identified with the legs of $\gamma^\prime$ using the leg labeling, and the vertices of $\gamma\circ_v\gamma^\prime$ are (Vert$(\gamma)\setminus v)\sqcup$Vert$(\gamma^\prime)$, with the inherited genus labeling.

There is a canonical isomorphism $Edges(\gamma)\sqcup Edges(\gamma^\prime)\cong Edges(\gamma\circ_v\gamma^\prime)$.  Thus the wedge product $\mathfrak{K}(\gamma)\tensor\mathfrak{K}(\gamma^\prime)\longrightarrow \mathfrak{K}(\gamma\circ_v\gamma^\prime)$ is an isomorphism. 
 This induces a degree $-1$ isomorphism via the composite:
 \begin{equation}\label{Ksigns}
 \Sigma\mathfrak{K}(\gamma)\tensor\Sigma\mathfrak{K}(\gamma^\prime)\stackrel{(-1)^{|\gamma|}}\longrightarrow \Sigma^2\mathfrak{K}(\gamma)\tensor\mathfrak{K}(\gamma^\prime)\stackrel{\Sigma^2 -\wedge -}\longrightarrow \Sigma^{2}\mathfrak{K}(\gamma\circ_v\gamma^\prime)\cong \Sigma\mathfrak{K}(\gamma\circ_v\gamma^\prime).
 \end{equation}
We denote this composite by $\epsilon=\epsilon(\gamma,v,\gamma^\prime)$.

Generating operations $\mu_\gamma$ for a weak modular structure on $A$ may be
composed in the following sense. Let $\gamma$ be a modular graph of type $(g,n)$.  If $v$ is a vertex of $\gamma$ of genus $h$ and $\gamma^\prime$ is a graph of type $(h,\text{fl}(v))$, we define $\mu_\gamma\circ_v\mu_{\gamma^\prime}$ to be the following composition:
\begin{align}\label{comps}
\Sigma^{-1}\mathfrak{K}(\gamma\circ_v\gamma^\prime)^\ast \tensor A(\gamma\circ_v\gamma^\prime)  \stackrel{\epsilon^\ast \tensor -}\longrightarrow & 
\Sigma^{-1}\mathfrak{K}(\gamma)^\ast\tensor \Sigma^{-1}\mathfrak{K}(\gamma^\prime)^\ast \tensor  A(\gamma^\prime) \tensor
A(\gamma\setminus v)
\\ \nonumber
 \stackrel{-\tensor \mu_{\gamma^\prime} \tensor -}\longrightarrow & \Sigma^{-1}\mathfrak{K}(\gamma)^\ast  \tensor A(\gamma) \stackrel{\mu_\gamma} \to A(g,n).
\end{align}

We stress here that while $\mu_\gamma\circ_v\mu_{\gamma^\prime}$ is defined precisely when $\gamma\circ_v\gamma^\prime$ is defined, it is not the case that 
$\mu_\gamma\circ_v\mu_{\gamma^\prime}$ and $\mu_{\gamma\circ_v\gamma^\prime}$
are equal; they don't even have the same degree. The latter is a generating operation, the former is not.

\subsubsection{Differential Constraint}

Fix a modular graph $\gamma$ with a nest $N$ (Definition $\ref{nestdef}$).
Define leg$(N)$ to be the set of flags of $\gamma$ which are adjacent to a vertex in the closure of $N$, but are not part of an edge of $N$.  We define the {\bf flag closure} of $N$ to be the modular leg$(N)$-graph formed by attaching these flags to the closure of $N$ as specified by $a_\gamma$, the adjacency map in $\gamma$.  We denote the flag closure of $N$ by $\hat{N}$.

Given a nest $N$ on $\gamma$, we may also form $\gamma/N$, defined to be the modular graph formed by collapsing the edges of $N$ to a common vertex, which we also call $N$.  For any modular graph $\gamma$ and any nest $N$, the vertex substitution $\gamma/N\circ_N \hat{N}$ is well defined and hence the composition $\mu_{\gamma/N}\circ_N \mu_{\hat{N}}$ is well defined for any weak modular operad.  Conversely, any vertex substitution is of this form. 

Denote the total genus (Subsection $\ref{graphssec2}$)  of a modular graph $\gamma$ by $g(\gamma)$.  We may now give a more combinatorial characterization of a weak modular operad. 

\begin{proposition}\label{wmo}  A weak modular operad structure on a $\mathbb{V}$-module $A$ is equivalent to a collection of $Aut(\gamma)$-invariant, degree $0$ maps
	\begin{equation}\label{genops} \Sigma^{-1}\mathfrak{K}(\gamma)^\ast \tensor A(\gamma)\stackrel{\mu_\gamma}\to A(g(\gamma),|leg(\gamma)|),
	\end{equation}
	one for each modular graph $\gamma$, which are $S_{|leg(\gamma)|}$-equivariant and which satisfy the condition
\begin{equation}\label{dconstraint}
d(\mu_\gamma) = \ds\sum_{\substack{ \text{non empty nests }  \\ N  \text{ on } \gamma} } \mu_{\gamma/N}\circ_N \mu_{\hat{N}}.
\end{equation}
\end{proposition} 
\begin{proof}  
	
	From the above analysis it follows that an algebra over   $F(\Sigma^{-1}((\mathbb{M}_{\mathfrak{K}})^\ast)$ is equivalent to the family of operations in Formula $\ref{genops}$.  It remains to verify that such an algebra structure is compatible with the differential in the cobar construction if and only if Formula $\ref{dconstraint}$ is satisfied.
	
	The differential on $F(\Sigma^{-1}((\mathbb{M}_{\mathfrak{K}})^\ast)(\vec{v})$ is
	determined on the image of the inclusion
	 $$\Sigma^{-1}(\mathbb{M}_{\mathfrak{K}}(\vec{v}))^\ast \hookrightarrow F(\Sigma^{-1}((\mathbb{M}_{\mathfrak{K}})^\ast))(\vec{v})$$
	This map is in turn determined by its image on elements which are homogeneous with respect to the direct sum decomposition over the set $M(\vec{v})$.  Thus it is enough to consider the cobar differential on a summand $\Sigma^{-1}\mathfrak{K}(\gamma)^\ast$ for a modular graph $\gamma$ of type $\vec{v}$.  
	
The operad structure of $\mathbb{M}_\mathfrak{K}$ is given by graph insertion, with the convention that the target of a composition has the mod 2 edge order given by $out\wedge in$.  Thus, $\mathfrak{K}(\gamma)$ is in the image of operadic structure maps corresponding to non-empty nests $N$ on $\gamma$.  In the notation of Formula $\ref{Ksigns}$, each such operadic structure map induces the map 
\begin{equation*}
\epsilon(\gamma/N,N,\hat{N})\colon \Sigma\mathfrak{K}(\gamma/N)\tensor\Sigma\mathfrak{K}(\hat{N})\longrightarrow \Sigma\mathfrak{K}(\gamma)
\end{equation*}
on the suspension $\Sigma\mathfrak{K}(\gamma)$.

Therefore the cobar differential restricted to $\Sigma^{-1}\mathfrak{K}(\gamma)^\ast$ can be written as $
\sum_N \epsilon(\gamma/N,N,\hat{N})^\ast$, where the sum is taken over non-empty nests on $\gamma$.  The result then follows by direct comparison with the definition of $\mu_\gamma$ (Formula $\ref{comps}$).\end{proof}

The reason no signs appear in Formula $\ref{dconstraint}$ is because they have been encoded in the isomorphism in Formula $\ref{Ksigns}$.  We now give an even more down-to-earth presentation of weak modular operads by making these signs explicit.

Let $|\gamma|$ denote the number of edges in a modular graph $\gamma$.  An {\bf edge ordered modular graph} is a modular graph $\gamma$ along with a bijection $\psi\colon  Edges(\gamma) \to \{1\cdc |\gamma|\}$.  We denote such an edge ordered modular graph by $(\gamma,\psi)$

Given an edge ordered modular graph $(\gamma, \psi)$ and a nest $N$ on $\gamma$, we consider both $N$ and $N^c :=Edges(\gamma)\setminus N$ with the total orders inherited from $\psi$.  We then define $\kappa=\kappa(\gamma, N)\in S_{|\gamma|}$ to be the unique $(|N^c|,|N|)$-shuffle for which $\kappa(\{1\cdc |N^c|\})=\psi(N^c)$.  In other words, $\kappa$ permutes the induced edge order on the vertex substitution $\gamma/N\circ_N\hat{N}$ (under the convention ``outside before inside'') to the edge order given by $\psi$.

 Any automorphism $\phi$ of a modular graph $\gamma$ determines a permutation of its edges which we denote by $\tilde{\phi}$. We also write $sgn(\sigma)\in \{-1,+1\}$ for the sign of a permutation $\sigma$.  With this notation Proposition $\ref{wmo}$ directly translates to:

\begin{corollary}\label{wmocor}  A weak modular operad structure on a $\mathbb{V}$-module $A$ is equivalent to a collection of degree $|\gamma|-1$ maps,
	\begin{equation}\label{genops2}  A(\gamma)\stackrel{\mu(\gamma,\psi)}\longrightarrow A(g(\gamma),|leg(\gamma)|),
	\end{equation}
	one for each edge ordered modular graph $(\gamma,\psi)$, 
which are $S_{|leg(\gamma)|}$-equivariant and which satisfy

	\begin{itemize}
		
		\item $\mu(\gamma,\psi\tau)=sgn(\tau)\mu(\gamma,\psi)$ for $\tau\in S_{Edges(\gamma)}$,
		
		
		\item  $\mu(\gamma, \psi)\circ A(\phi) =\mu(\gamma, \tilde{\phi}\psi)$ for $\phi \in Aut(\gamma)$,
		
		\item 	$
		d(\mu(\gamma,\psi)) = \ds\sum_{\substack{ \text{non empty nests }  \\ N  \text{ on } \gamma} } (-1)^{|\gamma/N|} sgn(\kappa(\gamma, N)) \left(\mu(\gamma/N, \psi/N)\circ_N \mu(\hat{N}, \psi_{|N})  \right)$.
	\end{itemize} 

\end{corollary}

\begin{definition}  A {\bf weak $\mathfrak{K}$-modular operad} is an algebra over $(\M_\mathfrak{K})_\infty:=\Omega(\M^\ast)$.	
\end{definition}

If we replace $\mathbb{M}_\mathfrak{K}$ with $\mathbb{M}$ in the analysis leading up to Proposition $\ref{wmo}$, The only difference is the absence of the $\mathfrak{K}(\gamma)$ tensor factor.  This yields the following characterization of weak $\mathfrak{K}$-modular operads.

\begin{proposition}\label{wmo2}  A weak $\mathfrak{K}$-modular operad structure on a $\mathbb{V}$-module $A$ is equivalent to a collection of $Aut(\gamma)$-invariant, degree $-1$ maps
$	\mu_\gamma\colon A(\gamma)\to A(g(\gamma),|leg(\gamma)|)$,
	one for each modular graph $\gamma$, which are $S_{|leg(\gamma)|}$-equivariant and which satisfy the condition
	\begin{equation}
	d(\mu_\gamma) = \ds\sum_{\substack{ \text{non empty nests }  \\ N  \text{ on } \gamma} } \mu_{\gamma/N}\circ_N \mu_{\hat{N}}.
	\end{equation}
\end{proposition}

\section{Graph homology and the weak Feynman transform.}\label{secWFT}  In this section we:
\begin{itemize}
	\item Define the weak Feynman transform of a weak modular operad and use Theorem $\ref{koszulthm}$ to establish its homotopy theoretic properties.
	\item Define (hairy) graph homologies and their associated Massey products.
	\item  Use the weak Feynman transform to show that Massey products hit all graph homology classes.
	\item Construct and analyze spectral sequences which compute the homology of the weak Feynman transform from the homology of the classical Feynman transform.
	\item Compute leg free commutative graph homology in genus 3 using these spectral sequences.
	\item Give a list of interesting and expected future directions stemming from the results of this paper.
\end{itemize}

In this section we will consider the categories of weak modular operads with their $\infty$-morphisms.  To simplify the presentation we, following \cite{GeK2}, restrict attention to $\mathbb{V}$-modules which are finite dimensional in each graded component.  This will ensure that the bar and cobar constructions for algebras are intertwined by linear duality (Lemma $\ref{fdbar}$).  We could drop this assumption at the cost of working with four categories of twisted and untwisted modular operads and cooperads.

\begin{definition}\label{FTDef}  The weak Feynman transform is a pair of functors
	\begin{equation*}
\FT^+\colon \left\{ \text{weak } \text{modular operads} \right\}\to \left\{\mathfrak{K}\text{-modular operads} \right\} 
	\end{equation*}
and		\begin{equation*}
\FT^-\colon \left\{ \text{weak } \mathfrak{K}\text{-modular operads} \right\}\to \left\{ \text{modular operads} \right\} 
\end{equation*}

	defined by $\FT^+ :=  \mathsf{B}^\ast_{\M_\infty}$ and $\FT^- := \mathsf{B}^\ast_{(\M_\mathfrak{K})_\infty}$.
\end{definition}

Let's unpack this definition.  The weak Feynman transform of a weak modular operad $A$ may be identified with the free $s^{-1}\M^!\cong\mathbb{M}_\mathfrak{K}$-algebra on $A^\ast$.  By definition this is given object-wise by the space:
\begin{equation}\label{FT1}
\FT^+(A)(g,n)\cong \left[\ds\bigoplus_{\substack{\vec{v} \text{ such that} \\ v_0=(g,n)} } \underline{\M_\mathfrak{K}}(\vec{v})\tensor_{in(\vec{v})} A^\ast(in(\vec{v}))\right]_{S_{|\vec{v}|}},
\end{equation}
where $\underline{\M_\mathfrak{K}}= \M_\mathfrak{K}\oplus \op{I}$ (see Remark $\ref{non-unital algebras}$).
A homogeneous element of this space specifies a graph whose vertices are ordered and labeled by elements of $A^\ast$.  Taking coinvariants of the $S_{|\vec{v}|}$ action has the effect of simply ignoring the vertex order of the graph.  Taking coinvariants with respect to the $in(\vec{v})$-action means that permuting labels of flags which are then glued together to form edges has no effect, thus then these elements are $Aut(\gamma)$-coinvariant.
Thus, splitting each summand in line $\ref{FT1}$ over the underlying graphs yields the following decomposition of graded vector spaces:
\begin{equation*}
\FT^+(A)(g,n)\cong \ds\bigoplus_{(g,n)\text{-graphs } \gamma} (\mathfrak{K}(\gamma) \tensor A^\ast(\gamma))_{Aut(\gamma)}
\end{equation*}
where $A^\ast(\gamma)$ and $\mathfrak{K}(\gamma)$ are as defined in Subsection $\ref{wmosec}$.  The direct sum is taken over all modular graphs of type $(g,n)$, along with the corolla $\ast_{g,n}$ (see Remark $\ref{noedgesremark})$, with convention $A(\ast_{g,n})=A(g,n)$.

Let us now describe the differential on these spaces.  For each modular graph $\gamma$ of type $(g,n)$, define a degree $-1$ map $d_\gamma := (\Sigma \mu_\gamma)^\ast$, where $\mu_\gamma$ is given in Formula $\ref{genops}$.  We then define
\begin{equation*}
d\colon A^\ast(g,n)\to \ds\bigoplus_{(g,n)\text{-graphs } \gamma} (\mathfrak{K}(\gamma)\tensor A^\ast(\gamma))_{Aut(\gamma)}
\end{equation*}
by first summing over all such $d_\gamma$ and then passing to coinvariants.  Due to the genus labeling, there are only finitely many $\gamma$ for a fixed $(g,n)$ (i.e.\ $\mathbb{M}$ is reduced in the sense of Definition $\ref{reduced}$) and so this sum is finite.  This map extends uniquely to a degree $-1$ map of $\mathfrak{K}$-twisted modular operads $d\colon \FT^+(A)\to\FT^+(A)$.   

Conversely, since $s^{-1}\M_\mathfrak{K}^!\cong\M$, the weak Feynman transform of a weak $\mathfrak{K}$-modular operad $B$ may be identified with a free modular operad (not twisted):
\begin{equation*}
\FT^-(B)(g,n)\cong \ds\bigoplus_{(g,n)\text{-graphs } \gamma} B^\ast(\gamma)_{Aut(\gamma)}.
\end{equation*}
In this case the differential may also be described as a sum over all graphs, although now the differential has degree -1 because all the operations had degree -1 to begin with.  The comment about finiteness for terms in the differential still applies.

We remark that the above description may be compared with the original construction of Getzler and Kapranov to show that the Feynman transform of a (strict) modular operad, which we denote $\FTGK$, (resp.\ strict $\mathfrak{K}$-twisted modular operad) as defined in \cite{GeK2} agrees with Definition $\ref{FTDef}$.
This follows from the fact that for a strict modular operad, only the graphs with $1$ edge act non-trivially and so $d$ coincides with the \cite{GeK2} differential, call it $d_\FTGK$.  The external differential in the weak Feynman transform is of the form $d=d_\FTGK$+higher terms, but it is not in general true that $d_\FTGK$ is square zero.

\subsection{Parity} \label{parity}

To formalize considerations of modular versus $\mathfrak{K}$-twisted modular operads we define a category
\begin{equation*}
\{\pm \text{modular operads} \}:= \{\text{modular operads} \} \coprod \{\mathfrak{K}\text{-twisted modular operads}\}
\end{equation*}
Here we take the disjoint union of both objects and morphisms.  From now on we refer, by abuse of terminology, to objects of this category simply as modular operads, and we refer to their parity as odd (for those which are $\mathfrak{K}$-twisted) or even (for those which are not) as needed.  We repeat this construction to form the category of $\{\pm \text{weak modular operads} \}$, with their $\infty$-morphisms.

Note that with this definition $\FT^\pm$ (resp. $\FTGK^\pm$) combine to define parity reversing endofunctors on these categories.  In particular, if we define $\FT:=\FT^+\sqcup \FT^-$, then we may state the following immediate corollary of Theorem $\ref{koszulthm}$ without reference to parity:

\begin{corollary}\label{ftcor}  The functor $\FT$ sends $\infty$-quasi-isomorphisms to quasi-isomorphisms. 
\end{corollary}

We remark that Theorem $\ref{koszulthm}$ also shows that, denoting the bar-cobar resolution as $\FT^2$, the natural transformation $\FT^2\Rightarrow id$ consists entirely of quasi-isomorphisms.

\subsection{Definition of graph homology.}
Cyclic operads may be defined as the full subcategory consisting of those modular operads which have $A(g,n)=0$ for $g\geq 1$.  We similarly define the category of weak cyclic operads to be the full subcategory of weak modular operads which is $0$ in genus $\geq 1$.  Notice with our parity conventions, this gives us a notion of both even and odd (weak) cyclic operads. 

We denote the inclusion functor from (weak) cyclic operads to (weak) modular operads by $\iota_!$.  This functor is right adjoint to the functor which forgets higher genus, call this functor $\iota^\ast$.  This functor is itself a right adjoint to a certain $\iota_\ast$, usually called the modular envelope.  This choice of notation for the triple of adjoint functors $(\iota_\ast,\iota^\ast,\iota_!)$ may be justified by considering their compatibility with the bar construction, see section 9.1 of \cite{W6}.

\begin{definition}  For a cyclic operad $\op{O}$, the $\op{O}$-graph homology is the modular operad $\op{G}_{\op{O}}:=H_\ast(\FTGK(\iota_!\op{O}))$.  
\end{definition}

We emphasize the parity change: if $\op{O}$ was even then the $\op{O}$-graph homology is odd and vice versa.  We remark that the relationship between odd cyclic operads and cyclic operads was extensively studied in \cite{KWZ}.  In particular there is an isomorphism of categories between even and odd cyclic operads given by the functor $s\Sigma^{-1} $, shift and suspend.  In this way we can form canonical odd cyclic operads associated to a given cyclic operad.  Thus to any cyclic operad we can consider both its graph homology (the odd modular operad $\op{G}_{\op{O}}$) and the graph homology of its oddification, the even modular operad $\op{G}_{s\Sigma^{-1} \op{O}}$.  Both will play a role in our examples below.

For the reader fluent in the language of twisting cocycles \cite{GeK2}, a cyclic operad may be viewed as a Det-modular operad, and the cocycle Det is cohomologous to $\mathfrak{K}$ via the coboundary associated to $s\Sigma^{-1}$ \cite[Proposition 4.14]{GeK2}.  When comparing to loc.cit., note that here we are using homological conventions, and our definition of operadic suspension is opposite to \cite{GeK2}.

\subsection{Massey products for graph homology}

For us, the terminology ``Massey product'' refers to the transferred weak modular operad structure on the homology of a modular operad (resp.\ the $\mathfrak{K}$-twisted analog).  Let us first give a formal definition.

Let $\rho\colon \mathbb{M}\to End_A$ be a modular operad structure on a dg $\mathbb{V}$ module $A$.  Choose a deformation retract of $\mathbb{V}$-modules:	\begin{equation*}
\xymatrix{  \save !R(-.7) \ar@(ul,dl)_{h} \restore  A \ar@/_1pc/@{->>}[rr]^{\pi} & & \ar@/_1pc/@{->}[ll]_{\iota} H_\ast(A) 
} \end{equation*}
This endows $H_\ast(A)$ with the structure of a weak modular operad $\hat{\rho}\colon \mathbb{M}_\infty\to End_{H_\ast(A)}$ via
Theorem $\ref{htt}$.

\begin{definition}
Let $\gamma$ be a modular graph of type $v_0=(g(\gamma), |leg(\gamma)|)$.  The Massey product $\text{mp}_\gamma$ of type $\gamma$ associated to $\rho,\iota,\pi$ and $h$ is defined to be the weak modular operadic structure map $\mu_\gamma$ defined in Proposition $\ref{wmo}$, with respect to the weak modular operad structure $\hat{\rho}$. Explicitly
\begin{equation*}
\text{mp}_\gamma\colon\mathfrak{K}(\gamma)\tensor H_\ast(A)(\gamma)\to H_\ast(A)(v_0)
\end{equation*}
\end{definition}

We recall that one may derive explicit combinatorial formulas for the transferred structure $\hat{\rho}$ in terms of $\rho,\iota,\pi$ and $h$ (see \cite[Chapter 10.3]{LV} as well as the proof of Theorem $\ref{htt}$ above) and hence explicit combinatorial formulas for $\text{mp}_\gamma$.

There is an alternate description of these operations using Corollary $\ref{wmocor}$.  Let $\gamma$ be modular graph of type $v_0$ whose vertices are of type $(v_1\cdc v_r)$.  Fix an edge order $\sigma\colon Edges(\gamma)\stackrel{\cong}\to \{1\cdc |\gamma|\}$.  To the pair $(\gamma, \sigma)$ we associate the operation $\text{mp}_{(\gamma,\sigma)}:= \mu_{(\gamma,\sigma)}$.  Explicitly, $\text{mp}_{(\gamma,\sigma)}$ is an operation
\begin{equation*}
\text{mp}_{(\gamma,\sigma)}\colon \ds\bigotimes_{i=1}^r H_\ast(A)(v_i)\to H_\ast(A)(v_0)
\end{equation*}
of degree $|\gamma|-1$.  These operations satisfy the three properties given in Corollary $\ref{wmocor}$, with the additional condition that the internal differential is zero in this case.

We similarly define Massey products for a $\mathfrak{K}$-twisted modular operad to be the weak $\mathfrak{K}$-twisted modular operadic structure maps (Proposition $\ref{wmo2}$) relative to the transferred $\Omega(\mathbb{M}^\ast)$-algebra structure on its homology.

In the case of modular operads arising in graph homology, we have the following structural result:

\begin{theorem} \label{vanishing}  Every graph homology class in genus $g\geq 1$ is in the image of some Massey product.
\end{theorem}
\begin{proof}  
	Fix $\op{O}$ a cyclic operad (of either parity) and choose a deformation retract of $\mathbb{V}$-modules between $\op{G}_{\op{O}}$ and $\FTGK(\iota_!\op{O})$ using Lemma $\ref{eqlem}$.  Using Theorem $\ref{htt}$, we then endow $\op{G_O}$ with the structure of a weak modular operad such that, as weak modular operads $\op{G_O}\sim \FTGK(\iota_!\op{O})$.  Applying the functor $\FT$, we employ Corollary $\ref{ftcor}$ to conclude $\FT^2\sim 0$ for genus $\geq 1$. 
	
	Fix a graph homology class $\eta \in \op{G_O}(g,n)$ with $g\geq 1$.  Consider an element in the complex $\FT(\op{G_O})(g,n)$ corresponding to the corolla with $n$ flags, genus label $g$ and vertex label $\eta^\ast$, a non-zero linear functional supported on the span of $\eta$.  This element is not a boundary (since it labels a corolla and the internal differential is $0$), so it can not be a cycle since $H_\ast(\FT(\op{G_O}))(g,n)\cong \iota_!\op{O}(g,n)$ is zero in genus $g\geq 1$.  It follows that the $\FT$ differential $d$ has $d(\eta^\ast)\neq 0$.  On the other hand if $\eta$ was in the image of no Massey product, we would have $d(\eta^\ast)=0$, hence the claim.
\end{proof}

We remark that if $\op{O}$ is a Koszul cyclic operad, then $\op{G_O}(0,-)$ is simply the Koszul dual cyclic operad $\op{O}^!$, and in particular all classes in genus $0$ are generated by cyclic operadic compositions of the generators of $\op{O}^!$.  So in this case we may conclude that {\it all} homology classes are hit by Massey products of generators of $\op{O}^!$.

\subsection{Filtrations of the weak Feynman transform.}\label{sssec}  Given a weak modular operad whose internal differential is zero, we may consider the underlying (strong) modular operad and its (classical) Feynman transform.  The goal of this subsection will be to come up with filtrations which isolate this differential and hence produce spectral sequences which compute the homology of the weak Feynman transform from the homology of the (classical) Feynman transform.  

\subsubsection{Internal degree filtration}\label{ss1}  
Let $A$ be a weak modular operad of even parity with internal differential $d_A=0$.  A homogeneous element in $\FT(A)(g,n)$ is an element of some $\mathfrak{K}(\beta)\tensor A^\ast(\beta)$, for some modular graph $\beta$.  If $\beta$ has $s$ edges, then we say such an element of degree $m$ has internal degree $r=m+s$.  With respect to the bigrading $(m,r)$, the part of the differential corresponding to a graph $\gamma$ takes:
\begin{equation*}
\FT(A)(g,n)_{m,r}\stackrel{d_\gamma}\longrightarrow \FT(A)(g,n)_{m-1,r-1+e(\gamma)}
\end{equation*}
In particular, since $d_A=0$, the $\FT$ differential can not decrease the internal grading.  We may thus define a filtration on the chain complex $\FT(A)(g,n)$ by defining $F^t(\FT(A)(g,n))\subset \FT(A)(g,n)$ to be the subspace of elements of internal degree $\geq -t$.  

We define $(E^\ast_{\bullet, \bullet}(A),d_\ast)$ to be the family of spectral sequences associated to this filtration at each $(g,n)$.
\begin{lemma} \label{sslem} The spectral sequences $(E^\ast_{\bullet, \bullet}(A),d_\ast)$ take the following form:
\begin{itemize}
	\item $(E^0(A), d_0)$ is the (classical) Feynman transform of the (strong) modular operad underlying $A$. 
	\item  $E^1(A)$ is the homology of the Feynman transform of the (strong) modular operad underlying $A$.  
	\item  $d_n$ is induced by blowing up graphs with $n+1$ edges.  In particular if $A$ is a strong modular operad then $d_n=0$ for $n\geq 1$. 
	\item  If each $A(g,n)$ is bounded below then $E^\infty(A) \cong H_\ast(\FT(A))$.
\end{itemize}
\end{lemma}
\begin{proof}
	We observe that with respect to this bigrading, $r$ is preserved by $d_\gamma$ for graphs $\gamma$ with $1$ edge and increases otherwise.  Thus $d_0$ blows up one edge, which is precisely the classical Feynman transform differential of the underlying modular operad. 
	If each $A(g,n)$ is bounded below, then each $\FT(A)(g,n)$ is bounded above, and thus this exhaustive filtration is also bounded below.
\end{proof}

\begin{remark}  There is a corresponding spectral sequence for $B$ of odd parity.  In this case we filter by row degree $r-s$ so that $d_\gamma$ raises the row/filtration degree for graphs with more than one edge and preserves the row for graphs with one edge.
\end{remark}

\subsubsection{Genus label filtration}\label{ss2}  Recall (Definition $\ref{modulargraphdef}$) that modular graphs come with a genus labeling of the vertices.  Given a modular graph $\gamma$, we let $\ell(\gamma)$ denote the sum of its vertex labels and $g(\gamma)$ denote its total genus.  These natural numbers are related by the formula $g(\gamma) = \ell(\gamma) + \beta_1(\gamma)$, where $\beta_1(\gamma)$ is the first Betti number of $\gamma$ when viewed as a CW complex.

In this section we let $B$ be a weak modular modular operad (of either parity) and we do not assume that $d_B=0$.
Define $F^q(\FT(B)(g,n)) \subset \FT(B)(g,n)$ to be the subspace of homogeneous elements whose corresponding modular graph has $\ell\leq q$.  Since each $d_\gamma$ decreases $\ell$ by a nonnegative integer (namely $\beta_1(\gamma)$), this filtration is compatible with the differential.  
If we use the bigrading $(r-\ell,\ell)$, for $\FT(B)$, then this is just filtration by rows.  Denote the associated spectral sequences over all $(g,n)$ by $(L^\ast_{\bullet,\bullet}(B), d_\ast)$. Notice this filtration is bounded:
\begin{equation*}
0\subset F^0(\FT(B)(g,n)) \subset ... \subset F^g(\FT(B)(g,n))= \FT(B)(g,n)
\end{equation*}
Thus $(L^\ast_{\bullet,\bullet}(B), d_\ast)$ are upper half plane spectral sequences which converge level-wise to $H_\ast(\FT(B))$.

The spectral sequences $(L_{\bullet,\bullet}^\ast(B), d_\ast)$  are most interesting when we can understand the differential $d_0$.  To fix a context in which this is possible we first make the following observation: many modular operads which are not formal have underlying cyclic operads which are formal.  So when applying homotopy transfer theory to construct weak modular operads, the underlying weak cyclic operads will often be (strong) cyclic operads.

\begin{lemma}  Suppose that the weak cyclic operad $\iota^\ast(B)$ is a (strong) cyclic operad.  Then the upper half plane spectral sequences $(L^0_{\bullet,\bullet}(B), d_0)$ have bottom rows $(L^0_{\bullet,0}(B), d_0)=(\FTGK(\iota_!\iota^\ast(B)), d)$.
\end{lemma}
\begin{proof}  For any weak modular operad we have $F^0(\FT(B)) = \FT(\iota_!\iota^\ast(B))$.  The condition that the weak cyclic operad $\iota^\ast(B)$ is strong ensures that the only non-zero compositions in $\iota_!\iota^\ast(B)$ are generated by one-edged trees, from which the equality of the differentials follows.
\end{proof}


\begin{corollary}\label{cycopss} If $B$ is a weak modular operad such that $\iota^\ast(B)$ is a (strong) cyclic operad then  $\op{G}_{\iota^\ast{B}}\cong L^1_{\bullet,0}(B)$.  \end{corollary}

\subsection{Comparing Lie and Commutative graph homology.} \label{liesec} 

In this subsection we will see how the spectral sequences above may be used to relate variants of graph homology related to the commutative and Lie operads respectively.  We begin by recalling some known information.

In \cite{CHKV}, the authors introduce and study a family of groups extensions $\Gamma_{g,n}$ of the outer automorphism group of the free group $Out(F_g)$, having the property that $H_\ast(\Gamma_{g,n})\cong \op{G}_{s\Sigma^{-1}Lie}(g,n)$, see also \cite[Theorem 8.8]{CKV13}.  This allows the authors to perform low genus computations of Lie graph homology using the Leray-Serre spectral sequence.

Here are several results from \cite{CHKV}:
\begin{itemize}
	\item The homology of $\Gamma_{g,n}$ is concentrated between degrees $0$ and $2g+n-3$ inclusive -- this upper bound is called the virtual cohomological dimension (VCD) of $\Gamma_{g,n}$.
	\item Classes in the VCD can not be in the image of the modular operad generated by $H_0(\Gamma_{0,3})$ for degree reasons.
	\item The class generating $H_0(\Gamma_{0,3})$ along with those in the VCD generate all classes under the modular operad structure in genera $\leq 2$.  This is conjectured to be the case for all genera. 
	\item  The $S_n$ representation $H_i(\Gamma_{1,n})$ is non-zero if and only if $i$ is even and $0\leq i \leq n-1$.  In this case, $H_i(\Gamma_{1,n})$ is irreducible of type $V_{n-i,1^i}$, and in particular has dimension ${n-1 \choose i}$.
	\item Consequently, $H_{VCD}(\Gamma_{1,n})$ is the alternating representation of the group $S_n$ when $n$ is odd and is $0$ when $n$ is even.  For $n$ odd we fix a non-trivial alternating class $g_n$. \end{itemize}

Next we recall the computations of \cite{CGP}.  The authors consider complexes of metric genus labeled graphs $\Delta_{g,n}$ computing the top weight homology of the moduli space of punctured Riemann surfaces \cite{CGP}.  This homology may be computed via cellular chain complexes which coincide (up to shift in degree) with the linear dual of the Feynman transform of $\iota_\ast Com$ (compare \cite[Remark 3.3]{CGP}).  Explicitly:  
\begin{equation*}
H_{\ast-1}(\Delta_{g,n})\cong H_\ast(\FTGK(\iota_\ast Com)(g,n)^\ast).
\end{equation*}
In genus 1 this homology has rank $(n-1)!/2$ in degree $n$ and $0$ elsewhere (Theorem 1.2 of \cite{CGP}).  It is generated by vertex labellings of an $n$-gon.

It is interesting to see how the two spectral sequences in Subsection $\ref{sssec}$ can be played off each other in this example.  Abstractly, this can be done for arbitrary $(g,n)$, and in particular establishes Corollary $\ref{sscor}$, but we find the most traction when either $g$ or $n$ is small.

\subsubsection{The case $g=1$.}

Let's first perform some analysis in the case $g=1$. Consider the genus label filtration spectral sequence (Subsection $\ref{ss2}$) applied to $B=H_\ast(\Delta):=\{H_\ast(\Delta_{g,n})\}_{g,n}$, viewed as a weak modular operad via homotopy transfer.  In genus $1$ it takes the form:

\begin{enumerate}
	\item  There are two rows: the top corresponds to trees with a distinguished vertex (labeled by $1$).
	\item  The bottom row on page $L^0$ is $\FTGK( \iota_!(s\Sigma^{-1}Lie))$ (since $\iota^\ast(H_\ast(\FTGK(\iota_\ast Com)))=s\Sigma^{-1}Lie$).
	\item  The bottom row on page $L^1$ is $H_\ast(\Gamma_{1,n})$ (by Corollary $\ref{cycopss}$).
	\item  It converges to $k$ in every bidegree.
\end{enumerate}
See Figure $\ref{fig:ss2}$ for an example.
\begin{figure}
	\centering
	\includegraphics[width=0.7\linewidth]{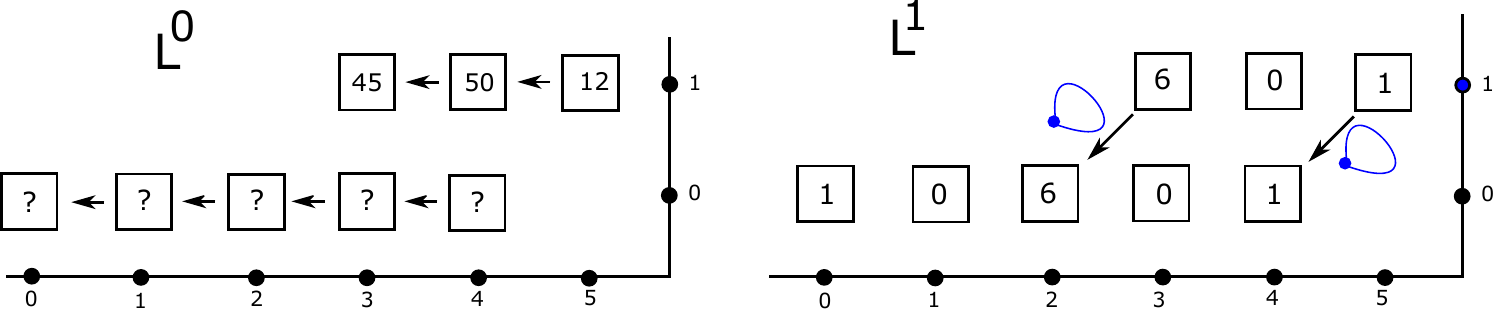}
	\caption{Dimensions in the spectral sequence $L_{\bullet,\bullet}(H_\ast(\Delta))(1,5)$.  Here the sum of genus labels is on the vertical axis and total degree is on the horizontal axis. The bottom row on page 1 is $H_\ast(\Gamma_{1,5})$.}
	\label{fig:ss2}
\end{figure}

As an immediate corollary we see that the row $\ell=1$ of this spectral sequence computes the reduced homology $\widetilde{H}_\ast(\Gamma_{1,n})$.  We remark that this complex is substantially smaller than the $\ell=0$ row; when $n=4$ the bottom row has dimension 174, the top row has dimension 9; when $n=3$ the bottom row has dimension 18, the top row has dimension 1.  In particular the cycle in $Lie(2n+3)^{S_2}$ representing $g_{2n+1}$ may be found by applying the differential $d_1$ to the anti-invariant class of $H_\ast(\Delta_{1,2n+1})$, which is represented by a sum over labellings of a $2n+1$-gon with a leg at each vertex.

By contrast, we may consider $H_\ast(\Gamma):=\{ H_\ast(\Gamma_{g,n})\}_{g,n}$, which we also view as a weak modular operad via homotopy transfer, and consider the internal degree filtration spectral sequence (Subsection $\ref{ss1}$) applied to $A=H_\ast(\Gamma)$. In genus 1 it takes the following form:
\begin{enumerate}
\item  Only the non-negative even rows are non-zero.
\item  For $i\geq 1$, the $2i^{th}$-row is a complex of trees with a distinguished vertex having genus label 1.  The distinguished vertex is in turn labeled by a class in $H_{2i}(\Gamma_{1,m})$, where $m$ is its valence, while the non-distinguished vertices carry commutative labels of degree $0$.  
\item  The $0^{th}$-row is $\FTGK(\iota_\ast Com)(1,n)$. 
\item  The spectral sequence converges to $0$ since $g\geq 1$.
\end{enumerate}
See Figure $\ref{fig:ss1}$ for an example when $n=5$.

\begin{figure}
	\centering
	\includegraphics[scale=1]{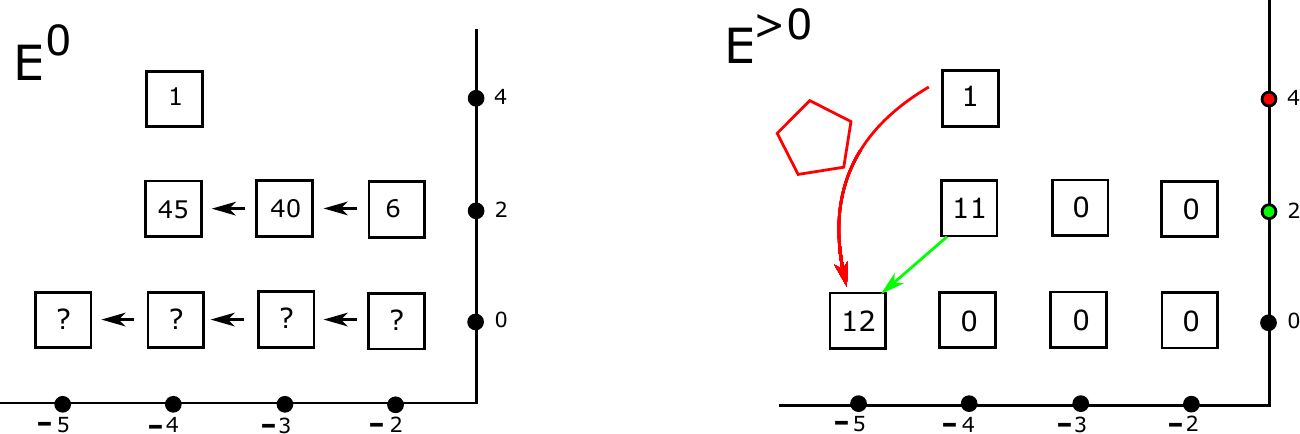}
	\caption{Dimensions in the spectral sequence $E_{\bullet,\bullet}(H_\ast(\Gamma))(1,5)$.  
		The vertical axis shows the negative of the internal degree (which are the degrees in $H_\ast(\Gamma)$ before linear dualization). The horizontal axis shows the total degree.  The sum of the two coordinates is negative the number of edges. Differentials corresponding to Massey products appear on pages 2 and 4.  The bottom row on page 1 is $H_{\ast+1}(\Delta_{1,5})^\ast$.}
	\label{fig:ss1}
\end{figure}

The higher differentials in this spectral sequence depend on the choice of deformation retract, but there are some features which are true regardless of choices.  For example:

\begin{proposition}\label{Lieprop} (Massey Products for $H_\ast(\Gamma_{1,m})$)  The semi-classical weak modular operad $H_\ast(\Gamma_{g\leq 1}, n)$ is generated by $H_0(\Gamma_{0, 3})$.
\end{proposition}
\begin{proof}  By ``semi-classical'' we mean the restriction to genera $\leq 1$, after \cite{GetSC}.
	The results of \cite{CHKV} reduce the statement to showing that the classes of degree equal to the VCD are in the image of some Massey product with genus $0$ source.   In genus $1$, the homology classes of degree equal to the VCD are spanned by $g_{2m+1}\in H_{2m}(\Gamma_{1,2m+1})$.  By Theorem $\ref{vanishing}$, $g_{2m+1}$ must be in the image of some Massey product
	\begin{equation*}
	\text{mp}_{(\gamma,\sigma)}\colon (\tensor_i H_\ast(\Gamma_{g_i,n_i}))\to H_{2m}(\Gamma_{1,2m+1}),
	\end{equation*}
	such that  $\beta_1(\gamma)+\sum g_i=1$.
	
	Suppose that $\beta_1(\gamma)=0$.  Then the modular graph $\gamma$ is a tree with $s\geq 1$ edges along with some distinguished vertex $v_i$ of type $(1,n_i)$.  By stability considerations we must have $n_i+s\leq 2m+1$, since vertices of type $(0,2)$ and $(0,1)$ are not allowed in $\gamma$.  The degree of $\text{mp}_{(\gamma,\sigma)}$ is $s-1$ and the degree of the input is determined by the genus 1 label and must be $\leq n_i-1$.  Thus, the maximum output degree of such an operation is $n_i-1+s-1 \leq 2m-1$, which would be a contradiction. We thus conclude $\beta_1(\gamma)=1$ from which it follows that $g_i=0$ for all $i$.  The claim then follows from the fact that each $H_\ast(\Gamma_{0,n_i})$ is generated by $H_0(\Gamma_{0,3})$ under cyclic operadic compositions.
\end{proof}

\subsubsection{The case $n=0$.}

We conclude this paper with a complete calculation of the leg free commutative graph homology in genus $3$.  This calculation agrees with the computer calculations done in \cite[Appendix A]{CGP}.
The novel feature of this calculation is that it requires no analysis of any differentials; it relies only on the representation theory of $H_\ast(\Gamma_{1,n})$ and the elementary properties of the spectral sequence $E^\ast_{\bullet,\bullet}(H_\ast(\Gamma))$ constructed above.

\begin{theorem}\label{genus3} $H_{d}(\FTGK(\iota_\ast Com))(3,0) = \begin{cases}
		k & \text{ if } d=-6 \\
		0 & \text{ else }
	\end{cases}$
\end{theorem}
\begin{proof}  We use the computations from \cite{CHKV} that if $r>0$ then:
	\begin{equation}\label{vanishing1}
		 H_r(\Gamma_{g,n})=0 \text{ if } (g,n)=(3,0),(2,0),(2,1),(2,2),(2,3),(1,1),(1,2),(0,n).
	\end{equation}
In particular the genus 1 and 2 computations are explicitly presented in \cite[Proposition 2.7 and Theorem 2.10]{CHKV},  the genus $0$ computation is equivalent to Koszulity of the Lie operad, and the $(g,n)=(3,0)$ computation is known from the calculation of $H_\ast(Out(F_3))$.

Let $r>0$ and suppose $x\in E_{m,-r}^0(H_\ast(\Gamma))(3,0)$ is a non-zero homogeneous vector.  Any such $x$ has an underlying odd modular graph, call it $\gamma$, with $V$ vertices and $E$ edges.  The graph $\gamma$ has no legs and total genus $3$, so satisfies

\begin{equation*}
 3=\beta_1(\gamma)+\sum_{v\in Vert(\gamma)} g(v) \ \ \ \text{ where } \ \ \ \beta_1(\gamma):=E-V+1,
\end{equation*}
and where $g(v)$ is the genus label of vertex $v$.  Notice that since $\gamma$ appears in row $-r\neq 0$, it can not have $g(v)=0 \ \forall \ v$, since $\Gamma_{0,n}$ has no higher homology.  Thus $\beta_1(\gamma)<3$.  So, writing $||v||$ for the valence of a vertex $v$, for a fixed vertex $v_0$ we calculate:

\begin{equation*}
3\leq ||v_0|| \leq (\sum_{v\in Vert(\gamma)}||v||) -3(V-1) \leq 2E-3V+3 = 2(\beta_1(\gamma)-1)-V+3 \leq 5-V
\end{equation*}

The second inequality follows from the fact that all vertices are at least trivalent.  The third inequality follows from the fact that $\gamma$, being of type $(3,0)$, has no legs.  Notice that one consequence of this inequality is that $\gamma$ has either 1 or 2 vertices.

We will now argue that there is a unique such $\gamma$ for which $x$ can be non-vanishing.  Suppose $\gamma$ had a vertex $v$ with $g(v)=3$.  Then $\beta_1(\gamma)=E-V+1=0$, and $V$ can not be 1, since $H_r(\Gamma_{3,0})=0$, so $V=2$ and $E=1$.  This violates stability, since the vertex which is not $v$ would have valence $1$ and genus $0$.  We conclude that $\gamma$ has no vertex with $g(v)=3$

Suppose  that $\gamma$ has a vertex with $g(v)=2$.  By Equation $\ref{vanishing1}$, $v$ must have $\geq 4$ adjacent flags to be non-vanishing.  If $V=1$, this implies $E\geq 2$ and all edges are loops, contradicting $g(\gamma)=3$.  If $V=2$, then the formula $3-g(v)\geq\beta_1(\gamma)=E-V+1$ implies $E$ is at most 2, and since $\gamma$ is connected this would imply $
||v||<4$.  We conclude that $\gamma$ has no vertex with $g(v)=2$

Suppose that $\gamma$ has a lone vertex $v$ with $g(v)=1$.  Then $\beta_1(\gamma)=E-V+1=2$, and so $\gamma$ has $2$ edges which must be loops.  Define $X$ to be the set of flags adjacent to $v$. Then the homogeneous element $x$ determines an element in $(H_r(\Gamma_{1,X})^\ast)_{Aut(\gamma)}$ which labels $v$.  Let us calculate the dimension of this space. Choose an isomorphism $X\cong \{1,2,3,4\}$ such that the preimages of $1,2$ and $3,4$ form the two edges.  The calculation of the dimension will not depend on this choice of isomorphism. Since $r>0$, the homology group $H_r(\Gamma_{1,4})$ is only non-zero when $r=2$, in which case it is isomorphic to $V_{2,1,1}$ as an $S_4$-representation \cite{CHKV}. The automorphism group of $\gamma$ is the subgroup of $S_4$ generated by the permutations $(12),(34), (13)(24)$, which is isomorphic to the dihedral group $D_{4}$ with $4*2=8$ elements.  Consider $\mathbb{Z}_2\oplus\mathbb{Z}_2\subset D_4$ to be the subgroup generated by $(12)$ and $(34)$.  Then:
\begin{equation*}
(H_2(\Gamma_{1,X})^\ast)_{Aut(\gamma)}\cong H_2(\Gamma_{1,4})^{Aut(\gamma)}\cong V_{2,1,1}^{D_4}\subset V_{2,1,1}^{\mathbb{Z}_2\oplus \mathbb{Z}_2}
\end{equation*} 
It is an elementary exercise in representation theory (using the Littlewood-Richardson rule \cite{FH}) that the dimension of $(V_{2,1,1})^{\mathbb{Z}_2\oplus\mathbb{Z}_2}= 0$.  Thus such an $x$ could not be non-zero.  We conclude that $\gamma$ does not have a lone vertex of genus $1$.

Since $\gamma$ corresponds to an element with internal degree $r>0$, it must have some vertex of non-zero genus. Having ruled out the above cases, we conclude that $\gamma$ has two vertices $v$ and $w$, at least one of which has genus $1$.  Suppose both have genus $1$.  Then $\beta_1(\gamma)=E-V+1=1$, implies $\gamma$ has $2$ edges and so $4$ flags.  On the other hand, Equation $\ref{vanishing1}$ implies that each vertex has at least $3$ adjacent flags, which is a contradiction.

The only remaining possibility is that $\gamma$ has two vertices, $v$ and $w$ with $g(v)=1$ and $g(w)=0$.  From $\beta_1(\gamma)=E-V+1=2$, we conclude there are $3$ edges and hence $6$ flags.  Each vertex must have valence 3 or higher ($w$ by stability and $v$ by Equation $\ref{vanishing1}$), so each vertex has valence equal to $3$.  Let us show that if $x$ is not zero then such a $\gamma$ has no loops.

If there is a loop adjacent to $v$, then $v$ is labeled by the alternating representation of $S_3$, but the automorphism of $\gamma$ which exchanges the flags in such a loop acts by the identity, and so $x$ will vanish.  Thus there are no loops adjacent to $v$, and hence three edges connecting $v$ to $w$, and thus no loops adjacent to $w$.  There is a unique such graph, up to isomorphism, and a one dimensional space on which such an $x$ can lie. It does not vanish in $E^0_{-5,-2}(H_\ast(\Gamma))(0,3)$ because permuting the edges produces the sign of the permutation twice: once since the edges are odd and once since the flags at $v$ carry the alternating representation.  Therefore we conclude, by doing no calculations with any differentials, that $E^1_{m,r}(H_\ast(\Gamma))(0,3)$ is $0$ unless $r=0$ or $(m,r)=(-5,-2)$ in which case $E^0_{-5,-2}(H_\ast(\Gamma))(0,3)\cong k$.

Since $E^1_{m,0}(H_\ast(\Gamma))(3,0)=H_m(\FTGK \iota_\ast Com)(3,0)$, the theorem follows by convergence of this spectral sequence to $0$. \end{proof}

\subsection{Future Directions}\label{futuredirections}

We conclude by outlining several directions for future study.

{\bf I.}  The complex $\oplus_g\FTGK(\iota_\ast Com)(g,0)$ is closely related to a graph complex $\mathsf{GC}_2$ whose $0^{th}$-cohomology coincides with the Grothendieck-Teichm{\"u}ller Lie algebra $\mathfrak{grt}_1$ \cite{WTw}.  Using this correspondence Willwacher describes a family of graph cohomology classes $\sigma_{2j+1}$ corresponding to known conjectural generators of $\mathfrak{grt}_1$ whose leading terms are given by wheel graphs. 
Formally, it must be possible to detect these classes using the spectral sequence $E_{\bullet,\bullet}(H_\ast(\Gamma))$ and Theorem $\ref{genus3}$ does precisely that in the case $j=1$.  It would be desirable to witness this correspondence for $j>1$.

{\bf II.}   A similar but seemingly separate question concerns witnessing homology classes of the outer automorphism group of the free groups $Out(F_g)$.  We recall from \cite{CHKV} that there are two infinite sequences of such classes, the Morita and Eisenstein classes, which are not known to be non-trivial in general.  Formally, the genus label filtration spectral sequence tells us that any non-trivial class can be witnessed by graphs labeled by commutative graph homology classes.   To what extent is it possible to describe, in analogy with (I) above, particular elements in this spectral sequence which correspond to Morita or Eisenstein classes?

{\bf III.} Consider $\op{M}_{g,n}$, the moduli spaces of Riemann surfaces of genus $g$ and $n$ punctures and $\overline{\op{M}}_{g,n}$, their Deligne-Mumford compactifications.  Then  $H_\ast(\overline{\op{M}}_{g,n})$ forms a modular operad which is formal \cite{GNPR} while $H_\ast(\op{M}_{g,n})$ forms a $\mathfrak{K}^{-1}$-twisted modular operad which is not formal \cite{AP}.  Equivalently, $\mathfrak{p}H_{\ast}(\op{M}_{g,n}):= H_{\ast+6-6g-2n}(\op{M}_{g,n})$ forms a $\mathfrak{K}$-twisted modular operad \cite{GeK2}.  In this example the spectral sequence by edge filtration may also be constructed topologically  \cite{KSV2}. 
		The notion of weak modular operads allows for reinterpreting the chain level Koszul duality between $\overline{\op{M}}_{g,n}$ and $\op{M}
	_{g,n}$ \cite[Proposition 6.11]{GeK2} as the existence of a weak $\mathfrak{K}^{-1}$-twisted modular operad structure on $H_\ast(\op{M}_{g,n})$ along with an $\infty$-quasi-isomorphism of weak $\mathfrak{K}$-twisted modular operads:
	\begin{equation*}
	\FTGK (H_\ast(\overline{\op{M}}_{g,n})) \stackrel{\sim}{\rightsquigarrow} \mathfrak{p}H_\ast(\op{M}_{g,n}).
	\end{equation*}
	It would be interesting to give an explicit description of the higher operations having target $H_\ast(\op{M}_{1,n})$. 
	These operations are encoded in the differential of $\FT(H_\ast(\op{M}_{g\leq 1,n}))$, which is in turn reflected as relations in $H_\ast(\overline{\op{M}}_{1,n})$, so this amounts (after \cite{PetG1})  to reinterpreting Getzler's elliptic relation \cite{GetGW} as a higher homology operation on $H_\ast(\op{M}_{g\leq 1,n})$ landing in $H_3(\op{M}_{1,4})$.  With these higher operations we could then view $\FT (\mathfrak{p}H_\ast(\op{M}_{g,n}))$ as a model for semi-classical homotopy cohomological field theories, compare \cite{DSVV}.

{\bf IV.}  A related problem would be to study the weak modular operad $H_\ast(\widehat{\op{M}}_{g,n})$ where $\widehat{\op{M}}_{g,n}$ is the topological modular operad of moduli spaces of Riemann surfaces with parameterized boundary components.  In parallel with the genus $0$ story (cf \cite{DCV,Wardsdp}), I expect the weak Feynman transform  $\FT (H_\ast(\widehat{\op{M}}_{g,n}))$ to be related to an extension of $H_\ast(\op{M}_{g,n})$ to an ``unstable weak modular operad'' which includes the unstable term $H_\ast(\op{M}_{0,2}):= H_\ast(BS^1)$.  
	These weak modular operads (and their associated weak wheeled operads) could then be employed to develop the homotopy theory of higher genus analogs of Batalin-Vilkovisky algebras; of interest from a number or perspectives \cite{GCTV},\cite{DSV13}.

{\bf V.}   Subsequent approaches to modular operads up-to-homotopy have recently appeared, see \cite{HRY} and \cite{MB2}, and it would be desirable to work out precise relationships between these different notions.   It should be possible to directly compare Definition $\ref{wmodef}$ to the strongly homotopy modular operads defined in \cite[Theorem 26]{MB2} using the language of operadic categories.  The precise relationship with \cite{HRY} seems more subtle and we refer to page 3 of loc.cit.\ for a discussion of this question.  It should also be possible to adapt the methods here-in to study up-to-homotopy analogs of many other generalizations of  operads, see Remark $\ref{cubicalrmk}$ and Corollary $\ref{cycopscor}$.



\end{document}